\documentclass[a4paper,11pt, leqno]{amsart}
\usepackage[latin1]{inputenc}
\usepackage{hyperref}
\usepackage{tikz-cd}
\usetikzlibrary{decorations.pathmorphing}
\usepackage{articlemacro}
\usepackage{mathtools}
\usepackage{stmaryrd}
\usepackage{nicematrix}
\makeatletter
\let\@wraptoccontribs\wraptoccontribs
\makeatother
\newcommand{\B}[1]{B#1}
\DeclareMathOperator{\BK}{BK}\newcommand{\Bunline}{\underline{\mathcal{B}\textup{un}}}

\DeclareMathOperator{\Disp}{Disp}

\DeclareMathOperator{\Gr}{Gr}
\DeclareMathOperator{\Hck}{\textup{Hck}}
\newcommand{\Hckbar}[1]{#1\text{-}\overline{\Hck}}
\DeclareMathOperator{\Rees}{Rees}

\DeclareMathOperator{\type}{type}
\newcommand{\psz}{\dlbrack z \drbrack}
\newcommand{\lsz}{\dlparent z \drparent}
\newcommand{\ur}{\textup{ur}}
\newcommand{\tensor}[1]{\underset{#1}{\otimes}}
\DeclareMathOperator{\Zip}{Zip}

\begin{document}

\title{Moduli of truncated shtukas and displays}

\author{Eva Viehmann}
\address{Westf\"alische Universit\"at M\"unster\\ Fachbereich f\"ur Mathematik und Informatik FB10\\
Einsteinstra{\ss}e 62\\
48149 M\"unster, Germany}
\email{viehmann@uni-muenster.de}
\author{Torsten Wedhorn}
\address{Technische Universit\"at Darmstadt\\
Fachbereich Mathematik FB04\\
Schlossgartenstra{\ss}e 7 \\
64289 Darmstadt, Germany}
\email{wedhorn@mathematik.tu-darmstadt.de}
\contrib[with an appendix by]{Christopher Lang}
\address{Technische Universit\"at Darmstadt\\
Fachbereich Mathematik FB04\\
Schlossgartenstra{\ss}e 7 \\
64289 Darmstadt, Germany}
\email{clang@mathematik.tu-darmstadt.de}
%with an Appendix by Christopher Lang\\
%Technische Universit\"at Darmstadt,\\
%Fachbereich Mathematik FB04,\\
%Schlossgartenstra{\ss}e 7, \\
%64289 Darmstadt, Germany
%}
%\author[2]{Torsten Wedhorn}
%\affil[1]{Westf\"alische Universit\"at M\"unster, Fachbereich f\"ur Mathematik und Informatik FB10, Einsteinstra{\ss}e 62, 48149 M\"unster, Germany}
%\affil[2]{Technische Universit\"at Darmstadt, Fachbereich Mathematik FB04, Schlossgartenstra{\ss}e 7, 64289 Darmstadt, Germany} 

%\subtitle{}


\begin{abstract}
We study moduli spaces of truncated local shtukas and truncated displays and describe them as concrete quotient stacks. To do this, we develop a general formalism of frames that can be applied in both cases and is also used to study prismatic displays and prismatic $F$-gauges.
\end{abstract}

\maketitle

%------------------------------------------------------------------

%\noindent{\scshape Abstract.\ }

%==================================================================

%==================================================================

\section*{Introduction}

Often cohomology of geometric objects takes the form of a module together with an additional structure. As an example one may think of de Rham cohomology equipped with its Hodge filtration. And sometimes the central objects of study are even themselves ``modules with additional structures'' such as shtukas. In this paper we start a systematic study of moduli spaces of such ``modules with additional structures''. We focus on local shtukas and on displays, even though the general formalism that we develop can also applied to prismatic displays, prismatic Breuil--Kisin modules, or prismatic $F$-gauges as explained in the last section. Here we use two main techniques: Geometrization and truncation.

\subsubsection*{Geometrization}

The idea of geometrization has been used very fruitfully in many areas, in particular recently with $p$-adic cohomology theories. Let us focus here on one aspect: Instead of studying modules over a commutative ring $R$ with some additional structure $\Sigma$ we look for a geometric object $R^{\Sigma}$ such that modules over $R^{\Sigma}$ are the same as modules over $R$ with additional structure $\Sigma$.

A well-known example are modules $M$ over $R$ with descending $\ZZ$-filtrations $\Fil^{\bullet}M$ (where we leave the precise definition of descending $\ZZ$-filtration vague for the moment). The corresponding geometrization is defined as follows. Let the multiplicative group $\GG_m$ act on the affine line $\AA^1$ by $(t,x) \sends t^{-1}x$ and form the corresponding quotient stack $R^{\Fil} := [\GG_{m,R}\backslash \AA^1_R]$. Then it is well-known (see also Proposition~\ref{CharFilteredVB} below) that a vector bundle over $R^{\Fil}$ is the same as a vector bundle over $R$ with a descending $\ZZ$-filtration by subbundles. We have an analogous result in the derived category: One can identify $\Dcal_{\textup{qc}}(R^{\Fil})$ with the filtered derived category of $R$ \cite{Moulinos_Filtrations}.

This point of view has in our opinion at least three advantages.
\begin{assertionlist}
\item
Algebraic constructions can be studied via geometry: For instance, if a filtered vector bundle over $R$ is viewed as a vector bundle over $R^{\Fil}$, its graded vector bundle can be obtained by pullback via the closed immersion $\B{\GG_{m,R}} = [\GG_{m,R}\backslash \{0\}] \to R^{\Fil}$.
\item
There is a very natural way to endow these geometrized algebraic structures with additional symmetries by a group scheme $G$: In the above example a filtered module over $R$ with $G$-structure can be simply defined as a $G$-bundle over $R^{\Fil}$. Compared to the classical way via Tannakian formalism this works also for group schemes $G$ for which one might not be able to identify $G$-bundles over $R$ with exact monoidal functors from the category of algebraic representations of $G$ on finite projective $R$-modules to the category of vector bundles (see \cite[A.28, A.30]{Wedhorn_ExtendBundles}).
\item
One may obtain via geometrization more natural definitions in those cases where the classical definitions do not yield satisfying results: The category of filtered modules defined as usual is not abelian \cite[II \S5.17]{GelfandManin_MethodsHomological}, but the category of all quasi-coherent modules over $R^{\Fil}$ certainly is. So we suggest to redefine a filtered module as a quasi-coherent module over $R^{\Fil}$ (see Definition~\ref{DefFilteredModuleRing}). This definition will have several advantages later in the paper.
\end{assertionlist} 

\medskip

\noindent
Let us illustrate the idea of geometrization via another example, which will be central in this paper, namely local shtukas (with one leg). Let $\kappa$ be a finite field with $q$ elements. Let $G$ be a reductive group over $\kappa$. Then a local $G$-shtuka over a $\kappa$-algebra $R$ is by definition a $G$-bundle $\Escr$ over $R\psz$ together with an isomorphism $\phi^*(\Escr[1/z]) \iso \Escr[1/z]$, where $\phi$ is the partial Frobenius
\begin{equation}\label{EqIntroPhiShtuka}
\phi(\sum_{n=0}^{\infty}a_nz^n) = \sum_{n=0}^{\infty}a_n^qz^n
\end{equation}
on $R\psz$. This structure is now geometrized as follows. Let $R^{\Hck}$ be the scheme obtained by gluing two copies of $\Spec R\psz$ along the open subscheme $\Spec R\lsz$. Then a $G$-bundle on $R^{\Hck}$ is the same as a Hecke triple $(\Escr_1,\Escr_2,\alpha)$, where $\Escr_1$ and $\Escr_2$ are $G$-bundles and $\alpha\colon \Escr_1[1/z] \iso \Escr_2[1/z]$ is an isomorphism.

To obtain a geometrization of $G$-shtukas, we denote by $\tau\colon \Spec R\psz \to R^{\Hck}$ the inclusion onto one of the copies of $\Spec R\psz$ and by $\sigma$ the inclusion onto the other copy precomposed with $\phi$. Given a $G$-bundle on $R^{\Hck}$ viewed as a triple $(\Escr_1,\Escr_2,\alpha)$, its pullback via $\tau$ (resp.~via $\sigma$) is $\Escr_1$ (resp.~$\phi^*\Escr_2$). Hence we define
\[
R^{\Sht} := \colim (\xymatrix{\Spec R\psz
\ar@<.5ex>[r]^-(.4){\tau}\ar@<-.5ex>[r]_-(.4){\sigma} & R^{\Hck}}),
\]
where we take the coequalizer in some $\infty$-category containing the classifying stacks $\B{G}$ of reductive groups $G$. Then $G$-bundles on $R^{\Sht}$ can be identified with local $G$-shtukas over $R$.

\subsubsection*{Truncation}

The second main technique is to approximate the objects that we are interested in (in this paper, mainly local shtukas and displays) by their truncations. For local shtukas and the corresponding Hecke triples this takes the following shape: The naive mod-$z^N$-reduction of a Hecke triple or a local shtuka loses to much information, we would essentially end up with a (pair of) $G$-bundles over $R\psz/(z^N)$. Instead, we pass to a ``completion'' (in the sense of compactification) of the stacks $R^{\Hck}$ and $R^{\Sht}$. Let us explain this for $R^{\Hck}$, since then one obtains the completion of $R^{\Sht}$ again by taking co-equalizers.

If $R$ is a field, then $R^{\Hck}$ is obtained from gluing two copies of a complete discrete valuation ring along their generic point. It is a regular one-dimensional scheme over $R\psz$. Now we aim to enrich its special fiber without changing the groupoid of $G$-bundles over it. For this we add a point of codimension $2$. Of course this makes only sense in the world of algebraic stacks in which negative dimensions exist. More precisely, we define for arbitrary $R$ the quotient stack
\[
R^{\Hck,\infty} = [\GG_{m}\backslash \Spec R\psz[t,u]/(tu-z)],\tag{*}
\]
where $\GG_m$ acts on $t$ (resp.~$u$) with degree $-1$ (resp.~$1$). If $R$ is a field, such algebraic stacks have also been studied in detail in \cite{AHLH_ExistenceModuli}.

There is an open embedding $R^{\Hck} \to R^{\Hck,\infty}$ identifying $R^{\Hck}$ with the locus where $u$ or $t$ is invertible. Pullback via this open immersion defines a functor from the groupoid $\Bun_G(R^{\Hck,\infty})$ of $G$-bundles on $R^{\Hck,\infty}$ to $\Bun_G(R^{\Hck})$. General results on extension of $G$-bundles (e.g.,\cite{Wedhorn_ExtendBundles}) show that this functor is fully faithful and even an equivalence if $R$ is a field (Proposition~\ref{BunGHckbar}). For general $R$ one can also describe the essential image of the functor (Proposition~\ref{VBEssentialImage} for vector bundles, Corollary~\ref{RestrictGBunReAvne0} for general $G$-bundles).

Now we can also define truncated Hecke stacks $R^{\Hck,N}$ for $N \geq 1$ by replacing $R\psz$ with $R\psz/(z^N)$ in (*). Via a very similar coequalizer construction we also obtain truncated shtuka stacks $R^{\Sht,N}$ (Definition~\ref{DefTruncatedShtukaStack}). By definition, an $N$-truncated local $G$-shtuka over $R$ is then a $G$-bundle on $R^{\Sht,N}$. In Theorem~\ref{ColimitNTruncShtuka} we see
\[
\Bun_G(R^{\Sht,\infty}) = \lim_{N<\infty} \Bun_G(R^{\Sht,N}).
\]
Hence we may approximate local $G$-shtukas by truncated local $G$-shtukas. In fact, we deduce this result from the fact that one has
\[
\colim_{N<\infty}R^{\Sht,N} = R^{\Sht,\infty},
\]
where we form the colimit in the 2-category of Adams stacks, whose theory we recall in Appendix~\ref{APPADAMS}.

We also show that the moduli space of truncated local $G$-shtukas decomposes as a direct sum of open and closed connected components parametrized by conjugacy classes $[\mu]$ of cocharacters of $G$. For each connected component we give a concrete description as a quotient stack (Theorem~\ref{DescribeModSpaceTruncShtukas}). This is an application of the general classification result explained below (Theorem~\ref{IntroClassThm}).

Finally we relate 1-truncated local $G$-shtukas to $G$-zips as defined in \cite{PWZ2} again by geometrization: Let $R^{\Zip}$ be the algebraic stack defined by Yaylali \cite{Yaylali_DerivedFZips} such that $G$-bundles over $R^{\Zip}$ can be identified with $G$-zips (Definition~\ref{DefFZipStack}). Then there are universal homeomorphisms $R^{\Sht,1} \to R^{\Zip} \to R^{\Sht,1}$ (Proposition~\ref{CompareZip}).

\subsection*{General results}

Most of the above results on (truncated) local shtukas exist also for (truncated) displays. In addition, it is also possible to study prismatic displays, prismatic Breuil--Kisin modules, or prismatic $F$-gauges with these techniques with analogous arguments. Therefore we develop a general formalism such that all of the above cases are special cases of our formalism.

\subsubsection*{Quotient stacks by $\GG_m$}

Since all of our geometrizations are algebraic quotient stacks by $\GG_m$ we start by studying these stacks, their fixed point locus, attractor locus, and repeller locus (Section~\ref{Sec:Attractor}). Then we recall some facts about quasi-coherent modules, modules of finite presentation and vector bundles on such stacks (Section~\ref{QCOHGGmSTACK}).

Finally, we consider colimits of such stacks. To keep it a purely algebraic theory and to avoid using GAGA theorems comparing $G$-bundles over formal schemes and over algebraic schemes we here work in the 2-category of Adams stacks developed by Sch\"appi and recalled in Appendix~\ref{APPADAMS}. The results of Sch\"appi also give us a tool at hand to identify colimits of such stacks (Theorem~\ref{AdamsStackCocomplete}). We use this to describe certain colimits of quotient stacks by $\GG_m$ (Proposition~\ref{ColimitQuotientStack}).

\subsubsection*{Filtered rings and their Rees stacks}

We then develop a general formalism of filtered modules and filtered algebras. According to our principle of geometrization, a filtered module over a ring $A$ is defined as a quasi-coherent module over the quotient stack $A^{\Fil}$ (Definition~\ref{DefFilteredModuleRing}). We describe the attached derived graded module. We describe vector bundles over $A^{\Fil}$ and show that they are simply vector bundles with a filtration by subbundles (Proposition~\ref{CharFilteredVB}).

Next we study filtered rings, where we only consider filtrations in non-negative degrees. Thus the natural definition according to our principles is the following.

\begin{intro-definition}\label{IntroDefFilteredRing}
A (non-negatively graded) filtered ring with underlying ring $A$ is a quasi-coherent algebra over $A^{\Fil}$ whose underlying $A$-algebra is $A$ and such that its derived graded module is concentrated in non-negative degrees.
\end{intro-definition}
% (Definition~\ref{DefFilteredRing}). 
This can also be made more explicit (Section~\ref{Subsec:FilteredRing}) as a sequence of (not necessarily injective maps) maps of $A$-modules
\[
\cdots \ltoover{t} \Fil^2 \ltoover{t} \Fil^1 \ltoover{t} \Fil^0 = A
\]
endowed with a multiplication and satisfying certain compatibilities.

Every quasi-coherent algebra $\Bscr$ over $A^{\Fil}$, given by a $\ZZ$-graded $A$-algebra $B$, corresponds to an affine morphism $\Re(B) = \Spec \Bscr \to A^{\Fil}$, and $\Re(B)$ is called the \emph{Rees stack}. We describe the fixed point, the attractor, and the repeller locus of the Rees stacks (Remark~\ref{ReesAttractor}). For a filtered ring $(A,(\Fil^j))$, the complement of the attractor locus can be identified with $\Spec A$ and one obtains an open immersion
\[
\tau\colon \Spec A \lto \Re(A, (\Fil^j)).
\]

\subsubsection*{Frames}

We develop a general formalism of frames that generalizes the notion of higher frames defined by Lau \cite{Lau_HigherFrames}.

\begin{intro-definition}(Definition~\ref{DefFrame}) \label{IntroDefFrame}
A \emph{frame} is a filtered ring $(A,(\Fil^j)_j)$ together with a map of algebraic stacks
\[
\sigma\colon \Spec A \lto \Re(A, (\Fil^j))
\]
which factors through the complement of the repeller locus.

The composition of $\sigma$ with the structure morphism $\Re(A, (\Fil^j)) \to \Spec A$ corresponds then to a ring endomorphism of $A$ which is denoted by $\phi$.
\end{intro-definition}

As in the special case of local shtukas we thus obtain the attached \emph{frame stack}
\[
\Fcal(A) := \colim (\xymatrix{\Spec A
\ar@<.5ex>[r]^-(.4){\tau}\ar@<-.5ex>[r]_-(.4){\sigma} & \Re(A, (\Fil^j))})
\]
These constructions indeed generalize the local shtuka case explained above:

\begin{intro-example}\label{IntroExampleSht}
Let $\kappa$ be a finite field with $q$ elements, let $R$ be a $\kappa$-algebra, and for $1 \leq N \leq \infty$ set $R_N := R\psz/(z^N)$. Consider the filtered ring
\[
\cdot \ltoover{\cdot z} R_N \ltoover{\cdot z} R_N \ltoover{\cdot z} R_N
\]
defining the Rees stack $\Re(R_N,z)$. Then the complement $\Re(R_N,z)^{\ne -}$ of the repeller locus is canonically isomorphic to $\Spec R_N$ and we define $\sigma$ as the composition
\[
\sigma\colon \Spec R_N \ltoover{\phi} \Spec R_N \cong \Re(R_N,z)^{\ne -} \lto \Re(R_N,z),
\]
where $\phi$ is given by \eqref{EqIntroPhiShtuka}. Then $(R_N,z,\sigma)$ is a frame and we obtain the associated frame stack $\Fcal(R_N,z)$. There are functorial isomorphisms
\[
\Re(R_N,z) \cong R^{\Hck,N}, \qquad \Fcal(R_N,z) \cong R^{\Sht,N}.
\]
Moreover, for $N = \infty$, the complement of the fixed locus in $R^{\Hck,\infty}$ is $R^{\Hck}$, and the corresponding substack of $R^{\Sht,\infty}$ is $R^{\Sht}$.
\end{intro-example}

This example is a special case of frames attached to filtered rings given by $A$-linear maps $v\colon L \to A$, where $L$ is an invertible module over a ring $A$. This is also of interest in the prismatic context. We consider it in detail in Section~\ref{Sec:ReesVAdic}.

\subsubsection*{Moduli spaces of $G$-bundles on frame stacks}

Our main object of study are moduli spaces of $G$-bundles over frame stacks for reductive groups $G$ (in fact, many of our results hold for more general group schemes). For this we study frames in families. More precisely, we consider a functor $A$ that sends an object $R$ in a base category (below this will be the category of algebras over some fixed ring) to a pair consisting of a frame $(A_R, (\Fil_{A_R}^j), \sigma_R)$ and an isomorphism $A_R/t(\Fil^1_{A_R}) \liso R$. We obtain for any stack $\Xcal$ the attached ``loop prestack''  $L^A\Xcal$ defined by $R \sends \Xcal(A_R)$. In the local shtuka case we have $A_R = R\psz$ and in this case, $L^A\Xcal$ is the positive loop stack of $\Xcal$. We assume that the following conditions are satisfied.
\begin{definitionlist}
\item\label{IntroAssumptiona}
$R \sends \Fil^j_{A_R}$ is representable by an affine scheme for all $j$.
\item\label{IntroAssumptionb}
$(A_R,t(\Fil^1_{A_{R}}))$ is a henselian pair.
\item\label{IntroAssumptionc}
$\B{L^AG} \to L^A\B{G}$ is an isomorphism.
\end{definitionlist}
They are satisfied in the explicit contexts discussed below and at the same time allow us to first abstractly discuss general classification results.
For every reductive group $G$ we obtain the prestacks of $G$-bundles
\begin{align*}
\Bunline_G(A)\colon R &\sends \Bun_G\left(\Re(A_R,(\Fil^j_{A_R}))\right), \\
\Bunline_G(\Fcal(A))\colon R &\sends \Bun_G\left(\Fcal(A_R,(\Fil^j_{A_R})),\sigma_R\right).
\end{align*}
The above hypotheses then imply that these are stacks for the fpqc topology. We show that both of these stacks admit decompositions into open and closed substacks
\begin{align*}
\Bunline_G(A) &= \coprod_{[\mu]}\Bunline^{[\mu]}_G(A), \\
\Bunline_G(\Fcal(A)) &= \coprod_{[\mu]}\Bunline^{[\mu]}_G(\Fcal(A)),
\end{align*}
where $[\mu]$ runs through the conjugacy classes of cocharacters of $G$. In this introduction only we ignore the notational problem that there might exist conjugacy classes that are not defined over the same ring over which $G$ is defined. We attach to the functor $R \sends (A_R,(\Fil_{A_R}^j))$ and to a pair $(G,[\mu])$ a functor of groups $E_A(G,[\mu])$ \eqref{EqDefEXcalMu} which is called the display group and which can be shown to be an affine group scheme (not necessarily of finite type) using Assumption~\ref{IntroAssumptiona}. Then the next result is our main general classification result.

\begin{intro-theorem}(Theorem~\ref{GeneralBunlineGMu}, Theorem~\ref{DescribeBunlineFramStack})\label{IntroClassThm}
One has equivalences of stacks
\begin{align*}
\Bunline^{[\mu]}_G(A) &\cong \B{E_A(G,[\mu])}, \\
\Bunline^{[\mu]}_G(\Fcal(A)) &\cong [E_A(G,\mu)\backslash L^AG],
\end{align*}
where the action of the affine group scheme $E_A(G,\mu)$ on $L^AG$ is induced by $\sigma$ and $\tau$ \eqref{EqGenActionEonG}.
\end{intro-theorem}

The proof of the first equivalence uses the Assumption~\ref{IntroAssumptionb} and results in \cite{Wedhorn_ExtendBundles}, recalled in the form used here in Corollary~\ref{LiftingGBundlesReesStack}. The second equivalence follows then quite formally using Assumption~\ref{IntroAssumptionc}.

\subsubsection*{Rees stacks and frames given by the $v$-adic filtration}

The underlying filtered rings for (truncated) shtukas, for (truncated) displays over perfect rings, and for prismatic displays and prismatic $F$-gauges are all given by the so-called $v$-adic filtration, where $A$ is a ring and $v\colon L \to A$ is an $A$-linear map from a line bundle $L$ over $A$. Any ring endomorphism $\phi$ of $A$ induces a frame structure. The attached frame stack is denoted by $\Fcal(A,v,\phi)$. In Section~\ref{Sec:ReesVAdic} we discuss this example in detail. We describe vector bundles on the attached Rees stack $\Re(A,v)$  (Proposition~\ref{DescribeVBReesStack}). Moreover, we obtain the following result.

\begin{intro-theorem}(Corollary~\ref{GBundlesZAdicFRameStack})\label{IntroResNe0}
Let $v$ be injective and let $\Fcal(A,v,\phi)^{\ne0}$ be the open complement of the fixed point locus. Then for every smooth affine group scheme $G$ over $A$, pullback via the inclusion yields a fully faithful functor
\begin{equation}\label{IntroEqResNe0}
\Bun_G(\Fcal(A,v,\phi)) \lto \Bun_G(\Fcal(A,v,\phi)^{\ne0}).
\end{equation}
We can describe its essential image. Moreover \eqref{IntroEqResNe0} is an equivalence if $G$ is reductive, $A$ is a discrete valuation ring, and $v$ is the inclusion of the maximal ideal of $A$.
\end{intro-theorem}

In the local shtuka case, \eqref{IntroEqResNe0} describes the restriction of $\infty$-truncated local $G$-shtukas to local $G$-shtukas. In the prismatic case, \eqref{IntroEqResNe0} describes the functor from prismatic $G$-displays to Breuil--Kisin $G$-modules (Section~\ref{Sec:PrismaticDisplays}).

%---------------------------------------------------------------------

\subsection*{Further applications}

Beyond the case of local shtukas we also apply our general formalism to displays and to prisms.

\subsubsection*{Moduli spaces of truncated $G$-displays}

In the case of displays we rephrase, extend, and geometrize results by Lau and Daniels. The general formalism developed in the first four sections allows us to define for every $p$-adically complete ring $R$ stacks $R^{\Disp}$ and truncations $R^{\Disp,N,n}$ for integers $N > n \geq 1$. If $R$ is a ring of characteristic $p$, then we can also define truncations for $N = n$, and the stack is then denoted by $R^{\Disp,N}$. A ($(N,n)$-truncated) $G$-display over $R$ is then by definition a $G$-bundle over $R^{\Disp}$ (over $R^{\Disp,N,n}$). In an appendix, Christopher Lang proves that $\GL_h$-displays of type $[\mu_d]$, where $\mu_d$ is the cocharacter
\begin{equation}\label{IntroEqStandardMu}
t \sends \twomatrix{tI_d}{0}{0}{I_{h-d}}
\end{equation}
are indeed the same as $(\GL_h,\mu_d)$-displays as defined by B\"ultel and Pappas, justifying our definition.

Then Theorem~\ref{IntroClassThm} allows us to describe the moduli space of $G$-displays (resp.~of $(N,n)$-truncated $G$-displays) as a quotient stack of the positive Witt loop group by the action of a display group (Theorem~\ref{DescribeGDisplays}) recovering results of Lau in \cite{Lau_HigherFrames}. This description also implies that the moduli space of $G$-displays is the limit of the moduli space of truncated $G$-displays (Corollary~\ref{LimitTruncatedDisplays}).

In this case, the following proposition implies that $1$-truncated $G$-displays are in fact the same as $G$-zips.

\begin{intro-proposition}\label{IntroDisplayZip}
Let $R$ be a ring of characteristic $p$. Then there is a canonical isomorphism of algebraic stacks $R^{\Zip} \iso R^{\Disp,1}$.
\end{intro-proposition}

\subsubsection*{Existence of Traverso bounds}

Both in the shtuka case and in the display case we can now approximate local $G$-shtukas (resp.~$G$-displays) by their truncations. For displays without additional structure choose $G = \GL_h$ and $\mu = \mu_d$ as in \eqref{IntroEqStandardMu}. Then Dieudonn\'e theory yields an identification between ($N$-truncated) $p$-divisible groups of height $h$ and dimension $d$ over an algebraically closed field $k$ of characteristic $p > 0$ and ($N$-truncated) $\GL_h$-displays of type $\mu_d$ over $k$. In this case, it is a classical result that the isomorphism class of such a $p$-divisible group depends only on its $N_0$-truncation for some $N_0$ depending only on $h$ and $d$. These bounds have been studied by Traverso \cite{Traverso_SullaClass} and by Lau, Nicole, and Vasiu \cite{LNV_Traverso}. Here we vastly generalize this result using our descriptions of the moduli spaces to arbitary reductive groups $G$, arbitrary conjugacy classes $[\mu]$, and both for local shtukas and for displays:

\begin{intro-theorem}(Theorem~\ref{ExistTraverso}, Theorem~\ref{ExistTraversoDisplay})\label{IntroTraverso}
Let $G$ be a reductive group and let $[\mu]$ be a conjugacy class of cocharacters of $G$. Let $k$ be an algebraically closed field of characteristic $p > 0$. Then there exists an $N_0 \geq 1$, dependent only on $G$ and $[\mu]$, such that any two $G$-shtukas of type $[\mu]$ (or any two $G$-displays of type $[\mu]$) over $k$ are isomorphic if and only if their $N_0$-truncations are isomorphic.
\end{intro-theorem}

\subsubsection*{Prismatic displays and prismatic $F$-gauges}

In the last section we apply the general results to filtered rings and frames obtained from a bounded prism $\Aline = (A,\delta,I)$. The first case is to consider the $I$-adic filtration. This is a special case of Section~\ref{Sec:ReesVAdic}. The corresponding frame stack is denoted by $\Aline^{\Disp}$. If $\Aline$ is the prism attached to a perfect ring $R$ in characteristic $p$, we have $\Aline^{\Disp} = R^{\Disp}$, justifying the notation. An immediate application of Theorem~\ref{IntroResNe0} yields a fully faithful functor from the groupoid of $G$-displays over $\Aline$ and Breuil-Kisin $G$-bundles over $\Aline$ and determines its essential image (Proposition~\ref{BKDisplay}). This reproves (and slightly generalizes) a result of Ito \cite[5.3.8]{Ito_PrismaticGDisplaysDescent}.

The second filtration that we consider is the Nygaard filtration and the attached Nygaard frame. The corresponding frame stack is called the syntomic frame stack of $\Aline$ and is denoted by $\Aline^{\Syn}$. This notation  is justified by the fact, proven by Bhatt, that if $\Aline$ is the initial object in the prismatic site of a quasi-regular semi-perfectoid ring $R$, then $\Aline^{\Syn} \cong R^{\Syn}$, where $R^{\Syn}$ denotes (an algebraic version of) the syntomification of $R$ defined by Bhatt and Lurie \cite{BhattLurie_AbsolutePrismatic}. Then a prismatic $G$-gauge over $\Aline$ is by definition a $G$-bundle on $\Aline^{\Syn}$. We construct a map of stacks $\Aline^{\Disp} \to \Aline^{\Syn}$ that yields by pullback a functor attaching prismatic $G$-displays to prismatic $G$-gauges which may be seen as the extension of a result of Ito (Remark~\ref{NygaardPrismaticFrame}). 

%---------------------------------------------------------------------

\bigskip

\noindent{\scshape Acknowledgements.\ }
We thank Christopher Lang, Timo Richarz, and Can Yaylali for helpful remarks. This project was supported by the Deutsche Forschungsgemeinschaft (DFG, German Research Foundation) via TRR 326 and the Cluster of Excellence Mathematics M\"unster.

%---------------------------------------------------------------------

\tableofcontents

\bigskip\bigskip
%---------------------------------------------------------------------

\subsection*{Notation}

We always denote by $p$ a prime.

Although we will not need the full machinery of $\infty$-categories, it is convenient to consider ordinary categories or $(2,1)$-categories as special cases of $\infty$-categories via their nerves. Similarly, we consider ordinary groupoids as special cases of $\infty$-groupoids a.k.a. anima. For instance if we have a site $\Ccal$ (in the classical sense), a functor on $\Ccal$ with values in groupoids or in $(2,1)$-categories is a sheaf if it is a sheaf when considered with values in the $\infty$-category of anima or of $\infty$-categories. This puts us firmly in the frame work of \cite{Lurie_HTT}. An equivalent approach without using the language of $\infty$-categories can be found in \cite{Stacks}.

Let $S$ be a scheme. As usual we denote by $\Affrel{S}$ the full subcategory of the category of $S$-schemes consisting of scheme morphisms $X \to S$ with $X$ an affine scheme. A (classical) prestack over $S$ is a functor of $(2,1)$-categories from the category $\Affrel{S}$ to the $(2,1)$-category of groupoids.

A stack over $S$ is a prestack over $S$ that satisfies descent for the fppf-topology. We call a stack $\Xcal$ \emph{algebraic} if its diagonal is representable by algebraic spaces and if there exists a surjective and smooth morphism $U \to \Xcal$, where $U$ is a scheme.
%\footnote{This is slightly more restrictive than the definition in the stacks project since there stacks are only defined to be sheaves for the fppf-topology. Every algebraic stack in the sense of the stacks project that has a quasi-affine diagonal is an algebraic stack in our sense by \cite[10.7]{LMB_ChampAlg} or \cite[3.5]{BhattHL_Tannakian}.}.
We denote by $|\Xcal|$ the underlying topological space of $\Xcal$.

Let $\Xcal$ and $\Ycal$ be stacks over a stack $\Scal$. We denote by $\Hom_{\Scal}(\Xcal,\Ycal)$ the groupoid of morphisms $\Xcal \to \Ycal$ of stacks over $S$. The stack $T \sends \Hom_S(\Xcal \times_S T,\Ycal)$ is denoted by $\Homline_S(\Xcal,\Ycal)$. More generally, let $\Xcal\colon T \sends \Xcal_T$ be a functor from the category $\Affrel{S}$ to the $(2,1)$-category of stacks over $S$. Then we denote by $\Homline_S(\Xcal,\Ycal)$ the prestack $T \sends \Hom_S(\Xcal_T,\Ycal)$. This is a stack, if $T \sends \Xcal_T$ is a sheaf for the fppf topology.

If $\Xcal$ is a pre-stack, then we have the notion of an $\Oscr_{\Xcal}$-module, see \cite[06WB]{Stacks}. Via the tensor product it is a symmetric monoidal category. A vector bundle on $\Xcal$ is a finite locally free $\Oscr_{\Xcal}$-module in the sense of \cite[03DE]{Stacks}. We denote the exact monoidal category of vector bundles on $\Xcal$ by $\Vec{\Xcal}$.

Let $S$ be an algebraic space and let $G$ be a group algebraic space over $S$ such that $G \to S$ is an fpqc covering (see \cite[App.~A]{Wedhorn_ExtendBundles}), e.g., if $G \to S$ is quasi-compact and flat or if $G \to S$ is flat and locally of finite presentation. Then for every stack $\Xcal$ over $S$ we denote by $\Bun_G(\Xcal)$ the groupoid of morphism $\Xcal \to \B{G}$ over $S$, where $\B{G}$ is the classifying stack of $G$. The pointed set of isomorphism classes of objects in $\Bun_G(\Xcal)$ is denoted by $H^1(\Xcal,G)$. We refer to \cite[Appendix A]{Wedhorn_ExtendBundles} for more details on $G$-bundles on stacks.

Let $\phi\colon A \to B$ be a map of rings. For any $A$-module $M$ we set $\phi^*(M) = B \otimes_A M$. If $N$ is a $B$-module and $u\colon M \to N$ is a $\phi$-linear map (i.e., an additive map satisfying $u(am) = \phi(a)u(m)$ for $a \in A$ and $m \in M$), then we denote the corresponding $B$-linear map $\phi^*(M) \to N$, $b \otimes m \sends bu(m)$ by $u^{\#}$.

%%% Local Variables: 
%%% mode: latex
%%% TeX-master: "CycleClassesEO"
%%% End: 

\section{Quotient stacks by $\GG_m$}

\subsection{Describing quotients by $\GG_m$ in the affine case}\label{Sec:DescribeGGmQuotients}

Let $A$ be a ring. Recall that the functor $B \sends \Spec B$ induces a contravariant equivalence of the category of $\ZZ$-graded $A$-algebras and the category of affine $A$-schemes with a $\GG_{m,A}$-action \cite[Exp.~I, 4.7.3.1]{SGA3I}. The $\GG_m$-equivariant morphism $\Spec B \to \Spec A$, where $\Spec A$ is endowed with the trivial $\GG_m$-action, induces an affine morphism of algebraic stacks
\[
\Xcal := [\GG_{m,A}\backslash(\Spec B)] \to \B{\GG_{m,A}}
\]
sitting in a Cartesian diagram
\[\xymatrix{
\Spec B \ar[r] \ar[d] & \Xcal \ar[d] \\
\Spec A \ar[r] & \B{\GG_{m,A}}.
}\]
Conversely, every affine morphism $\Xcal \to \B{\GG_{m,A}}$ yields by base change an affine scheme $\Spec B$ over $\Spec A$ with $\GG_m$-action, i.e., a $\ZZ$-graded $A$-algebra.

Let us describe the points of $\Xcal$:

\begin{remark}\label{MapsToQuotientbyGGm}
If $C$ is any $A$-algebra, the groupoid of maps $f\colon \Spec C \to \Xcal$ over $\Spec A$ is the groupoid of pairs $(L,\alpha)$, where
\begin{definitionlist}
\item
$L$ is a projective $C$-module of rank $1$ (a line bundle on $C$) corresponding to the $\GG_m$-bundle $\Spec T(L)$ with $T(L) := \bigoplus_{n\in \ZZ}L^{\otimes n}$ over $C$ and
\item
$\alpha\colon B \to T(L)$ is a map of graded $A$-algebras or, equivalently, a map of graded $C$-algebras $C \otimes_A B \to T(L)$.
\end{definitionlist}
Then $\Spec T(L)$ is the underlying $C$-scheme of the $\GG_m$-bundle corresponding to the line bundle $L$ and one has a cartesian diagram
\begin{equation}\label{EqMapToGGmSpace}
\begin{aligned}\xymatrix{
\Spec T(L) \ar[r]^{\Spec f^*} \ar[d] & \Spec B \ar[d] \\
\Spec C \ar[r]^-f & [\Spec B/\GG_m].
}\end{aligned}
\end{equation}
\end{remark}

\begin{example}\label{MapsToAA1GGm}
Let $A$ be a ring, let $d \in \ZZ$ be an integer, and let $A[t]$ be the polynomial ring which we endow with a $\ZZ$-grading by setting $\deg(t) := d$. We denote the corresponding $\GG_{m,A}$-scheme by $(\AA^1_A)^{(d)}$. Let $C$ be an $A$-algebra. Then the groupoid $[\GG_{m,A}\backslash(\AA^1_A)^{(d)}](C)$ is the groupoid of pairs $(M,v)$, where $M$ is a line bundle on $C$ and $v\colon C \to M^{\otimes d}$ is a $C$-linear map (equivalently, an element in $M^{\otimes d}$). For $d = -1$, we usually consider $v$ has a $C$-linear map $M \to C$.

More generally, let $L$ be a line bundle on $A$ and consider $\Sym_A L = \bigoplus_{j\geq 0}L^{\otimes i}$ which we endow with a $\ZZ$-grading by letting $L$ sit in degree $d$ to get a $\GG_m$-scheme $(\Sym_A L)^{(d)}$. Then for every $A$-algebra $C$ the groupoid $[(\Sym_AL)^{(d)}/\GG_{m,A}](C)$ is the groupoid of pairs $(M,v)$, where $M$ is a line bundle on $C$ and $v\colon L \otimes_A C \to M^{\otimes d}$ is a $C$-linear map.
\end{example}

\begin{remark}\label{AffineMapstoGGmQuotients}
Let $A$ be a ring, let $B$ and $C$ be $\ZZ$-graded $A$-algebras, and set $\Xcal := [\GG_{m,A}\backslash(\Spec B)]$ and $\Ycal := [\GG_{m,A}\backslash \Spec C]$. One has equivalences of groupoids
\begin{align*}
\Hom_{\B{\GG_{m,A}}}(\Xcal,\Ycal) &\cong \Hom^{\GG_{m,A}}_A(\Spec B, \Spec C) \\
&\cong \{\text{$\ZZ$-graded maps $C \to B$ of $A$-algebras}\},
\end{align*}
where $\Hom^{\GG_{m,A}}_A(-, -)$ denotes the set of $\GG_m$-equivariant morphisms of $A$-schemes. This shows in particular that the groupoid $\Hom_{\B{\GG_{m,A}}}(\Xcal,\Ycal)$ is a set.
%
%affine morphisms $\Ycal \to \Xcal$ over $\Spec A$ correspond to morphisms $B \to C$ of $\ZZ$-graded $A$-algebras:
%\begin{assertionlist}
%\item
%If $\varphi\colon B \to C$ is a map of $\ZZ$-graded $A$-algebras, then the corresponding map $\Spec C \to \Spec B$ induces a map of stacks $g\colon \Ycal := [\GG_{m,A}\backslash(\Spec C)] \to \Xcal$ sitting in cartesian diagram
%\[\xymatrix{
%\Spec C \ar[r] \ar[d] & \Spec B \ar[d] \\
%\Ycal \ar[r]^-g & \Xcal
%}\]
%In particular, $g$ is affine.
%\item
%Conversely, let $g\colon \Ycal \to \Xcal$ be an affine morphism of stacks. Then $(\Spec B) \times_{\Xcal} \Ycal$ is affine over $\Spec B$ and hence of the form $\Spec C$ for a $B$-algebra $C$. The map $\Spec C \to \Ycal$ is base change of a $\GG_m$-torsor, hence it has the structure of $\GG_m$-torsor. This yields an action of $\GG_m$ on $\Spec C$, in other words, a $\ZZ$-grading on $C$, and an isomorphism $\Ycal \cong $. The morphism $\Spec C \to \Spec B$ is $\GG_m$-equivariant and hence corresponds to a map of $\ZZ$-graded $A$-algebras.
%\end{assertionlist}
%If $B \to C$ is a map of $\ZZ$-graded $A$-algebras, then the groupoid of affine maps $\Ycal := [\GG_{m,A}\backslash(\Spec C)] \to \Xcal$ can be identified by base change with the set of maps of $\ZZ$-graded $A$-algebras $B \to C$. Therefore $\Hom(\Ycal,\Xcal)$ is a set.
\end{remark}

%---------------------------------------------------------------------

\subsection{The attractor and the repeller locus}\label{Sec:Attractor}

We will consider the following schemes over $A$ with $\GG_{m,A}$-action.
\begin{assertionlist}
\item
$\PP^1_A$ with its canonical $\GG_m$-action.
\item
$\AA^1_A = (\AA^1_A)^+$ as the open $\GG_m$-invariant subscheme $\PP^1_A \setminus \{\infty\}$.
\item
$(\AA^1_A)^-$ as the open $\GG_m$-invariant subscheme $\PP^1_A \setminus \{0\}$.
\item
$\GG_{m,A}$ as the open $\GG_m$-invariant subscheme $\PP^1_A \setminus \{0,\infty\}$.
\item
$S = \Spec A$ with the trivial $\GG_{m,A}$-action.
\end{assertionlist}
Let $X$ be a scheme with $\GG_m$-action over $A$. Define sheaves for the fpqc topology with $\GG_m$-action on the category of $A$-algebras by
\begin{align*}
X^0 &:= \Homline_A^{\GG_m}(\Spec A,X), \\
X^+ &:= \Homline_A^{\GG_m}((\AA^1_A)^+,X), \\
X^- &:= \Homline_A^{\GG_m}((\AA^1_A)^-,X),
\end{align*}
called the \emph{space of fixed points}, \emph{the attractor}, and \emph{the repeller}, respectively. Note that $\Homline_A^{\GG_m}(\GG_m,X) = X$ by restricting a $\GG_m$-equivariant map $f\colon \GG_m \to X$ to its evaluation $f(1)$ in the unit section of $\GG_m$. This description shows that the automorphism $(-)^{-1}$ of $\GG_m$ induces the identity on $\Homline_A^{\GG_m}(\GG_m,X) = X$.

The zero section of $(\AA^1_A)^+$ and the $\infty$-section of $(\AA^1_A)^-$ are $\GG_m$-equivariant and hence yield by functoriality $\GG_m$-equivariant maps
\[
X^+ \lto X^0, \qquad X^- \lto X^0.
\]
Similarly, the inclusions $\GG_{m,A} \to (\AA^1_A)^+$ and $\GG_{m,A} \to (\AA^1_A)^-$ induce $\GG_m$-equivariant maps
\[
X^+ \lto X, \qquad X^- \lto X.
\]
If $X$ is separated, then $X^+ \to X$ and $X^- \to X$ is a monomorphism since $\GG_m$ is schematically dense in $(\AA^1_A)^+$ and in $(\AA^1_A)^-$. The $\GG_m$-action on $X$ is called \emph{attracting} (resp.~\emph{repelling}) if $X^+ \to X$ (resp.~$X^- \to X$) is an isomorphism.

We will be mostly interested in the affine case. Then attractor, reppeler, and fixed point locus have the following explicit description.

\begin{example}\label{DescribeAttractorRepellerAffine}
Let $A$ be a ring and let $B = \bigoplus_{n\in\ZZ} B_n$ be a $\ZZ$-graded $A$-algebra. Let $I^+, I^- \subseteq B$ be the graded ideals generated by $B_{>0}$ and $B_{<0}$ respectively and let $I^0 = I^+ + I^-$ be the ideal generated by $B_{>0} + B_{<0}$. The ideals $I^{\pm} = \bigoplus_{n\in \ZZ}I^{\pm}_n$ are explicitly given by
\begin{equation}\label{EqDescribeIMinus}
\begin{gathered}
I^{\pm}_n = \begin{cases}
B_n,&\text{if $\pm n > 0$}; \\
\sum_{i > 0}B_{\pm i}B_{\mp i+n},&\text{if $\pm n \leq 0$.}
\end{cases}
\end{gathered}
\end{equation}
One obtains that
\begin{equation}\label{EqDescribeIZero}
\begin{gathered}
I^0_n = \begin{cases}
B_n,&\text{if $n \ne 0$}; \\
\sum_{i>0}B_{-i}B_i,&\text{if $n = 0$}
\end{cases}
\end{gathered}
\end{equation}
Set $B^+ := B/I^-$, $B^- := B/I^+$, and $B^0 := B/I^0 = B_0/\sum_{i>0}B_{-i}B_i$. Then $\Spec B^+$ is the attractor, $\Spec B^-$ is the repeller, and $\Spec B^0 = (\Spec B)^{\GG_m}$ is the fixed point space of the $\GG_m$-scheme $\Spec B$ \cite[\S1]{Richarz_SpacesGGmAction}. We have a commutative diagram of maps of graded $A$-algebras
\[\xymatrix{
B_0 \ar[r] \ar[d] & B_{\geq 0} \ar[d] \ar[r] & B \ar@{=}[d] \\
B^0 \ar[r] & B^+ & B \ar[l]
}\]
where the vertical and the lower horizontal maps are surjective and the maps $B^0 \to B^+$ and $B \to B^+$ correspond to the above defined maps $X^+ \to X^0$ and $X^+ \to X$. We have a similar description with $(-)^-$ instead of $(-)^+$ and $B_{\leq0}$ instead of $B_{\geq0}$.
\end{example}

If $X$ is a group algebraic space and $\GG_m$ acts on $X$ by group  automorphisms, then $X^+$, $X^-$, and $X^0$ are functors of groups. We will need the following special case explained in \cite[2.1]{CGP_PseudoReductive}

\begin{remark}\label{GGmActionGroupScheme}
Let $A$ be a ring, let $G$ be an affine group scheme over $A$ and let
\begin{equation}\label{EqGGmActionOnGroup}
\mu\colon \GG_{m,A} \times G \lto G
\end{equation}
be a $\GG_m$-action by group automorphisms, in other words, $\mu$ is a cocharacter of $\Autline(G)$. For instance, if $\mu$ is cocharacter of $G$, then we obtain $\mu$ as above by composing $\mu$ with $G \to \Autline(G)$ given by conjugation.

For a general $\mu$ as in \eqref{EqGGmActionOnGroup} we set
\[
P^{\pm}(\mu) := G^{\pm}, \qquad G^0 = P^+(\mu) \cap P^-(\mu) = \Cent_G(\mu) := G^{\mu(\GG_m)},
\]
and these are closed $\GG_m$-invariant subgroup schemes of $G$ whose formation is compatible with arbitary base change $\Spec A' \to \Spec A$. One has $P^{\pm}(\mu^{-1}) = P^{\mp}(\mu)$. The canonical $\GG_m$-equivariant maps $P^{\pm}(\mu) \lto \Cent_G(\mu)$ are surjective homomorphisms of group schemes. We set
\[
U^{\pm}(\mu) := \Ker(P^{\pm}(\mu) \to \Cent_G(\mu)).
\]
The multiplication map yields an isomorphism of group schemes\footnote{In \cite[2.1.8]{CGP_PseudoReductive}, this is only formulated if $G$ is of finite presentation and if $\mu$ is given by a cocharacter of $G$, but the proof in loc.~cit. shows that this hypothesis is superfluous.}
\[
\Cent_G(\mu) \rtimes U^{\pm}(\mu) \liso P^{\pm}(\mu)
\]
If $G$ is of finite presentation (resp.~smooth) over $A$, then $P^{\pm}(\mu)$, $U^{\pm}(\mu)$, and $\Cent_G(\mu)$ are of finite presentation (resp.~smooth) over $A$.

If $G$ is a reductive group scheme over $A$, then the group scheme of outer automorphisms of $G$ is \'etale locally constant and hence every $\mu$ is given \'etale locally by a cocharacter of $G$. Then $P^+(\mu)$ and $P^-(\mu)$ are parabolic subgroups of $G$ with unipotent radicals $U^+(\mu)$ and $U^-(\mu)$ and common Levi subgroup $\Cent_G(\mu)$.
%
%We denote the quotient stack $[\GG_{m,A}\backslash G]$ by $\Gcal_{\mu}$. It is a group stack. 
\end{remark}

\subsection{Quasi-coherent modules and vector bundles over quotient stacks by $\GG_m$}\label{QCOHGGmSTACK}

We continue to denote by $A$ a ring and by $B$ a $\ZZ$-graded $A$-algebra. Set $X = \Spec B$, $\Xcal = [\GG_{m,A}\backslash X]$, and let $\pi\colon X \to \Xcal$ be the canonical map.

\paragraph{Quasi-coherent modules on $\Xcal$.}
The abelian category of quasi-coherent $\Oscr_{\Xcal}$-modules is equivalent to the category of $\ZZ$-graded $B$-modules, where maps are given by homogeneous $B$-linear maps of degree $0$: If $\Mscr$ is a quasi-coherent $\Oscr_{\Xcal}$-module, $\pi^*\Escr$ is a quasi-coherent module over $X = \Spec B$ and hence corresponds to a $B$-module $M$. By descent, this yields an equivalence between quasi-coherent $\Oscr_{\Xcal}$-modules $\Mscr$ and $\GG_m$-equivariant $B$-modules, i.e., with $\ZZ$-graded $B$-modules.

This equivalence respects tensor products and hence is an equivalence of symmetric monoidal abelian categories. Here the tensor product of two graded $B$-modules $M$ and $N$ is the graded $B$-module whose underlying module is the tensor product $M \otimes_B N$ obtained by forgetting the graduations. It is endowed with the $\ZZ$-grading such that $(M \otimes_B N)_d$ is generated by $m \otimes n$ with $m \in M_k$, $n \in N_{l}$ such that $k + l = d$.

Since $\pi$ is affine, $\pi_*$ is exact on quasi-coherent modules. If $M$ is a $B$-module corresponding to a quasi-coherent $\Oscr_X$-module $\Mscr$, then $\pi_*\Mscr$ is the quasi-coherent module corresponding to the graded $B$-module $M$ with $M= M_0$.

Let $f\colon \Xcal \to \Spec B_0$ be the structure morphism, then this is a good moduli space in the sense of \cite{Alper_Good}. If $\Mscr$ is a quasi-coherent $\Oscr_{\Xcal}$-module corresponding to a graded $B$-module $M$, then $f_*\Mscr$ is the quasi-coherent module corresponding to the $B_0$-module $M_0$. Conversely, if $\Nscr$ is a quasi-coherent module over $\Spec B_0$ corresponding to a $B_0$-module $N$, then $f^*\Nscr$ corresponds to the $\ZZ$-graded $B$-module $M$ with $M_0 = N$ and $M_j = 0$ for $j \ne 0$.

By descent for $\pi$, a quasi-coherent $\Oscr_{\Xcal}$-module $\Mscr$ is of finite type (resp.~of finite presentation, resp.~flat, resp.~a vector bundle) if and only if the underlying $B$-module of the corresponding $\ZZ$-graded $B$-module $M$ has the same property.

\paragraph{Twisted line bundles.}

For $e \in \ZZ$ we denote by $B(e)$ the $\ZZ$-graded $B$-module defined by $B$ with its shifted grading, i.e., $B(e)_d = B_{e+d}$. It corresponds to a line bundle on $\Xcal$ that we denote $\Oscr_{\Xcal}(e)$. For every quasi-coherent $\Oscr_{\Xcal}$-module $\Mscr$ we set $\Mscr(e) := \Mscr \otimes_{\Oscr_{\Xcal}} \Oscr_{\Xcal}(e)$. If $\Mscr$ corresponds to the graded $B$-module $M$, then $\Mscr(e)$ corresponds to the graded $B$-module $M(e)$ with $M(e)_d = M_{d+e}$.

If $\Mscr$ and $\Nscr$ are quasi-coherent $\Oscr_{\Xcal}$-modules given by $\ZZ$-graded modules $M$ and $N$, respectively, then $\Hom_{\Oscr_{\Xcal}}(\Mscr,\Nscr)$ is identified with the set of $B$-linear map $M \to N$ of degree $0$. For instance, one has
\begin{equation}\label{EqGradingGlobalSection}
\begin{aligned}
\Gamma(\Xcal,\Nscr(e)) &= \Hom_{\Oscr_{\Xcal}}(\Oscr_{\Xcal},\Nscr(e)) \\
&= \Hom_{B}(B,N(e))^0 = N_e,
\end{aligned}
\end{equation}
where $\Hom_B(-,-)^0$ denotes $B$-linear maps of $\ZZ$-graded $B$-modules of degree $0$.

\paragraph{Vector bundles and modules of finite presentation on $\Xcal$.}

A $\ZZ$-graded $B$-module is called \emph{graded free} if it is isomorphic to a direct sum of graded $B$-modules of the type $B(e)$. It is called \emph{finite graded free} if there exists $r \geq 0$ and $e_1,\dots,e_r \in \ZZ$ such that $M \cong \bigoplus_{i=1}^rB(e_i)$.

To describe vector bundles on $\Xcal$ recall that a $\ZZ$-graded $B$-module $M$ is called \emph{projective} if it satisfies the following equivalent conditions.
\begin{equivlist}
\item
$M$ is projective as an object in the abelian category of $\ZZ$-graded $B$-modules.
\item
$M$ is projective as a $B$-module.
\item
$M$ is a direct summand of a graded free $B$-module.
\end{equivlist}
Therefore we have the following characterization of vector bundles.

\begin{proposition}\label{CharVBonQuotientByGGm}
Let $\Mscr$ be a quasi-coherent $\Oscr_{\Xcal}$-module corresponding to a $\ZZ$-graded $B$-module $M$. Then the following assertions are equivalent.
\begin{assertionlist}
\item
$\Mscr$ is a vector bundle.
\item
$\Mscr$ is flat and of finite presentation.
\item
$\Mscr$ is a direct summand of a finite graded free module, i.e. to a graded module of the form $\bigoplus_{i=1}^r\Oscr_{\Xcal}(e_i)$ for some $r \geq 0$ and some $e_i \in \ZZ$.
\item
$\Mscr$ is Zariski locally on $\Xcal$ a finite graded free module.
\item
$M$ is a finite projective $B$-module.
\end{assertionlist}
\end{proposition}

To check whether $\Mscr$ is of finite presentation we will use the following criterion.

\begin{proposition}\label{CharModFP}
Let $\Mscr$ be a quasi-coherent $\Oscr_{\Xcal}$-module corresponding a $\ZZ$-graded $B$-module $M$. Then the following assertions are equivalent.
\begin{equivlist}
\item\label{CharModFPi}
$\Mscr$ is of finite type (resp.~of finite presentation).
\item\label{CharModFPii}
$M$ is of finite type (resp.~of finite presentation) as a $B$-module.
\item\label{CharModFPiii}
There exists an exact sequence of graded $B$-modules
\[
F \lto M \lto 0 \qquad\qquad (\text{resp.} \qquad F' \lto F \lto M \lto 0),
\]
where $F$ and $F'$ are finite graded free $B$-modules.
\end{equivlist}
\end{proposition}

\begin{proof}
The equivalence of \ref{CharModFPi} and \ref{CharModFPii} is clear by descent. Let us prove the equivalence of \ref{CharModFPii} and \ref{CharModFPiii}. Since $F$ and $F'$ are both finitely generated $B$-modules, it is clear that \ref{CharModFPiii} implies \ref{CharModFPii}. Conversely, suppose that $M$ is of finite type. Let $m_1,\dots,m_r$ be generators of $M$ which we may assume to be homogeneous, say of degree $e_i \in \ZZ$. Then $m_i$ defines a $B$-linear map $B(-e_i) \to M$, sending $1 \in B(-e_i)_{e_i} = B_0$ to $m_i$. Hence the $m_i$ define a surjection $u\colon \bigoplus_{i=1}^rB(-e_i) \to M$ of graded $B$-modules. If $M$ is of finite presentation, we can apply the same argument to the finitely generated $B$-module $\Ker(u)$.
\end{proof}

\paragraph{Derived categories of quasi-coherent modules.}

The category $D_{\qc}(\Xcal)$ is equivalent to the category of objects in $D(B)$ that are endowed with a $\ZZ$-grading.

\paragraph{Base change.}

Let $B \to C$ be a map of $\ZZ$-graded $A$-algebras. Then we obtain an induced map of algebraic stacks $f \colon [\Spec C/\GG_{m,A}] \to [\Spec B/\GG_{m,A}]$. If $\Escr$ is a quasi-coherent module over $[\Spec B/\GG_{m,A}]$ corresponding to a $\ZZ$-graded $B$-module $M$, then $f^*\Escr$ corresponds to the graded $C$-module $C \otimes_B M$.

We will also use the derived pullback: Let $E \in D_{\qc}([\Spec B/\GG_{m,A}])$, considered as a $\ZZ$-graded family $(M_i)_i$ with $M_i \in D(B)$, then $Lf^*E \in D_{\qc}([\Spec C/\GG_{m,A}])$ corresponds to $(M_i \Lotimes_B C)_{i\in\ZZ}$.

%---------------------------------------------------------------------

\subsection{Colimits of quotient stacks by $\GG_m$}

Our next goal is to consider colimits of quotient stacks. For this we first have to agree where we form the colimit. To have a purely algebraic theory and not only a theory over formal schemes, we would like to have that for an $I$-adically complete ring $A$ one has $\Spec A = \colim_n \Spec A/I^n$. Of course, if we form $\colim_n \Spec A/I^n$ in the category of fppf sheaves or in the $\infty$-category of all derived stacks we would obtain $\Spf A$. Hence we will form the colimit in the 2-category of Adams stacks (Definition~\ref{DefAdamsStack}) studied in detail by Sch\"appi. He proves that there, one indeed has that $\Spec A = \colim_n \Spec A/I^n$:

\begin{example}(\cite[4.3]{Schaeppi_DescentTannaka})\label{AdicColimitAdam}
Let $A$ be a ring and let $I_0 \supseteq I_1 \supseteq I_2 \supseteq \dots$ be a sequence of ideals in $A$ such that the ideal $I_n/I_{n+1}$ is contained in the Jacobson radical of $A/I_{n+1}$ for all $n \geq 0$. Set $\Ahat := \lim_nA/I_n$. Then the canonical morphisms $\Spec A/I_n \to \Spec \Ahat$ yield an isomorphism of Adams stacks
\[
\colim_n\Spec A/I_n \liso \Spec \Ahat.
\]
\end{example}

%To form the colimit of the $\Re(A_N)$ in the cocomplete 2-category of Adams stacks, we first remark that our stacks $\Re(A_N)$ are indeed Adams stacks by Proposition~\ref{QuotientStackAdam}. 
%

Let $A$ be a ring and let $n \sends B_n$ be an $\NN^{\opp}$-diagram of $\ZZ$-graded $A$-algebras $B_n = \bigoplus_{i\in \ZZ}(B_n)_i$. Define a $\ZZ$-graded $A$-algebra $B = \bigoplus_{i\in \ZZ}B_i$ with $B_i = \lim_{n}(B_n)_i$ where the multiplication $B_i \times B_j \to B_{i+j}$ is the limit of the multiplications on the $B_n$. Set $\Xcal := [\GG_{m,A}\backslash \Spec B]$ and $\Xcal_n := [\GG_{m,A}\backslash \Spec B_n]$. Note that $\Xcal$ and $\Xcal_n$ are indeed Adams stacks by Proposition~\ref{QuotientStackAdam}. We obtain maps of Adams stacks $\Xcal_n \to \Xcal_{n+1}$ for all $n$ and a map
\[
\colim_n \Xcal_n \lto \Xcal,
\]
where we form the colimit in the 2-category of Adams stacks.

\begin{proposition}\label{ColimitQuotientStack}
Suppose that the following conditions are satisfied.
\begin{definitionlist}
\item\label{ColimitQuotientStacka}
The maps of rings $(B_{n+1})_i \to (B_{n})_i$ are surjective for all $n \geq 1$ and all $i \in \ZZ$.
\item\label{ColimitQuotientStackb}
The kernel of $(B_{n+1})_0 \to (B_{n})_0$ is contained in the Jacobson radical of $(B_{n+1})_0$.
\item\label{ColimitQuotientStackc}
Then $(B_n)_0$-modules $(B_n)_i$ are finite projective for all $n \in \NN$ and $i \in \ZZ$. 
\end{definitionlist}
Then the map of Adams stacks $\colim_n \Xcal_n \lto \Xcal$ is an isomorphism.
\end{proposition}

\begin{proof}
We apply Theorem~\ref{AdamsStackCocomplete}. Therefore it suffices to to show that
\begin{assertionlist}
\item\label{ColimitQuotientStack1}
$\Vec(\Xcal) \lto \lim_{n}\Vec(\Xcal_n)$ is an equivalence and that
\item\label{ColimitQuotientStack2}
if $w\colon \Escr \to \Fscr$ is a map of finitely presented modules over $\Xcal$ such that its pullback to $\Xcal_n$ is surjective for some $n$, then $w$ is surjective.
\end{assertionlist}
Hypotheses \ref{ColimitQuotientStacka} and \ref{ColimitQuotientStackb} show by Example~\ref{AdicColimitAdam} that $\colim (B_n)_0 = B_n$ as Adams stacks. Hence $\Vec(B_0) \to \lim_n\Vec((B_n)_0)$ is an equivalence by the last assertion of Theorem~\ref{AdamsStackCocomplete}. By Hypothesis~\ref{ColimitQuotientStackc} this implies that
\[
\Vec(\Xcal) \lto \lim_{n}\Vec(\Xcal_n) \tag{*}
\]
is fully faithful if restricted to finite graded free modules. Since every vector bundle is a direct summand of a finite graded free module, this shows that (*) is fully faithful.

Let us show essential surjectivity. Let $(\Escr_n)_n \in \lim_{n}\Vec(\Xcal_n)$ which we consider as a compatible system $(M_n)_n$ of graded projective $B_n$-modules $M_n$. Choose a surjection $p_1\colon F_1 \to M_1$, where $F_1$ is finite graded free, say $F_1 = \bigoplus_s B_1(e_s)$ for finitely many $e_s \in \ZZ$. Set $F_n := \bigoplus_s B_n(e_s)$ and $F := \bigoplus_s B(e_s)$. By Hypothesis~\ref{ColimitQuotientStacka} we can lift $p_1$ inductively to $B_n$-linear maps $F_n \to M_n$. These maps are surjective in every degree by Nakayamas lemma and by Hypothesis~\ref{ColimitQuotientStackb}. Since $M_n$ is finite graded projective, we find sections $s_n$ of $p_n$ for all $n$.

We claim that we can arrange the $s_n$ inductively such that $s_{n+1}$ lifts $s_n$ for all $n$. Indeed, let $\sbar_{n+1}$ the reduction of $s_{n+1}$ and consider the composition
\[
M_{n+1} \to M_n \vartoover{50}{s_n - \sbar_{n+1}} \Ker(p_n)
\]
Since $\Ker(p_{n+1}) \to \Ker(p_n)$ is surjective by the snake lemma and since $M_{n+1}$ is graded projective, we can lift the composition to a homomorphism $t_{n+1}\colon M_{n+1} \to \Ker(p_{n+1})$ and we can replace $s_{n+1}$ by $s_{n+1} + t_{n+1}$ which proves the claim.

Hence we obtain a compatible family of idempotent endomorphisms $e_n := s_n \circ p_n$ of $F_n$ with $M_n = \Im(e_n)$ which defines an idempotent endomorphism $e$ of $F$, and $\Im(e)$ defines a vector bundle on $\Xcal$ whose reduction to $\Xcal_n$ is $\Escr_n$ for all $n$. This shows \ref{ColimitQuotientStack1}.

To prove \ref{ColimitQuotientStack2}, it suffices to show that $w$ induces a surjective map $M_j \to N_j$ for all $j$ where $M = \bigoplus M_j$ corresponds to $\Escr$ and $N = \bigoplus N_j$ to $\Fscr$. Since $\Fscr$ is of finite presentation, $N$ is the cokernel of a map of finite graded free modules. By \ref{ColimitQuotientStackc} this implies that each $N_j$ is a cokernel of a map of finite projective modules and hence of finite presentation. Therefore \ref{ColimitQuotientStack2} follows from Nakayama's lemma.
\end{proof}

Now suppose that $A$ is given the structure of an $O$-algebra, where $O$ is a Dedekind domain, and let $G$ be a flat affine group scheme of finite type over $O$. Then the classifying stack $\B{G}$ is an Adams stack by Example~\ref{BGAdams}.

\begin{corollary}\label{GeneralLimitGBundle}
Under the hypothesis of Proposition~\ref{ColimitQuotientStack}, there is a canonical isomorphism of groupoids
\[
\Bun_G(\Xcal) \iso \lim_n \Bun_G(\Xcal_n).
\]
\end{corollary}

\begin{proof}
By definition we have
\begin{align*}
\Bun_G(\Xcal) &= \Hom_O(\Xcal,\B{G}) \\
&= \Hom_O(\colim \Xcal_n,\B{G}) \\
&= \lim_n\Hom_O(\Xcal_n,\B{G}) \\
&= \lim_n\Bun_G(\Xcal_n),
\end{align*}
where the second equality holds by Proposition~\ref{ColimitQuotientStack}.
\end{proof}

%=====================================================================

\section{Filtered rings and their Rees stacks}

\subsection{Filtered modules over rings}\label{Subsec:FilteredModules}

Let $A$ be a ring. Let $A[t]$ be the polynomial ring with the $\ZZ$-grading given by $\deg(t) = -1$, let $(\AA^1_A)^-$ be the corresponding $\GG_{m,A}$-scheme, and set
\begin{equation}\label{EqDefTheta}
A^{\Fil} := [\GG_{m,A}\backslash (\AA^1_A)^-].
\end{equation}
To clarify the notation let us remark that here and in the following lines, $\GG_{m,A}\backslash$ always denotes a quotient by a $\GG_{m,A}$-action whereas all other $\setminus$-signs denote complements. The algebraic stack $A^{\Fil}$ has a stratification given by the open affine immersion
\begin{equation}\label{EqGenericTheta}
\Spec A = [\GG_{m,A}\backslash (\AA^1_A \setminus \{0\})] \lto A^{\Fil}
\end{equation}
induced by the inclusion $\AA^1_A \setminus \{0\} \to \AA^1_A$ and by the closed regular immersion of codimension 1
\begin{equation}\label{EqSpecialTheta}
\B{\GG_{m,A}} = [\GG_{m,A}\backslash \{0\}] \lto A^{\Fil}.
\end{equation}

Let $\varphi\colon A \to B$ be a map of rings. It induces a map of algebraic stacks
\begin{equation}\label{EqFilFunctorial}
\varphi^{\Fil}\colon B^{\Fil} \to A^{\Fil},
\end{equation}
and the diagram
\[\xymatrix{
B^{\Fil} \ar[rr]^{\varphi^{\Fil}} \ar[d] & & A^{\Fil} \ar[d] \\
\Spec B \ar[rr]^{\Spec \varphi} & & \Spec A
}\]
is cartesian with faithfully flat vertical maps of finite presentation.

\begin{definition}\label{DefFilteredModuleRing}
A \emph{filtered $A$-module} is a quasi-coherent $\Oscr_{A^{\Fil}}$-module. A \emph{map of filtered $A$-modules} is a homomorphism of quasi-coherent $\Oscr_{A^{\Fil}}$-modules.
\end{definition}

We obtain the abelian category of filtered $A$-modules. Let us make this definition more explicit. By Section~\ref{QCOHGGmSTACK} a quasi-coherent $\Oscr_{A^{\Fil}}$-module $\Mscr$ is given by a $\ZZ$-graded $A[t]$-module $M$ which can be described as a chain of $A$-modules
\[
\cdots \ltoover{t} M_{j+1} \ltoover{t} M_j \ltoover{t} M_{j-1} \ltoover{t} \cdots
\]
with $A$-linear maps $t\colon M_j \to M_{j-1}$. The underlying $A[t]$-module of $M$ corresponds to the quasi-coherent module obtained from $\Mscr$ by pullback via the canonical map $\AA^1_A \to A^{\Fil}$. Note that the maps $t\colon M_{j} \to M_{j-1}$ are not necessarily injective. A map of filtered $A$-modules is then a map of $\ZZ$-graded $A[t]$-modules.

In the sequel we will usually use notations like $\Mscr$, $\Escr$, $\Bscr$ etc.~to denote the quasi-coherent $\Oscr_{A^{\Fil}}$-module and $M$, $E$, $B$ etc.~to denote the corresponding $\ZZ$-graded $A[t]$-module. 

%Let $(M_j \ltoover{t} M_{j-1})_j$ and $(N_j \ltoover{t} N_{j-1})_j$ be quasi-coherent module over $\Xcal$. Then their tensor product is given by 

Occasionally we will also work with objects in the derived category $D_{\qc}(A^{\Fil})$. It can be identified with filtered derived category of $A$ by \cite{Moulinos_Filtrations}.

\begin{remark}\label{VariantsFilteredModule}
By definition, all filtrations are descending. Of course one can also work with increasing filtrations by considering quasi-coherent modules over $[\GG_m\backslash (\AA^1)^+]$, where $\GG_m$ acts on $\AA^1$ with the standard action.

More generally, let $L$ be a line bundle on $A$ and consider $\Sym_A(L) = \bigoplus_{j\geq0}L^{\otimes j}$. Then a quasi-coherent module over the quotient stack $[\GG_{m,A}\backslash \Sym_A(L)]$ is a family $(M_j)_{j\in \ZZ}$ of $A$-modules together with $A$-linear maps $M_j \otimes L \to M_{j+1}$ for all $j$.
\end{remark}

\paragraph{The underlying module.}

We can write $\AA^1_A \setminus \{0\} = \Spec A[t,t^{-1}]$ with
\[
A[t,t^{-1}] = \colim (A[t] \ltoover{t} A[t] \ltoover{t} \cdots).
\]
Therefore pullback along the open immersion \eqref{EqGenericTheta} sends a quasi-coherent $\Oscr_{A^{\Fil}}$-module $\Mscr$ corresponding to $M = (M_j \ltoover{t} M_{j-1})_j$ to the $A$-module
\[
M_{-\infty} := \colim (\cdots \ltoover{t} M_{j+1} \ltoover{t} M_j \ltoover{t} M_{j-1} \ltoover{t} \cdots),
\]
which we call the \emph{underlying $A$-module} of the filtered $A$-module $M$ or of $\Mscr$. It should not be confused with the $A$-module underlying the $A[t]$-module $M = \bigoplus_j M_j$. Occasionally, we will also denote a filtered module by $(M_{\infty}, \Fil^j)$ with $\Fil^j = M^j$, in particular if the maps $M_j \to M_{j-1}$ are all injective, e.g., in the case of vector bundles (see Proposition~\ref{CharFilteredVB} below).

More generally, for $\Mscr \in D_{\qc}(A^{\Fil})$ considered as an object $(\cdots \to M_j \to M_{j-1} \to \cdots)$ in the filtered derived category of $A$-modules, its pullback via $\Spec A \to A^{\Fil}$ is the element $\colim_{j\to-\infty} M_j$ of $D(A)$.

\paragraph{The graded module attached to a filtered module.}

Now consider the closed regular immersion $i\colon \B{\GG_{m,A}} \lto A^{\Fil}$ of codimension $1$. Let $\Mscr$ be in $D_{\qc}(A^{\Fil})$ corresponding to an object $M = (\cdots \to M_j \to M_{j-1} \to \cdots)$ in the filtered derived category of $A$-modules. Then
\[
\gr(M) := Li^*\Mscr
\]
is called the \emph{graded module attached to $M$}. It is an object in $D_{\qc}(\B{\GG_{m,A}})$ which we can also consider as a $\ZZ$-graded object in $D(A)$.

If $\Mscr$ is a filtered $A$-module given by $M = (\cdots \to M_j \to M_{j-1} \to \cdots)$, then $\gr(M)$ is given by
\begin{equation}\label{EqGradedOfFiltered}
\gr(M) = \bigoplus_{j\in \ZZ} (M_{j+1} \ltoover{t} M_j),
\end{equation}
where we consider $M_{j+1} \ltoover{t} M_j$ as a complex in $D(A)$ concentrated in degree $-1$ and $0$:
%This $\ZZ$-graded object in $D(A)$ is called the \emph{graded module attached to $M$}. It is denoted by $\gr(M)$\footnote{The description of $\gr(M)$ shows that it would be more natural to start with objects in $D_{\textup{qcoh}}(A^{\Fil})$. But since we do not need this generality, we decided to avoid this generality. Usually, we will consider only graded modules attached to filtered vector bundles $M$ in which case $\gr(M)$ is a simply a $\ZZ$-graded $A$-module, see below.}
We have $A = A[t]/(t) \cong (A[t] \ltoover{t} A[t])$ in the derived category $D(A[t])$.

In particular we find
\begin{equation}\label{EqTor1AFil}
\Tor^{A[t]}_1(A,M) = H^{-1}(\gr(M)) = \bigoplus_{j \in \ZZ} \Ker(M_{j+1} \ltoover{t} M_j).
\end{equation}
This also shows that $t\colon M \to M$ is injective if and only if $\Tor^{A[t]}_1(A,M) = 0$. In this case one has $Li^*\Escr = i^*\Escr$ which corresponds to the graded $\ZZ$-module $\bigoplus_{j \in \ZZ} \Coker(M_{j+1} \ltoover{t} M_j)$. This is for instance the case if $\Mscr$ is a flat $\Oscr_{A^{\Fil}}$-module or, equivalently, if $M$ is a flat $A[t]$-module.

\begin{defrem}\label{DefPosGraded}
A filtered $A$-module $M$ is called \emph{positively graded} if $\gr(M)_j = 0$ for all $j < 0$, in other words by \eqref{EqGradedOfFiltered}, if $t\colon M_{j+1} \to M_j$ is an isomorphism for all $j < 0$.
\end{defrem}

\paragraph{Filtered vector bundles.}

\begin{proposition}\label{CharFilteredVB}
A quasi-coherent $\Oscr_{A^{\Fil}}$-module $\Mscr$ corresponding to $(\dots \to M_j \ltoover{t} M_{j-1} \to \dots)$ is a vector bundle over $A^{\Fil}$ if and only if the following conditions are satisfied
\begin{definitionlist}
\item\label{CondVBAffineLinei}
$M_j$ is finite projective for all $j$,
\item\label{CondVBAffineLineii}
$t\colon M_j \to M_{j-1}$ is injective with finite projective cokernel for all $j$,
\item\label{CondVBAffineLineiii}
One has $M_j = 0$ for $j \gg 0$, and $t\colon M_j \to M_{j-1}$ is an isomorphism for $j \ll 0$.
\end{definitionlist}
\end{proposition}

\begin{proof}
Since every vector bundle corresponds to a direct summand of a finite graded free $A[t]$-module, the conditions are clearly necessary. Conversely, we have to show that the conditions imply that $M = \bigoplus_{j\in \ZZ}M_j$ is a finite projective $A[t]$-module. For this, we apply Proposition~\ref{FlatnessCrit} to $I = (t)$. Since \ref{CondVBAffineLineiii} holds and $M_j$ is of finite presentation for all $j$, $M$ is an $A[t]$-module of finite presentation by Proposition~\ref{CharModFP}. Moreover, \ref{CondVBAffineLinei} and \ref{CondVBAffineLineiii} imply that the $A$-module $M_{-\infty}$ is finite projective, therefore $M[1/t] = M_{-\infty} \otimes_A A[t,t^{-1}]$ is a finite projective $A[t,t^{-1}]$-module. Hence \ref{CondVBAffineLineii} implies that Condition~\ref{FlatnessCritb} of Proposition~\ref{FlatnessCrit} is satisfied by \eqref{EqTor1AFil}.
\end{proof}

Therefore the underlying module of a filtered vector bundle is again a vector bundle which can be identified with $M_j$ for $j \ll 0$. We obtain the classical version of a filtered vector bundle as, for instance, defined in \cite{Ziegler_FFF}. The attached graded module is the graded vector bundle $\bigoplus_j M_{j-1}/t(M_j)$. 

\begin{remdef}\label{VBSymGGm}
Let $L$ be a line bundle over $A$. Then working Zariski locally one sees that a vector bundle over $[\GG_{m,A}\backslash \Sym_A(L)]$ is a family $(N_j)_{j\in \ZZ}$ of finite projective $A$-modules together with $A$-linear maps $N_j \otimes_A L \to N_{j+1}$ for all $j$ such that $N_j = 0$ for $j \ll 0$, $N_j \otimes_A L \to N_{j+1}$ is an isomorphism for $j \gg 0$, and $N_j \otimes_A L \to N_{j+1}$ is injective with projective cokernel for all $j$.

We call such a datum a \emph{vector bundle with an (increasing) $L$-filtration}.
\end{remdef}

\paragraph{Pullback of filtered modules.}

Let $\psi\colon A \to C$ be a homomorphism of rings yielding a map $\psi^{\Fil}\colon C^{\Fil} \to A^{\Fil}$ \eqref{EqFilFunctorial}. If $\Mscr \in D_{\qc}(A^{\Fil})$ is a filtered object of $D(A)$, we obtain the derived pullback $L\varphi^{\Fil*}\Mscr$. Suppose that $\Mscr$ corresponds to $(M_j \ltoover{t} M_{j-1})_j$ with $M_j \in D(A)$, then $L\varphi^{\Fil*}\Mscr$ corresponds to $(M_j \Lotimes_A B \ltoover{t} M_{j-1} \Lotimes_A B)_j$.

If $\Mscr$ is a flat quasi-coherent module over $A^{\Fil}$ (e.g., if $\Mscr$ is a vector bundle over $A^{\Fil}$) or if $\varphi$ is flat, equivalently if $\varphi^{\Fil}$ is flat, then $L\varphi^{\Fil*}\Mscr =  \varphi^{\Fil*}\Mscr$.

%---------------------------------------------------------------------

\subsection{Filtered rings and their Rees stacks}\label{Subsec:FilteredRing}

Again let $A$ be a ring and define $A^{\Fil}$ as in \eqref{EqDefTheta}.

\begin{definition}\label{DefFilteredAlg}
\begin{assertionlist}
\item
A \emph{filtered $A$-algebra} is a quasi-coherent $\Oscr_{A^{\Fil}}$-algebra. A \emph{map of filtered $A$-algebras} is a map of quasi-coherent $\Oscr_{A^{\Fil}}$-algebras.
\item
A filtered $A$-algebra is called \emph{positively graded} if it is positively graded as a filtered module. 
\end{assertionlist}
\end{definition}

A filtered $A$-algebra $\Bscr$ is given by a $\ZZ$-graded $A$-algebra $B = \bigoplus B_j$ together with $A$-linear maps $t\colon B_j \to B_{j-1}$ for all $j \in \ZZ$ such that
\[
t(xy) = t(x)y = xt(y), \qquad\qquad\text{for $x \in B_j$, $y \in B_i$}.
\]
As above one sees that if $\Bscr$ is flat as a $A^{\Fil}$-algebra, then $t\colon B \to B$ is injective.

\begin{defrem}\label{DefFilteredRing}
A \emph{filtration on $A$} is a filtered positively graded $A$-algebra with corresponding $\ZZ$-graded $A$-algebra $B$ such that the canonical map $A \to B_0$ is an isomorphism.

A filtration on $A$ is a $\ZZ$-graded $A$-algebra $B = \bigoplus_{j\in \ZZ} B_j$ together with $A$-linear maps
\[
\cdots \ltoover{t} B_{j+1} \ltoover{t} B_j \ltoover{t} \cdots \ltoover{t} B_0 = A \liso B_{-1} \liso B_{-2} \liso \cdots.
\]
In this case, we also write usually $\Fil^j$ instead of $B_j$ and call $(A,(\Fil^j)_{j\geq 0})$ a \emph{filtered ring}. By definition we have an isomorphism $A[t] \iso \bigoplus_{j \leq 0}B_j$ and we usually identify $B_j$ with $A$ for $j \leq 0$.
\end{defrem}

\begin{remdef}\label{DefRees}
A quasi-coherent $\Oscr_{A^{\Fil}}$-algebra $\Bscr$ corresponds to an affine map of stacks
\begin{equation}\label{EqStructureMapRees}
\Spec \Bscr \lto A^{\Fil}.
\end{equation}
If $B$ is the $\ZZ$-graded $A[t]$-algebra corresponding to $\Bscr$ we call
\[
\Re(B) := \Spec \Bscr = [\GG_{m,A}\backslash \Spec B]
\]
the \emph{Rees stack attached to the filtered $A$-algebra $B$}. If $B$ is a filtered ring $(A,(\Fil^j))$, then we also write $\Re(A,(\Fil^j))$ or even $\Re(A)$ if there can't be any confusion about the filtration.
\end{remdef}

\begin{remdef}\label{UnderlyingRing}
Pulling back along the open immersion $\Spec A \to A^{\Fil}$ one obtains for every filtered $A$-algebra $B = (t\colon B_j \to B_{j-1})$ the \emph{underlying $A$-algebra} $B_{-\infty} = \colim_t B_j$, i.e., we have a cartesian diagram
\begin{equation}\label{EqGenericRees}
\begin{aligned}\xymatrix{
\Spec B_{-\infty} \ar[r] \ar[d] & \Re(B) \ar[d] \\
\Spec A \ar[r] & A^{\Fil}.
}\end{aligned}
\end{equation}
If $(A, (\Fil^i))$ is a filtered ring, its underlying ring is $A = B_0 \cong B_{-\infty}$.

For general filtered $A$-algebras, the notion of underlying $A$-algebra can be a bit counterintuitive, see the next Example~\ref{ExampleReesInvertible} for a non-trivial filtered $A$-algebra whose underlying $A$-algebra is the zero ring.
%Then the maps $t^j\colon B_j \to B_0$ for $j > 0$ induce a map of graded $A$-algebras
%\[
%B \lto A[t,t^{-1}]
%\]
%which is an isomorphism in degree $\leq 0$.
%The pullback to $\Bscr$ via $\Spec A \to A^{\Fil}$ is then given by $A = B_0 \cong B_{-\infty}$, i.e., the underlying ring of $\Bscr$ is $A$.
\end{remdef}

\begin{example}\label{TrivialFilteredAAlg}
Let $B_{-\infty}$ be any $A$-algebra, which we consider as quasi-coherent algebra $\Bscr_{-\infty}$ over $\Spec A$. Then pushforward along the affine open immersion $\nu\colon \Spec A \to A^{\Fil}$ defines a filtered $A$-algebra given by the $\ZZ$-graded $A[t]$-algebra $B_{-\infty}[t,t^{-1}]$.

If $\Bscr$ is any filtered $A$-algebra corresponding to $B = (t\colon B_j \to B_{j-1})$, then the canonical map $\Bscr \to \nu_*\nu^*\Bscr$ defines a map of filtered $A$-algebras
\[
B \lto B_{-\infty}[t,t^{-1}]\tag{*}
\]
which in degree $j \in \ZZ$ is given by the natural map $B_j \to B_{-\infty}$. The map (*) corresponds to a map of the associated Rees stacks
\begin{equation}\label{EqDefineTau1}
\tau\colon \Re(B_{-\infty}[t,t^{-1}]) = \Spec B_{-\infty} \lto \Re(B) 
\end{equation}
whose composition with $\Re(B) \to A^{\Fil}$ is the same as the composition $\Spec B_{-\infty} \to \Spec A \to A^{\Fil}$.

If $\Bscr$ defines a filtration $(\Fil^j)$ on $A$, then $B_{-\infty} = A$ and \eqref{EqDefineTau1} is a morphism of stacks over $A^{\Fil}$
\begin{equation}\label{EqDefineTau2}
\tau\colon \Spec A \lto \Re(A,(\Fil^j)).
\end{equation}
For a different way to describe $\tau$ in this case, see also Remark~\ref{ReesAttractor} below.
\end{example}

\begin{example}\label{ExampleReesInvertible}
Let $v\colon \Spec A \to A^{\Fil}$ be a section of the structure map $A^{\Fil} \to \Spec A$. It corresponds to a line bundle $L$ on $A$ and an $A$-linear map $v\colon L \to A$ (Example~\ref{MapsToAA1GGm}).

Define a filtered $A$-algebra $B$ by $B_j := L^{\otimes j}$ for $j \in \ZZ$ where $t\colon B_j \to B_{j-1}$ is given by multiplication with $v$. Its underlying $\ZZ$-graded algebra is $T(L)$ (Remark~\ref{MapsToQuotientbyGGm}) and we denote this filtered $A$-algebra by $T(v)$. One has
\[
\Re(T(v)) = [\GG_m\backslash \Spec T(L)] = \Spec A
\]
and the canonical map $\Spec A \to A^{\Fil}$ is just $v$. Its underlying $A$-algebra is
\[
B_{-\infty} = A[1/v] := \colim (\cdots \ltoover{v} L^{\otimes j+1} \ltoover{v} L^{\otimes j} \ltoover{v} \cdots).
\]
Here the multiplication of two elements $c,c' \in A[1/v]$ represented by $\ctilde \in L^{\otimes j}$ and $\ctilde'\in L^{\otimes j'}$ is defined as the image of $\ctilde \otimes \ctilde' \in L^{\otimes j+j'}$ in $A[1/v]$. Hence $\Spec A[1/v]$ is the largest open subscheme of $\Spec A$ on which $v$ is surjective and hence an isomorphism.

If $I = L$ is an invertible ideal of $A$ and $v$ is the inclusion, we also write $T(I)$ instead of $T(v)$ and $A[1/I]$ instead of $A[1/v]$.
For instance, if $I = (z)$ for a regular element $z \in A$ and $v$ is the inclusion, then $B_{-\infty} = A[1/z]$.

Another example is $v = 0$. In that case, the canonical map $\Spec A \to A^{\Fil}$ is the composition
\[
0\colon \Spec A \lto \B{\GG_{m,A}} \ltoover{i} A^{\Fil}.
\]
Its underlying $A$-algebra is the zero ring $A[1/0]$.
%\item\label{ExampleReesInvertible2}
%Define a filtration on $A$ by $B_j := L^{\otimes j}$ for $j \geq 0$ and $B_j := A$ for $j < 0$ and again define $t\colon B_{j+1} \to B_{j}$ to be the multiplication with $v$ for $j \geq 0$ and the identity for $j < 0$. We denote that attached Rees stack $[\GG_{m,A}\backslash \Spec B]$ by $\Re(A,v)$. If $I = L$ is an invertible ideal of $A$ and $v$ is the inclusion, we also write $\Re(A,I)$ instead of $\Re(A,v)$.
%\end{assertionlist}
\end{example}

\begin{remark}[Points of the Rees stack]\label{PointsReesStack}
Let $B = (B_{j+1} \ltoover{t} B_j)_j$ be a filtered $A$-algebra and let $C$ be an $A$-algebra. Then by Remark~\ref{MapsToQuotientbyGGm}, the groupoid $\Re(B)(C)$ is the groupoid of triples $(L,v,\alpha)$, where
\begin{definitionlist}
\item\label{PointsReesStack1}
$L$ is a line bundle on $C$ and $v\colon L \to C$ is a $C$-linear map,
\item
$\alpha\colon B \lto T(v)$ is a map of filtered $A$-algebras.
\end{definitionlist}
\end{remark}

Let us describe attractor, repeller and fixed locus in Rees stacks of filtered rings.

\begin{remdef}\label{ReesAttractor}
Let $(A,(\Fil^j))$ be a filtered ring, i.e., it is given by
\[
\cdots \ltoover{t} \Fil^{j+1} \ltoover{t} \Fil^j \ltoover{t} \cdots \ltoover{t} \Fil^0 = A \liso A \liso A \liso \cdots
\]
and $\Fil^j = A$ for $j < 0$. Let $B = \bigoplus_{j\in \ZZ}\Fil^j$ be the associated $\ZZ$-graded $A$-algebra. By \eqref{EqDescribeIMinus} and \eqref{EqDescribeIZero} the ideals defining the attractor, repeller, and fixed locus of $\Spec B$ are given by the $\ZZ$-graded ideals
\begin{equation}\label{EqIdeaFixRees}
\begin{aligned}
I^- &= tB = tA[t] \oplus \bigoplus_{j\geq0}t(\Fil^{j+1}), \\
I^+ &= \bigoplus_{j\leq0}t^{-j+1}(\Fil^{1}) \oplus \bigoplus_{j > 0}\Fil^{j}, \\
I^0 &= I^+ + I^- = t(\Fil^1) \oplus \bigoplus_{j \ne 0}\Fil^{j}.
\end{aligned}
\end{equation}
Hence we obtain
%\footnote{Here it would be more conceptual to define these loci in the derived setting. For instance, the ``correct'' definition of $\Rees(A)^0$ should be the 1-truncated animated ring $\cofib(d\colon \Fil^1 \to A)$ corresponding to the quasi-ideal $\Fil^1 \to A$. Here we define only $\pi_0$ of it.}
\begin{equation}\label{EqFixRees}
\begin{aligned}
B^+ = B/I^- &= \bigoplus_{j \geq 0} \Fil^j/t(\Fil^{j+1}), \\
B^- = B/I^+ &= \bigoplus_{j \leq 0} A/t^{-j+1}(\Fil^{1}) = (A/t(\Fil^1))[t], \\
B^0 = B/I^0 &= A/t(\Fil^1).
\end{aligned}
\end{equation}
For $? \in \{+,-,0\}$ we denote by
\[
\Re(B)^{?} := [\GG_{m,A}\backslash (\Spec B^{?})]
\]
the corresponding quotient stacks. Then we have
\begin{equation}\label{EqFilRingReppellorFix}
\Re(B)^0 = \B{\GG_{m,B^0}}, \qquad\qquad \Re(B)^- = (A/t\Fil^1)^{\Fil}
\end{equation}
and obtain a cartesian diagram of closed immersions
\begin{equation}\label{EqCartesianAttractorRepellor}
\begin{aligned}\xymatrix{
 & \Re(B)^+ \ar[rd] \\
\Re(B)^0 \ar[ru] \ar[rd] & & \Re(B), \\
 & \Re(B)^- \ar[ru]
}\end{aligned}
\end{equation}
i.e., $\Re(B)^0$ is the stack theoretic intersection of the closed substacks $\Re(B)^+$ and $\Re(B)^-$. We also define for $? \in \{+,-,0\}$ open substacks of $\Re(B)$ by
\[
\Re(B)^{\ne ?} := \Re(B) \setminus \Re(B)^?.
\]
Then
\[
\Re(B)^{\ne0} = \Re(B)^{\ne+} \cup \Re(B)^{\ne-}.
\]
Since $B^+ = B/(t)$, we have a cartesian diagram
\[\xymatrix{
\Re(B)^+ \ar[r] \ar[d] & \Re(B) \ar[d] & \Re(B)^{\ne +} \ar[l] \ar[d]^{\sim}\\
\B{\GG_{m,A}} \ar[r] & A^{\Fil} & \Spec A \ar[l],
}\]
where the right vertical arrow is an isomorphism since we can also describe $\Re(B)^{\ne+}$ as the locus where $t$ is invertible in $\Re(B)$ and \eqref{EqGenericRees} yields an isomorphism
\begin{equation}\label{EqReesStackNePlus}
\Spec A = \Spec B_{-\infty} \iso \Re(B)^{\ne+}.
\end{equation}
Composing this isomorphism with the inclusion $\Re(B)^{\ne+} \to \Re(B)$ is then the map $\tau$ defined in \eqref{EqDefineTau2}.
\end{remdef}

\begin{remark}\label{ExtensionFilteredRing}
Let $(A, (\Fil_A^j)_j)$ be a filtered ring and let $\psi\colon A \to C$ be a ring homomorphism. Set $\Fil^j_C := C \otimes_{A} \Fil^j_A$ for $j \in \ZZ$ with induced maps $t\colon \Fil_C^{j+1} \to \Fil_C^j$. Then $(C, (\Fil_C^j)_j)$ is again a filtered ring, denoted by $\psi^*(A, (\Fil_A^j)_j)$.

%Let $f\colon \Spec C \to \Spec A$ be the map of affine schemes corresponding to $\psi$ and let $f^{\Fil}\colon C^{\Fil} \to A^{\Fil}
%\]
%and $(t\colon C_{j+1} \to C_j)$ corresponds to the quasi-coherent $\Oscr_{C^{\Fil}}$-algebra $f^{\Fil*}\Bscr$.
%
Let $\Re(A)$ and $\Re(C)$ be the corresponding Rees stacks. Then one has a 2-cartesian diagram of algebraic stacks
\[\xymatrix{
\Re(C) \ar[r]^{\Re(\psi)} \ar[d] & \Re(A) \ar[d] \\
C^{\Fil} \ar[r]^{\psi^{\Fil}} \ar[d] & A^{\Fil} \ar[d] \\
\Spec C \ar[r]^{\Spec \psi} & \Spec A.
}\]
Similarly, one has
\begin{equation}\label{EqExtendAttractor}
\begin{aligned}
\Re(C)^{?} &= C \otimes_{A} \Re(A)^{?}, \\
\Re(C)^{\ne ?} &= C \otimes_A \Re(A)^{\ne ?}
\end{aligned}
\end{equation}
for $? \in \{+,-,0\}$.
\end{remark}

%\begin{remark}\label{ExtensionFilteredRing}
%Let $(A,(\Fil_A^j))$ be a filtered ring and let $\psi\colon A \to C$ be a ring homomorphism. Set $\Fil^j_C := C \otimes_A \Fil_A^j$ for $j \geq 0$. Then $(C, (\Fil^j_C)_j)$ is again a filtered ring, denoted by $(A,(\Fil_A^j)) \otimes_A C$ or by $\psi^*(A,(\Fil_A^j))$.
%
%Let $\Ascr$ be the quasi-coherent $\Oscr_{A^{\Fil}}$-algebra corresponding to $(A,(\Fil_A^j))$ and let $f\colon \Spec C \to \Spec A$ be the map of affine schemes corresponding to $\psi$. Then $f$ induces a map
%\[
%\Theta_f = \Theta_{\psi}\colon \Theta_C \to A^{\Fil}
%\]
%and $(C, (\Fil^j_C)_j)$ corresponds to the quasi-coherent $\Oscr_{\Theta_C}$-algebra $\Theta_f^*\Ascr$.
%
%If $B$ is the $\ZZ$-graded $A$-algebra defined by $(A,(\Fil_A^j))$, then the $\ZZ$-graded $C$-algebra defined by $(C, (\Fil^j_C))$ is $C \otimes_A B$. This shows that one has a 2-cartesian diagram of algebraic stacks
%\[\xymatrix{
%\Re(C, (\Fil^j_C)) \ar[r] \ar[d] & \Re(A,(\Fil_A^j)) \ar[d] \\
%\Theta_C \ar[r] & A^{\Fil},
%}\]
%where the horizontal maps are induced by $\psi$. Similarly, one has
%\begin{equation}\label{EqExtendAttractor}
%\begin{aligned}
%\Re(C,(\Fil^j_C))^{?} &= \Theta_C \times_{A^{\Fil}} \Re(A,(\Fil^j_A))^{\ne ?}, \\
%\Re(C,(\Fil^j_C))^{\ne ?} &= \Theta_C \times_{A^{\Fil}} \Re(A,(\Fil^j_A))^{\ne ?}
%\end{aligned}
%\end{equation}
%for $? \in \{+,-,0\}$.
%\end{remark}

\begin{defrem}\label{DefMapFilteredRings}
A \emph{map of filtered rings} $(A (\Fil^j_A)) \to (C, (\Fil^j_C))$ is a pair consisting of a map $\psi\colon A \to C$ of rings and a map of filtered $C$-algebras $\psi_{\bullet}\colon \psi^*(A, (\Fil^j_A)) \to (C, (\Fil^j_C))$.

Equivalently, a map of filtered rings is given by a ring homomorphism $\psi\colon A \to C$ and by $A$-linear maps $\psi_j\colon \Fil^j_A \to \Fil^j_C$ for $j \in \ZZ$ such that $\psi_0 = \psi$ and such that for all $j \in \ZZ$ the diagram
\[\xymatrix{
\Fil^{j+1}_A \ar[r]^t \ar[d]_{\psi_{j+1}} & \Fil^j_A \ar[d]^{\psi_j} \\
\Fil^{j+1}_C \ar[r]^t & \Fil^j_C
}\]
commutes.
\end{defrem}

%---------------------------------------------------------------------

\subsection{Filtered modules over filtered rings}

\begin{definition}\label{DefFilteredModule}
Let $(A, (\Fil^j))$ be a filtered ring and let $\Re(A)$ be the attached Rees stack. Then a \emph{filtered module over $(A, (\Fil^j))$} is a quasi-cohrent module on $\Re(A)$.
\end{definition}

%If $(t\colon B_{j+1} \to B_j) = (A,(\Fil^i_A))$ is a filtered ring, then
%the category of quasi-coherent modules over $\Re(B)$ is equivalent to the category $\ZZ$-graded modules $M = \bigoplus_{j\in\ZZ}M_j$ over $B$. This endows each $M_j$ in particular with the structure of a $A$-module. Moreover, the multiplication with $t$ yields an $A$-linear map $t\colon M_j \to M_{j-1}$ for all $j$ and this determines the scalar multiplication of $\Rees(A)_{\leq0} = A[t]$. Altogether we see that
Then a filtered $(A, (\Fil^j))$-module can be also described as a $\ZZ$-graded $A$-module $M = \bigoplus_{i\in \ZZ}M_i$ endowed with the following additional structure.
\begin{definitionlist}
\item
For all $i \in \ZZ$ an $A$-linear map $t\colon M_i \to M_{i-1}$ and
\item
for all $j > 0$ and $i \in \ZZ$ an $A$-linear map $\Fil^j \otimes_A M_i \to M_{i+j}$ such that the diagram
\[\xymatrix{
\Fil^j \otimes_A M_i \ar[r] \ar[d] & M_{i+j} \ar[d]^t \\
\Fil^{j-1} \otimes_A M_i \ar[r] & M_{i+j-1}
}\]
commutes, where the left vertical map is induced by $\Fil^j \to \Fil^{j-1}$. 
\end{definitionlist}
We usually write $(M_{i+1} \to M_i)_i$ for such a datum.

We use the notation of Section~\ref{QCOHGGmSTACK}. In particular, for $e \in \ZZ$ we have the line bundle $\Oscr_{\Re(A)}(e)$ on $\Re(A)$. It corresponds to $M = \bigoplus_{j\in \ZZ}M_j$ with $M_j = \Fil^{i+e}$.

For a general quasi-coherent module $\Escr$ over $\Re(A)$ we set
\[
\Escr(e) := \Escr \otimes \Oscr_{\Re(A)}(e).
\]
If $\Escr$ is a quasi-coherent module on $\Re(A)$ with corresponding graded module $M = \bigoplus M_j$, then $\Gamma(\Re(A),\Escr) = M_0$ and more generally $\Gamma(\Re(A),\Escr(j)) = M_j$ for all $j \in \ZZ$ by \eqref{EqGradingGlobalSection}.

By Proposition~\ref{CharVBonQuotientByGGm}, a quasi-coherent module $\Escr$ on $\Re(A)$ is a vector bundle on $\Re(A)$ if and only if $\Escr$ is a direct summand of a module of the form of a finite direct sum $\bigoplus_i\Oscr(e_i)$ for finitely many $e_i \in \ZZ$.

%Let $R := A/t(\Fil^1)$ as before and consider the closed immersion
%\[
%i\colon \B{\GG_{m,R}} = \Re(B)^0 \lto \Re(B).
%\]
%The (non-derived) pullback of a quasi-coherent module $M$ over $\Re(B)$ via $i$ is given by the graded $R$-module $M/t(\Fil^1)M = \bigoplus_j M_j/t(\Fil^1)M_j$. If $M$ is a vector bundle, then $M/t(\Fil^1)M$ is a graded finitely generated projective $R$-module, where $R$ is endowed with the trivial grading.
%, i.e.,  $M/d(\Fil^1)M = \bigoplus_jL_j(j)$.

%=====================================================================

\section{The display group}

\subsection{Extension and liftings of $G$-bundles}

We will use the following two special cases of results in \cite{Wedhorn_ExtendBundles}. We continue to denote by $A$ a ring.

\begin{definition}\label{DefHenselianGradedRing}
Let $B = \bigoplus_{i\in \ZZ}B_i$ be a $\ZZ$-graded ring and set $J := \sum_{i>0}B_iB_{-i} \subseteq B_0$. Then we call $B$ \emph{henselian} if $(B_0,J)$ is a henselian pair.
\end{definition}

For instance, every $\ZZ$-graded ring $B$ with $B_i = 0$ for $i < 0$ is henselian.

\begin{theorem}(\cite[2.6]{Wedhorn_ExtendBundles})\label{LiftGBundles}
Let $B$ be a henselian $\ZZ$-graded $A$-algebra, set $X := \Spec B$, $\Xcal := [\GG_{m,A}\backslash X]$, and $R := B_0/\sum_{i>0}B_iB_{-i}$. Let $X^{+}$, $X^-$, and $X^0 = \Spec R$ be the attractor, the repeller, and the fixed point locus, of the $\GG_m$-scheme $X$ respectively. Set $\Xcal^{?} := [\GG_m\backslash X^{?}]$ for $? \in \{+,-,0\}$. In particular $\Xcal^0 = \B{\GG_{m,R}}$.

Let $G$ be an affine smooth group scheme over $A$. Then pullback by the inclusions induces full and essentially surjective morphisms of groupoids
\[\xymatrix{
 & \Bun_G(\Xcal^+) \ar[dl] \\
\Bun_G(\B{\GG_{m,R}}) & & \Bun_G(\Xcal) \ar[ul] \ar[dl] \\
& \Bun_G(\Xcal^-) \ar[ul].
}\]
In particular we have isomorphism of pointed sets
\[
H^1(\Xcal,G) \cong H^1(\Xcal^{\pm},G) \cong H^1(\B{\GG_{m,R}},G).
\]
\end{theorem}

\begin{corollary}\label{LiftingGBundlesReesStack}
Let $(A, (\Fil^j))$ be a filtered ring (Definition~\ref{DefFilteredRing}) such that $(A,t(\Fil^1))$ is a henselian pair. Set $R := A/t(\Fil^1)$. Then for every smooth affine group scheme $G$ over $A$, one has bijections
\begin{equation}\label{EqClassifyBundlesGStack}
H^1(\Re(A),G) \cong H^1(\Re(A)^+,G) \cong H^1(\Re(A)^-,G) \cong H^1(\B{\GG_{m,R}}).
\end{equation}
\end{corollary}

%\begin{example}\label{ReesZAdic}
%Let $A$ be a ring and let $z \in A$ be an element. Endow $A$ with the $z$-adic filtration (Section~\ref{Sec:ReesZAdic}). Setting $u := zt^{-1} \in \Rees(A)_1$ we obtain a presentation
%\begin{equation}\label{EqReesZAdic}
%\Rees(A) = A[t,u]/(tu - z), \qquad \deg(t) = -1, \deg(u) = 1.
%\end{equation}
%We have $R := A/d(\Fil^1) = A/(z)$ and
%\[
%\Rees(A)^+ = R[u], \qquad \Rees(A)^- = R[t], \qquad \Rees(A)^0 = R
%\]
%and hence the $t = 0$ (resp.~$u = 0$) locus in $\Re(A)$ is $\Re(A)^+ \cong [(\AA^1_R)^+/\GG_m]$ (resp.~$\Re(A)^- \cong [(\AA^1_R)^-/\GG_m]$). Their stack theoretic intersection is the $u = t = 0$ locus $\B{\GG_{m,R}}$ in $\Re(A)$. The composition $\Re(A)^- = [(\AA^1_R)^-/\GG_{m,R}] \to \Re(A)$ with the Rees map $\Re(A) \to [(\AA^1_A)^-/\GG_{m,A}]$ is the natural closed embedding.
%
%For every smooth affine group scheme $G$ over $A$ we obtain by Corollary~\ref{LiftingGBundlesReesStack} an isomorphism
%\begin{equation}\label{EqGBundle}
%H^1(\Re(A),G) \iso H^1(\B{\GG_{m,R}}).
%\end{equation}
%The open $u \ne 0$ (resp.~$t \ne 0$) locus in $\Rees(A)$ is given by $A[u,u^{-1}]$ (resp.~$A[t,t^{-1}]$). Their intersection is given by $A[1/z][u,u^{-1}] \cong A[1/z][t,t^{-1}]$. Therefore $\Re(A)^{\ne+} \cong \Spec A$, $\Re(A)^{\ne-} \cong \Spec A$, and $\Re(A)^{\ne+} \cap \Re((A)^{\ne-} = \Re(A)_{ut \ne 0} \cong \Spec A[1/z]$.
%\end{example}
%

\begin{theorem}(\cite[1.7, 1.13]{Wedhorn_ExtendBundles})\label{ExtendGBundle}
Let $B$ be a $\ZZ$-graded $A$-algebra, set $X := \Spec B$ and $\Xcal := [\GG_{m,A}\backslash X]$. Let $j\colon U \to X$ be the inclusion of an open quasi-compact $\GG_{m,A}$-invariant subscheme such that $\Oscr_X \iso j_*\Oscr_U$.
\begin{assertionlist}
\item\label{ExtendGBundle1}
For every affine flat group scheme $G$ over $A$, the pullback
\[
j^*\colon \Bun_G([\GG_m\backslash X]) \lto \Bun_G([\GG_m\backslash U])
\]
is fully faithful.
\item\label{ExtendGBundle2}
If $X$ is regular of dimension $2$ and $U$ contains every point of codimension $\leq 1$ of $X$ and if $G$ is a reductive group scheme over $A$, then $j^*$ is an equivalence of groupoids.
\end{assertionlist}
\end{theorem}

In Theorem~\ref{ExtendGBundle}~\ref{ExtendGBundle2} it suffices that $G$ is geometrically reductive in the sense of \cite[Section~9]{Alper_Adequate}.

%---------------------------------------------------------------------

\subsection{The stack of types of $G$-bundles for a reductive group scheme $G$}\label{Sec:StackTypesGBundles}

Let $S$ be an algebraic space and let $G$ be a smooth affine group algebraic space over $S$.

The groupoid $\Bun_G(\B{\GG_{m,S}})$ can be described as the groupoid of pairs $(\Escr,\lambda)$, where $\Escr$ is a $G$-bundle on $S$ and where $\lambda\colon \GG_{m,S} \to \Autline(\Escr)$ is a cocharacter by \cite[A.36]{Wedhorn_ExtendBundles}. A morphism $(\Escr,\lambda) \to (\Escr',\lambda')$ in $\Bun_G(\B{\GG_{m,S}})$ is an isomorphism $\Escr \iso \Escr'$ of $G$-bundles intertwining $\lambda$ and $\lambda'$.

\begin{definition}\label{DefTypeStack}
The stack $\Bunline_G(\B{\GG_{m,S}})$ on $\Affrel{S}$ that sends an affine $S$-scheme $T$ to the groupoid $\Bun_G(\B{\GG_{m,T}})$ is called the \emph{stack of types of $G$-bundles}.
\end{definition}

Recall that we have the following alternative description of $\Bunline_G(\B{\GG_{m,S}})$.

\begin{void}
Let
\[
X_*(G) := \Hom_{\textup{$S$-Grp}}(\GG_{m,S},G), \qquad \Xline_*(G) := \Homline_{\textup{$S$-Grp}}(\GG_{m,S},G)
\]
be the set (resp.~sheaf) of cocharacters of $G$. By \cite[XV, 8.11]{SGA3II}, $\Xline_*(G)$ is representable by an algebraic space which is surjective, smooth, and separated over $S$. Since $G$ is affine over $S$, the algebraic space $\Xline_*(G)$ is ind-quasi-affine over $S$ by \cite[5.8]{Cotner_HomSchemes}. Moreover, $G$ acts on $\Xline_*(G)$ by conjugation from the left and sending a cocharacter $\mu$ of $G$ to the pair $(P_0,\mu)$, where $P_0$ is the trivial $G$-bundle, yields an isomorphism
\[
[G\backslash \Xline_*(G)] \liso \Bunline_G(\B{\GG_{m,S}})
\]
by \cite[Section A.7]{Wedhorn_ExtendBundles}. In particular, the stack of types $\Bunline_{G}(\B{\GG_{m,S}})$ is a smooth algebraic stack with affine diagonal.
\end{void}

\begin{void}[Coarse moduli space of $\Bunline_G(\B{\GG_{m,S}})$]\label{FppfConjClasses}
Let $\Ccal_G$ be the coarse fppf-sheaf of $[G\backslash \Xline_*(G)]$, i.e., the fppf-sheafification of the presheaf on $\Affrel{S}$ that attaches to an affine $S$-scheme $S'$ the set of isomorphism classes of $[G\backslash \Xline_*(G)](S')$. Since $G \to S$ and $\Xline_*(G) \to S$ are both fppf-covering, so is the inertia stack
\[
\Ical_{[G\backslash \Xline_*(G)]} = [G\backslash (G \times_S \Xline_*(G))] \lto [G\backslash \Xline_*(G)].
\]
Therefore $\Ccal_G$ is a smooth algebraic space over $S$ (and in particular an fpqc-sheaf by Gabber's theorem) and $[G \backslash \Xline_*(G)] \to \Ccal_G$ is smooth and a gerbe (lacking a better reference: apply \cite[A.1]{AOV_TamePosChar} to the rigidification obtained by ``removing'' the inertia subgroup stack itself). For every geometric point $\sbar \to S$ we have
\[
H^1(G_{\sbar},\B{\GG_{m,\sbar}}) = \Ccal_G(\sbar).
\]
\end{void}

%Let us describe the fibers of the smooth gerbe $[G \backslash \Xline_*(G)] \to \Ccal_G$.
%
%\begin{remark}\label{FiberTypeGerbe}
%Let $\mgbar \in \Ccal_G(S)$. Choose an \'etale covering $f\colon S' \to S$ such that there exists a cocharacter $\mu\colon \GG_{m,S'} \to G_{S'}$ representing $f^*\mgbar$. Then the choice of $\mu$ yields a 2-cartesian diagram
%\[\xymatrix{
%\B{\Cent_{G_{S'}}(\mu)} \ar[r] \ar[d] & S' \ar[d]^{f^*\mgbar} \\
%[G \backslash \Xline_*(G)] \ar[r] & \Ccal_G.
%}\]
%Therefore the fiber of $[G \backslash \Xline_*(G)] \to \Ccal_G$ over $\mgbar$ is a gerbe over $S$ which is \'etale locally isomorphic to $\B{\Cent_{G_{S'}}(\mu)}$ for some $\mu$ as above.
%\end{remark}

\begin{defrem}\label{DefineGModMu}
Let $[\mu] \in [G\backslash \Xline_*(G)](S')$ be a conjugacy class of cocharacters defined over some $S$-scheme $S'$.
\begin{assertionlist}
\item\label{DefineGModMu1}
The corresponding $G$-bundle over $\B{\GG_{m,S'}}$ will be denoted by $\Escr_{[\mu]}$. If $[\mu]$ can be represented by some cocharacter $\mu$ of $G$, we also write $\Escr_{\mu}$ instead of $\Escr_{[\mu]}$.
\item\label{DefineGModMu2}
Also define the relatively affine auutomorphism group stack of $\Escr_{[\mu]}$ as
\begin{equation}\label{EqDescribeMuG}
[\mu]\backslash G_{S'} := \Autline(\Escr_{[\mu]}).
\end{equation}
Again, we write $\mu\backslash G_{S'}$ if $[\mu]$ is represented by a cocharacter $\mu$.

Then $[\mu]\backslash G_{S'}$ is a group stack over $\B{\GG_{m,S'}}$ sitting in a 2-cartesian diagram
\begin{equation}\label{EqGModMu}
\begin{aligned}\xymatrix{
G_{S'} \ar[r] \ar[d] & [\mu]\backslash G_{S'} \ar[d] \\
S' \ar[r] & \B{\GG_{m,S'}},
}\end{aligned}
\end{equation}
If $[\mu]$ can be represented by a cocharacter $\mu$ of $G_{S'}$, then $\mu\backslash G_{S'}$ is the quotient stack of $G_{S'}$ by the $\GG_m$-action via conjugation with $\mu$.

We may view $\Escr_{[\mu]}$ also as the trivial $[\mu]\backslash G_{S'}$-bundle via twisting \cite[(A.6.1)]{Wedhorn_ExtendBundles}.

\item\label{DefineGModMu3}
For every $S'$-scheme $T$ we define $\Bun_G^{[\mu]}(\B{\GG_{m,T}})$ to be the full subgroupoid of $\Bun_G(\B{\GG_{m,T}})$ of $G$-bundles that are \'etale locally on $T$ isomorphic to $\Escr_{[\mu]}\rstr{\B{\GG_{m,T}}}$. These $G$-bundles are said to be \emph{of type $[\mu]$}.

More generally, let $\Xcal$ be any stack over $\B{\GG_{m,S'}}$. Then a $G$-bundle $\Fscr$ on $\Xcal$ is said to be \emph{of type $[\mu]$} (or \emph{of type $\mu$} if a representing cocharacter $\mu$ of $[\mu]$ is chosen) if there exists an \'etale covering $\Xcal' \to \Xcal$ such that $\Fscr\rstr{\Xcal'} \cong \Escr_{[\mu]}\rstr{\Xcal'}$.
\item
Let $\Bunline^{[\mu]}_G(\B{\GG_{m,S'}})$ be the substack of $\Bunline_G(\B{\GG_{m,S'}})$ sending an $S'$-scheme $T$ to $\Bun^{[\mu]}_G(\B{\GG_{m,T}})$.
\end{assertionlist} 
\end{defrem}

\begin{remark}\label{MuGGeneral}
Occasionally, we will apply this construction of $[\mu]\backslash G$ to the more general case that $[\mu]$ is an $S'$-valued point of the quotient stack $[G\backslash \Homline(\GG_{m,S},\Autline(G))]$, where $G$ acts via conjugaction, i.e., $[\mu]$ is a $G$-conjugacy class of cocharacters of $\Autline(G)$. In this case we define $\Escr_{[\mu]}$ as the trivial $[\mu]\backslash G$-bundle over $\B{\GG_{m,S'}}$. 
\end{remark}

\begin{proposition}\label{BunlineMuStack}
With keep the notation as in Definition~\ref{DefineGModMu}. Then the following assertions hold
\begin{assertionlist}
\item
$\Bunline^{[\mu]}_G(\B{\GG_{m,S'}})$ is a neutral gerbe. More precisely, there exists a smooth affine group algebraic space $L$ over $S'$, an \'etale covering $S''$ of $S'$, a cocharacter $\mu$ of $G_{S''}$ representing $[\mu]$, and an isomorphism $L_{S''} \iso \Cent_{G_{S''}}(\mu)$ of group algebraic spaces over $S''$ such that
\begin{equation}\label{EqDescribeBunGMu}
\Bunline^{[\mu]}_G(\B{\GG_{m,S'}}) = \B{L}.
\end{equation}
\item
$\Bunline^{[\mu]}_G(\B{\GG_{m,S'}})$ is a smooth quasi-compact algebraic fpqc-stack with affine diagonal over $S'$.
\end{assertionlist}
\end{proposition}

%In the general situation of Definition~\ref{DefBunlineGMu} a $[\mu]$ as in Proposition~\ref{BunlineMuStack} always exists \'etale locally on $S$. Hence the prestack $\Bunline^{\mu}_G(\B{\GG_{m,S}})$ is always a stack \'etale locally on $S$.
%
\begin{proof}
Set $\Xcal := \Bunline_G(\B{\GG_{m,S'}})$ and let $L := \Autline_{\Xcal}([\mu])$ be the group algebraic space of automorphisms of the $S'$-valued point $[\mu]$ of $\Xcal$. Then the natural monomorphism $\B{L} \to \Xcal$ induces by \cite[A.5]{Wedhorn_ExtendBundles} an isomorphism $\B{L} \iso \Bunline_G^{[\mu]}(\B{\GG_{m,S'}})$. Now choose an \'etale covering $S'' \to S'$ such that $[\mu]$ can be represented over $S''$ by a cocharacter $\mu$ of $G_{S''}$. Since $\Autline(G_{S''},\mu) = \Cent_{G_{S''}}(\mu)$ we find $L_{S''} \cong \Cent_{G_{S''}}(\mu)$. Since $\Cent_{G_{S''}}(\mu)$ is a smooth affine algebraic group space over $S''$ (Remark~\ref{GGmActionGroupScheme}), so is $L$ over $S'$ by descent. In particular $\Bunline_G^{\mu}(\B{\GG_{m,S'}}) = \B{L}$ is a smooth quasi-compact algebraic fpqc-stack with affine diagonal.
\end{proof}

%\begin{remark}\label{DescribeBunlineMu}
%By definition, $\Bunline^{\mu}_G(\B{\GG_{m,S}})$ is a gerbe for the \'etale topology. 
%If $G$ is affine and smooth over $S$, then $\Cent_G(\mu)$ is affine and smooth over $S$ (Remark~\ref{GGmActionGroupScheme}). In particular, $\Bunline^{\mu}_G(\B{\GG_{m,S}})$ is a smooth quasicompact algebraic stack over $S$ with affine diagonal.
%
%The inclusion $\Bunline^{\mu}_G(\B{\GG_{m,S}}) \to \Bunline_G(\B{\GG_{m,S}})$ is a monomorphism.
%\end{remark}
%
%Let us parametrize the possible $\mu$ as in Definition~\ref{GBundlesTypeMu}.
%

There is the following alternative definition of $\Bunline^{[\mu]}_G(\B{\GG_{m,S}})$.

\begin{defrem}\label{AlternativeGBundleTypeMu}
Let $S'$ be an $S$-scheme and let $[\mu] \in [G\backslash \Xline_*(G)](S')$. Define a stack $B([\mu]\backslash G_{S'})$ over $\B{\GG_{m,S'}}$ as follows. If $T$ is a scheme and $f\colon T \to \B{\GG_{m,S'}}$ let $B([\mu]\backslash G_{S'})(T)$ be the groupoid of $f^*([\mu]\backslash G_{S'})$-bundles over $T$. By \eqref{EqGModMu} there is a 2-cartesian diagram
\[\xymatrix{
\Bunline^{[\mu]}_G(\B{\GG_{m,S'}}) \ar[r] \ar[d] & \B([\mu] \backslash G_{S'}) \ar[d] \\
S' \ar[r] & \B{\GG_{m,S'}}.
}\]
Since $\Bunline^{[\mu]}_G(\B{\GG_{m,S'}}) \to S'$ is a smooth algebraic gerbe with affine diagonal (Proposition~\ref{BunlineMuStack}), so is $B([\mu] \backslash G_{S'}) \to \B{\GG_{m,S'}}$.
\end{defrem}

%POSSIBLY TODO: Rewrite this for non-reductive groups, at least for non-connected groups with reductive unit component
%

From now on suppose that $G$ is reductive over $S$, i.e., $G \to S$ is affine and smooth with connected reductive geometric fibers.

\begin{remdef}\label{BasedRootDatum}
 Locally for the \'etale topology on $S$ one can choose a pinning of $G$ \cite[XXV, 1.1]{SGA3III} which in turn yields a based root datum $(X^*,\Phi,X_*,\Phi\vdual,\Delta)$. By descent, this allows to attach to $G$ globally a tuple of \'etale locally constant schemes over $S$ which we denote by $(\underline{X}^*,\underline{\Phi},\underline{X}_*,\underline{\Phi}\vdual,\underline{\Delta})$ and call it \emph{the} based root datum of $G$. By definition of a based root datum, this quintuple of locally constant schemes is endowed with additional structures: $\underline{X}^*$ and $\underline{X}_*$ are locally constant sheaves that are vector bundles over the constant sheaf of rings $\underline{\ZZ}$ and there is given a perfect pairing of $\underline{X}^*$ and $\underline{X}_*$ with values in $\underline{\ZZ}$. The schemes $\underline{\Phi}$, $\underline{\Phi}\vdual$, and $\underline{\Delta}$ are each locally constant finite $S$-schemes, i.e. finite \'etale $S$-schemes. There are given open and closed immersions $\underline{\Delta} \to \underline{\Phi} \to \underline{X}^*$ and $\underline{\Phi}\vdual \to \underline{X}_*$ and an isomorphism $\underline{\Phi} \iso \underline{\Phi}\vdual$ of locally constant sheaves denoted by $\alpha \sends \alpha\vdual$. All these data have to satisfy the conditions of a based root datum \cite[XXI, 1.1.1, 3.1.6]{SGA3III}.
If $S' \to S$ is an \'etale covering such that $G$ is split over $S'$, then the choice of a split Borel pair $T \subseteq B \subseteq G_{S'}$ yields a canonical isomorphism of $(\underline{X}^*(S'),\underline{\Phi}(S'),\underline{X}_*(S'),\underline{\Phi}\vdual(S'),\underline{\Delta}(S'))$ with the based root datum of $G$ defined by $(T,B)$.

To the based root datum of $G$ one attaches its Weyl group $\underline{W}_G$. It is a finite \'etale group scheme. It is equipped with an open and closed immersion of an $S$-scheme $\Iline\to \Wline_G$, the scheme of simple reflections in $\Wline_G$. The finite \'etale group scheme $\Wline_G$ acts on $\Xline^*$ and $\Xline_*$. The choice of a split Borel pair $(T,B)$ over an \'etale covering $S' \to S$ yields an isomorphism $(\underline{W}_G,\Iline) \times_S S' \cong (\Norm_{G_S'}(T)/T, I_{(B,T)})$ over $S'$ by \cite[XXII, 3.4]{SGA3III}, where $I_{(B,T)}$ is the set of simple reflections defined by $B$ in $\Norm_{G_S'}(T)/T$.
\end{remdef}

Attaching to $s \in S$ the based root datum $(\underline{X}^*(\sbar),\underline{\Phi}(\sbar),\underline{X}_*(\sbar),\underline{\Phi}\vdual(\sbar), \underline{\Delta}(\sbar))$, where $\sbar$ is some geometric point lying over $s$, defines a locally constant map from the underlying topological space of $S$ to the set of isomorphism classes of based root data.

To simplify our notation, we will in the sequel always assume that this map is constant.
%This is for instance the case if $G$ is obtained via base change from a reductive group scheme over a connected base scheme.
We call the isomorphism class of $(\underline{X}^*(\sbar),\underline{\Phi}(\sbar),\underline{X}_*(\sbar),\underline{\Phi}\vdual(\sbar), \underline{\Delta}(\sbar))$ \emph{the geometric based root datum} of $G$ and usually denote it simply by $(X^*,\Phi,X_*,\Phi\vdual,\Delta)$. Its Weyl group is given by $W_G := \underline{W}_G(\sbar)$, and it is called \emph{the geometric Weyl group}.

\begin{void}
Every cocharacter $\GG_m \to G$ of the reductive group scheme $G$, possibly defined over some \'etale covering of $S$, factorizes locally for the \'etale topology through some maximal torus $T$. Hence we obtain for the sheaf of conjugacy classes of cocharacters an identification
\[
\Ccal_G = \underline{W}_G\backslash \Xline_*,
\]
where we denote by $\underline{W}_G\backslash \Xline_*$ the \'etale sheaf quotient (or, equivalently, the fppf-sheaf quotient) of the action of the finite group scheme $\Wline_G$ on the scheme $\Xline_*$. Since $\underline{W}_G$ and $\Xline_*$ are both locally for the \'etale topology constant, so is $\Ccal_G$. In particular, it is represented by an \'etale and ind-quasi-affine algebraic space over $S$.

Attaching to $s \in S$ the set $\Ccal_G(\sbar) = W_G(T)(\sbar)\backslash X_*(T_{\sbar})$, where $\sbar \to S$ is a geometric point lying over $s$, defines a locally constant map on $S$ which is indeed constant by our above assumption. Its value is $W_G\backslash X_*$.
\end{void}

\begin{definition}\label{DefTypeGBundleBGGm}
Now let $\Escr$ be a $G$-bundle on $\B{\GG_{m,S}}$. Then its isomorphism class defines a locally constant map
\begin{equation}\label{EqDefTypeGBundle}
\begin{aligned}
\type(\Escr)\colon S &\lto W_G\backslash X_*, \\
s &\lsends \Escr_{\sbar} \in \Bunline_G(\B{\GG_{m,S}})(\sbar) = \Ccal_G(\sbar) = W_G\backslash X_*,
\end{aligned}
\end{equation}
called the \emph{type of $\Escr$}.
\end{definition}

Let $S'\to S$ be an \'etale covering such that $\Ccal_G \times_S S'$ is constant, for instance an \'etale covering that splits $G$. Then the type morphism yields a decomposition of the stack of types of $G$-bundles into open and closed substacks
\begin{equation}\label{EqDecomposeBunGBGGm}
\Bunline_G(\B{\GG_{m,S}})_{S'} = \coprod_{[\mu] \in W_G\backslash X_*}\Bunline^{[\mu]}_G(\B{\GG_{m,S'}}).
\end{equation}
By Proposition~\ref{BunlineMuStack}, $\Bunline^{[\mu]}_G(\B{\GG_{m,S'}})$ is a smooth gerbe over $S'$ which is \'etale locally isomorphic to $\B{\Cent_G(\mu)}$ where $\mu$ is some cocharacter of $G$ representing $[\mu] \in W_G\backslash X_*$.

\begin{remark}\label{SplittingNormal}
Let $S = \Spec O$ for a normal local ring $O$ with field of fractions $K$. Then any reductive group $G$ over $O$ is split over a finite \'etale extension $O'$ which we may assume to be the integral closure of $O$ in an $O$-unramified extension $K'$ of $K$. Let $K$ be a separable closure of $K$ and let $K^{\textup{ur}}$ be the maximal $O$-unramified extension of $K$ in $K^{\sep}$. If $\bar\eta$ denotes the geometric point of $\Spec O$ given by $K^{\sep}$, then we find $\Gal(K^{\textup{ur}}/K) \cong \pi_1(\Spec O, \bar\eta)$.

Hence every conjugacy class $[\mu]$ is defined over the integral closure $O^{\textup{ur}}$ of $O$ in $K^{\textup{ur}}$ and has a ring of definition which is a finite \'etale $O$-algebra contained in $O^{\textup{ur}}$. we call this ring of definition $O_{[\mu]}$. Then $\Bunline^{[\mu]}_G(\B{\GG_{m}})$ is defined over $O_{[\mu]}$.

In general, there exists no cocharacter of $G$ representing $\mu$ over $O_{[\mu]}$. But after passage to a finite \'etale extension $O'$ of $O_{[\mu]}$, $[\mu]$ can be represented by a cocharacter $\mu$ of $G_{O'}$. Then $\Bunline^{[\mu]}_G(\B{\GG_{m}})_{O'} \cong \B{\Cent_{G_{O'}}(\mu)}$. If $G_{O_{[\mu]}}$ is quasi-split (e.g., $O$ is henselian with finite residue field), one can choose $O' = O_{[\mu]}$. 
\end{remark}

\begin{example}\label{GLnBundleBGGm}
Consider the special case $G = \GL_n$. Then we can identify
\[
W_G\backslash X_* = \ZZ^n_{+} := \set{(\mu_1,\dots,\mu_n) \in \ZZ^n}{\mu_1 \geq \cdots \geq \mu_n}.
\]
Let $S$ be a scheme. A $\GL_n$-bundle over $\B{\GG_{m,S}}$ is the same as a $\ZZ$-graded vector bundle $\Escr = \bigoplus_{i\in \ZZ}\Escr_i$ of rank $n$ on $S$. Then its type is constant if and only if $\rk(\Escr_i)$ is constant for all $i$. In this case, $\Escr$ has type $\mu = (\mu_1,\dots,\mu_n)$ if
\[
\rk(\Escr_i) = n_i := \#\set{j \in \{1,\dots,n\}}{\mu_j = i}.
\]
Then only finitely many of the $n_i$ are nonzero and $n = \sum_{i\in\ZZ}n_i$. Since $G$ is a split reductive group, every conjugacy class of cocharacters is already defined over $S$ and even has a representative over $S$. We have
\[
\Bun_G^{\mu}(\B{\GG_{m,S}}) \cong \B{\prod_{i\in \ZZ}\GL_{n_i,S}}.
\]
\end{example}

\begin{example}\label{TorusBundleBGGm}
Let $\kappa$ be a field and let $T$ be a torus over $\kappa$. Suppose for simplicity that there exists a quadratic Galois extension $\kappa'/\kappa$ that splits $T$ and set $\Gamma := \Gal(\kappa'/\kappa) = \{\id,\sigma\}$. Then the locally constant group scheme $\Xline_*(T)$ becomes constant after base change to $\kappa'$, and $\Xline_*(T)$ is given by the free abelian group $X_*(T)$ with its $\Gamma$-action. Since $T$ is abelian, $W_T = 1$ and the action of $T$ on $\Xline_*(T)$ is trivial.

For a cocharacter $\mu \in X_*(T)$ there are now two cases.
\begin{assertionlist}
\item
If $\sigma(\mu) = \mu$, then $\mu$ defines a $\kappa$-rational point of $\Xline_*(T)$ and $\Bun_G^{\mu}(\B{\GG_{m,\kappa}}) = \B{T}$.
\item
If $\sigma(\mu) \ne \mu$, then a field of definition of $\mu$ is $\kappa'$ and $\Bun_G^{\mu}(\B{\GG_{m,\kappa'}}) \cong \B{T \otimes_k k'}$.
\end{assertionlist}
\end{example}

%Note that $\Autline(G) \to \Autline(\B{G})$ induces an isomorphism $[G \backslash \Autline(G)] \iso \Autline(\B{G})$.

%---------------------------------------------------------------------

\subsection{The display group attached to a functor of graded $R$-algebras}\label{EGENERAL}

In this section, we axiomatize and generalize results of Lau from \cite[5,6]{Lau_HigherFrames}. In particular, we define a more general version of the so-called display group of loc.~cit., which we also call the display group even though we do not use the notion of a display. Although the results in this section are much more general than in loc.~cit., most proofs are heavily inspired by Lau's proofs. 

We work in the following set-up. Let $O \to \kappa$ be a surjective map of rings. Let $R \sends \Bscr_R$ be a functor from the category of $\kappa$-algebras to the category of quasi-coherent algebras over $\B{\GG_{m,O}}$. We set $\Xcal_R := \Spec \Bscr_R$ and denote by $\Xcal$ the functor $R \sends \Xcal_R$ from the category of $\kappa$-algebras to the category of relatively affine schemes over $\B{\GG_{m,O}}$. For every $\kappa$-algebra $R$ one has a cartesian diagram
\[\xymatrix{
X_R \ar[r] \ar[d] & \Xcal_R \ar[d] \\
\Spec O \ar[r] & \B{\GG_{m,O}},
}\]
where $X_R = \Spec B_R$ is the spectrum of the $\ZZ$-graded $O$-algebra $B_R$ corresponding to $\Bscr_R$, and $\Xcal_R = [\GG_{m}\backslash X_R]$.

Let $G$ be a smooth affine group scheme over $O$, let $[\mu]$ be $G$-conjugacy class of coharacters of $\GG_m$ defined over $O$, let $\Escr_{[\mu]}$ be the corresponding $G$-bundle oover $\B{\GG_{m,O}}$ and let $[\mu]\backslash G$ the associated algebraic group quotient stack over $\B{\GG_{m,O}}$. Ty simplify the exposition, we assume that $[\mu]$ can be represented by a cocharacter $\mu$ of $G$ over $O$.
%\begin{equation}\label{EqMuAction}
%\mu\colon \GG_{m,O} \times G \to G
%\end{equation}
%be an action of $\GG_m$ by group automorphisms. As before (Remark~\ref{MuGGeneral}), we denote the algebraic group quotient stack over $\B{\GG_{m,O}}$ by $\mu\backslash G$.

For every $\kappa$-algebra $R$ we set
\begin{equation}\label{EqDefEXcalMu}
E_{\Xcal}(G,\mu)(R) := E_{\Xcal}(G,[\mu])(R) := \Hom_{\B{\GG_{m,O}}}(\Xcal_R, [\mu]\backslash G).
\end{equation}
Then $E_{\Xcal}(G,[\mu])(R)$ is a group (and not a group objects in groupoids) since we can identify the underlying groupoids via our choice of $\mu$ as follows
\begin{equation}\label{EqAltDescribeEMu}
\begin{aligned}
E_{\Xcal}(G,[\mu])(R) &= \Hom^{\GG_m}(X_R,G) := \{\text{$\GG_m$-equivariant morphisms $X_R \to G$}\} \\
&= \{\text{degree $0$ maps $\Gamma(G,\Oscr_G) \to B_R$ of $\ZZ$-graded $O$-algebras}\}
\end{aligned}
\end{equation}
by Remark~\ref{AffineMapstoGGmQuotients}, and the right-hand side is a set. Hence we obtain a presheaf of groups on the category of $\kappa$-algebras.

\begin{definition}\label{DefDisplayGroup}
The presheaf of groups on $\kappa$-algebras $E_{\Xcal}(G,[\mu])$ is called the \emph{$(G,[\mu])$-display group of $\Xcal$}.
\end{definition}

In all cases of interest for us, we will be able to apply Proposition~\ref{EMuRepresentable} below to show that $E_{\Xcal}(G,[\mu])$ is representable by an affine group scheme.

\begin{remark}\label{FunctorialityDisplayGroupInGMu}
Let $G'$ be a second smooth affine group scheme over $O$, let $\varphi\colon G \to G'$ be a map of group schemes and let $[\mu'] = [\varphi \circ \mu]$ be the $G'$-conjugacy class of cocharacters of $G'$ obtained as the image of $[\mu]$ under $\varphi$. Then $\varphi$ induces a map of algebraic stacks $[\mu]\backslash G \to [\mu']\backslash G'$ and hence by composition a map of presheaves of groups
\begin{equation}\label{EqFunctorialInGDisplayGroup}
E_{\Xcal}(G,[\mu]) \lto E_{\Xcal}(G',[\mu']).
\end{equation}
If $G \to G'$ is a monomorphism, the description~\eqref{EqAltDescribeEMu} shows that \eqref{EqFunctorialInGDisplayGroup} is a monomorphism.
\end{remark}

Next we will give an alternative description of $E_{\Xcal}(G,[\mu])$.

\begin{remdef}\label{StandardGMuBundle}
Let $\Xcal_{\kappa} \to \B{\GG_{m,O}}$ be the structure map and denote by $\Escr_{\Xcal}(G,[\mu])$ the pullback of $\Escr_{[\mu]}$ (Remark~\ref{MuGGeneral}) to $\Xcal_{\kappa}$. It is called the \emph{standard $G$-bundle of type $[\mu]$ for $\Xcal$}.

%We define a $G$-bundle $\Escr_{\Xcal}(G,\mu)$ on $\Xcal_{\kappa}$, , as follows. To define a $G$-bundle over $\Xcal_{\kappa}$ is the same as to define $\GG_m$-equivariant $G$-bundle $\tilde\Escr(G,\mu)$ over $X_{\kappa}$. We define $\tilde\Escr(G,\mu)$ as follows. The underlying $G$-bundle of $\tilde\Escr$ is the trivial $G$-bundle $G$ and the $\GG_m$-equivariant structure is given by the $\GG_m$-action \eqref{EqMuAction}. In other words, $\Escr(G,\mu)$ is the pullback of the trivial $[\mu\backslash G]$-bundle under $\Xcal_{\kappa} \to \B{\GG_{m,O}}$.

We denote by $\Autline(\Escr_{\Xcal}(G,[\mu]))$ the functor that sends a $\kappa$-algebra $R$ to the group of automorphisms of the $G$-bundle $\Escr_{\Xcal}(G,[\mu])_R$ over $\Xcal_R$, where $\Escr_{\Xcal}(G,[\mu])_R$ denotes the pullback of $\Escr_{\Xcal}(G,[\mu])$ under $\Xcal_R \to \Xcal_{\kappa}$. 
\end{remdef}

Since an automorphism of $\Escr(G,[\mu])_R$ is given by a point $g \in G(X_R) = \Hom(X_R,G)$ that is $\GG_m$-equivariant, we obtain the following alternative description of the display group.

\begin{proposition}\label{AutDecribeDisplayGroup}
There is a functorial isomorphism of presheaves of groups on the category of $\kappa$-algebras
\[
E_{\Xcal}(G,[\mu]) \iso \Autline(\Escr_{\Xcal}(G,[\mu])).
\]
\end{proposition}

\begin{example}\label{BunGMuAA1GGm}
Let $\mu\colon \GG_{m,O} \to G$ be a cocharacter.
\begin{assertionlist}
\item
Let $\Xcal_R = \B{\GG_{m,R}}$, viewed as a stack over $\B{\GG_{m,O}}$ via $O \to \kappa \to R$. Then $B_R = R$ with the trivial $\ZZ$-grading and we find
\[
E_{\Xcal}(\mu)(R) = \Hom^{\GG_m}(*,G) = \Cent_G(\mu)(R)
\]
Hence $E_{\Xcal}(\mu) = \Cent_G(\mu)$.
\item
Let $\Xcal_R = [\GG_{m,R}\backslash (\AA^1_R)^{\pm}]$, i.e., $B_R = R[t]$ endowed with the $\ZZ$-grading such that $\deg(t) = \pm1$. By \eqref{EqAltDescribeEMu} we find
\[
E_{\Xcal}(\mu)(R) = \Hom^{\GG_m}((\AA_R^1)^{\pm},G) = P^{\pm}(\mu)(R).
\]
Hence $E_{\Xcal}(\mu) = P^{\pm}(\mu)$.
\end{assertionlist}
\end{example}

\begin{example}\label{ETrivialGGmAction}
Suppose $\mu$ is the trivial cocharacter. Then
\[
E_{\Xcal}(G,\mu)(R) = \Hom_{\B{\GG_{m,O}}}(\Xcal_R, \B{\GG_{m,O}} \times G) = \Hom_{O}(\Xcal_R,G) = G((B_R)_0),
\]
where the last equality holds since $\Spec (B_R)_0$ is a good moduli space for $\Xcal_R$ and therefore $\Xcal_R \to \Spec (B_R)_0$ is initial in the category of morphisms from $\Xcal_R$ to algebraic spaces.
\end{example}

In Example~\ref{BunGMuAA1GGm}, $E_{\Xcal}(\mu)$ is representable by a group scheme. This fact is a special case of the following general principle.

\begin{proposition}\label{EMuRepresentable}
In the situation introduced in the beginning of this section
%let $\tau$ be a Grothendieck topology on the category $\kappa$-algebras. 
suppose that $R \sends (B_R)_i$ is representable for all $i \in \ZZ$ by a scheme
%(resp.~is a $\tau$-sheaf)
on the category of $\kappa$-algebras.

Let $Y = \Spec C$ be an affine $O$-scheme of finite type with $\GG_{m,O}$-action and set $\Ycal = [\GG_{m,O}\backslash Y]$. Then the set valued functor on $\kappa$-algebras
\begin{equation}\label{EqDefHomBGGmXY}
E(\Ycal)\colon R \sends \Hom_{\B{\GG_{m,O}}}(\Xcal_R, \Ycal)
\end{equation}
is representable by a scheme.
% (resp.~a $\tau$-sheaf).
\begin{assertionlist}
\item\label{EMuRepresentable1}
If $R \sends (B_R)_i$ is representable for all $i \in \ZZ$ by an affine scheme (resp.~by an affine $\kappa$-scheme of finite type), then \eqref{EqDefHomBGGmXY} is representable by an affine scheme (resp.~by an affine scheme of finite type).
\item\label{EMuRepresentable2}
Suppose that the following conditions are satisfied.
\begin{definitionlist}
\item\label{EMuRepresentablea}
The functor $R \sends (B_R)_i$ is representable for all $i \in \ZZ$ by a smooth affine scheme $Z_i$ of finite type.
\item
$Y$ is smooth over $O$.
\item
If $S \to R$ is a surjective map of $\kappa$-algebras with square zero kernel, then $B_S \to B_R$ has nilpotent kernel (note that $B_S \to B_R$ is surjective by \ref{EMuRepresentablea}).
\end{definitionlist}
Then \eqref{EqDefHomBGGmXY} is representable by a smooth affine scheme.
\end{assertionlist}
\end{proposition}

This proposition axiomatizes and generalizes \cite[5.3.1, 5.3.8]{Lau_HigherFrames} but the proof is essentially the same as in loc.~cit.

\begin{proof}
First assume that $Y = (\AA^1_O)^{(i)}$. Then
\[
E(\Ycal)(R) = \Hom^{\GG_m}(X_R,Y) = (B_R)_i
\]
and all assertions hold by assumption. Therefore the same holds also if $Y$ is a  product $\AA^n_O$, where $\GG_m$ acts on every coordinate with possibly different weights. In general, $Y$ is the equalizer of a pair of maps $\AA^n_O \rightrightarrows \AA^m_O$ where each coordinate is homegeneous of some degree. This shows all assertions except the last one about smoothness of $E(\Ycal)$ which is proved as in the second half of the proof of \cite[5.3.8]{Lau_HigherFrames}.
\end{proof}

\begin{example}\label{LaurentE}
Let $R \sends A_R$ be a functor from the category of $\kappa$-algebras to the category of $O$-algebras and define $\Xcal_R := \Spec A_R$ with structure map $\Xcal_R \to \Spec O \to \B{\GG_{m,O}}$. By \eqref{EqAltDescribeEMu} one has for every $\kappa$-algebra $R$
\[
E_{\Xcal}(G,\mu)(R) = \Hom^{\GG_{m,A_R}}(\GG_{m,A_R}, G_{A_R}) = G(A_R),
\]
where the last identity is induced by evaluation in $1$.

If $R \sends A_R$ is representable by an affine scheme, then $E_{\Xcal}(\mu)$ is representable by an affine group scheme by Proposition~\ref{EMuRepresentable}.
\end{example}

\begin{remark}\label{EFunctorialInX}
Let $R \sends \Xcal_R$ and $R \sends \tilde\Xcal_R$ be two functors as above and let $\alpha_R\colon \Xcal_R \to \tilde\Xcal_R$ be maps of stacks over $\B{\GG_{m,O}}$ functorial in $R$. By composition, the $\alpha_R$ induce a map of presheaves of groups
\[
E_{\alpha}\colon E_{\tilde\Xcal}(G,[\mu]) \lto E_{\Xcal}(G,[\mu]).
\]
\end{remark}

\begin{proposition}\label{EFunctorialSurjective}
With the notation of Remark~\ref{EFunctorialInX}, suppose that $\alpha_R$ is a closed immersion and that $((\Btilde_R)_0,\Ker((\alpha^*_R)_0))$ is a henselian pair for every $\kappa$-algebra $R$. Then the induced map
\[
E_{\alpha}\colon E_{\tilde\Xcal}(G,[\mu])(R) \lto E_{\Xcal}(G,[\mu])(R)
\]
is surjective for all $R$.
\end{proposition}

\begin{proof}
The adequate moduli spaces (in this case even good moduli spaces) of $\Xcal_R$ and $\tilde\Xcal_R$ (in the sense of \cite[5.1.1]{Alper_Adequate}) are $\Spec (B_R)_0$ and $\Spec (\Btilde_R)_0$, respectively, by \cite[9.1.4]{Alper_Adequate}. The hypothesis implies that the pair of algebraic stacks $(\tilde\Xcal_R,\Xcal_R)$ is henselian by \cite[3.6]{AHR_EtaleLocal} and $[\mu \backslash G]$ is smooth. Hence we conclude by \cite[7.9]{AHR_EtaleLocal}.
\end{proof}

\begin{example}\label{RepellingE}
Suppose that $\alpha$ in Remark~\ref{EFunctorialInX} is an isomorphism in degree $\leq 0$ and that the $\GG_m$-action on $G$ is repelling (i.e., the $\ZZ$-grading on $\Gamma(G,\Oscr_G)$ is concentrated in degree $\leq 0$). Then by \eqref{EqAltDescribeEMu}, $E_{\alpha}$ induces an isomorphism
\[
E_{\tilde\Xcal}(G,[\mu]) \liso E_{\Xcal}(G,[\mu]).
\]
\end{example}

%---------------------------------------------------------------------

\subsection{Display gerbes}\label{Sec:DisplayGerbes}

We now change our setting from Section~\ref{EGENERAL} slightly. Again we are given a surjective ring homomorphism $O \to \kappa$. Suppose that we are given a functor $R \sends (\Xcal_R,\alpha_R)$ from the category of $\kappa$-algebras to the category of pairs $(\Xcal,\alpha)$ consisting of an algebraic stack $\Xcal$ over $\B{\GG_{m,O}}$ such that $\Xcal \to \B{\GG_{m,O}}$ is affine. Let $B$ be the corresponding $\ZZ$-graded $O$-algebra (Section~\ref{Sec:DescribeGGmQuotients}) and let $\alpha$ be an isomorphism of $O$-algebras $\alpha\colon B_0/\sum_{i>0}B_iB_{-i} \iso R$.

As before, let $G$ be a smooth affine group scheme over $O$ and let $[\mu]$ be a conjugacy class of cocharacters of $G$ defined over $O$. We define a stack $\Bunline_G(\Xcal)$ over $\kappa$ as the fppf-stackification of the prestack given by
\[
R \sends \Bun_G(\Xcal_R), \qquad\qquad \text{$R$ a $\kappa$-algebra}.
\]

% $\mu\colon \GG_{m,O} \to G$ is a cocharacter of $G$. We obtain a $\GG_{m,O}$-action on $G$ given by conjugation with $\mu$ via
%\begin{equation}\label{EqMuAction2}
%(t,g) \sends \mu(t)g\mu(t)^{-1}.
%\end{equation}
Via $\alpha_R$ we can identify $\B{\GG_{m,R}} = \Xcal_R^0$ and pullback via the closed immersion $\B{\GG_{m,R}} \lto \Xcal_R$ yields a morphism of stacks
\begin{equation}\label{EqTypeDisplayMap}
t\colon \Bunline_G(\Xcal) \to \Bunline_G(\B{\GG_m}).
\end{equation}
The preimage of $\Bunline_G^{[\mu]}(\B{\GG_m})$ under $t$ is denoted by $\Bunline^{[\mu]}_G(\Xcal)$.

%The image of $\mu$ in $[G\backslash \Xline_*(G)](O)$ defines a type and we can therefore speak of the groupoid $\Bun_G^{\mu}(\Xcal_R)$ of $G$-bundles of type $\mu$ on $\Xcal_R$. We define a stack $\Bunline_G^{\mu}(\Xcal)$ over $\kappa$ as the fppf-stackification of the prestack given by
%\[
%R \sends \Bun_G^{\mu}(\Xcal_R), \qquad\qquad \text{$R$ a $\kappa$-algebra}.
%\]

\begin{theorem}\label{GeneralBunlineGMu}
Suppose that the $\ZZ$-graded algebra $B_R$ is henselian (Definition~\ref{DefHenselianGradedRing}) for all $R$. Then one has $\Bunline_G^{[\mu]}(\Xcal) \cong \B{E_{\Xcal}(G,[\mu])}$.
\end{theorem}

\begin{proof}
The restriction of \eqref{EqTypeDisplayMap} to a map $\Bunline^{[\mu]}_G(\Xcal) \to \Bunline^{[\mu]}_G(\B{\GG_m})$ is a gerbe for the fppf-topology because $B_R$ is henselian for all $R$ by Theorem~\ref{LiftGBundles}. By Proposition~\ref{BunlineMuStack}, $\Bunline_G^{[\mu]}(\B{\GG_m})$ is a gerbe over $\kappa$. Therefore $\Bunline_G^{[\mu]}(\Xcal)$ is a gerbe over $\kappa$. Hence it suffices to find a $G$-bundle $\Escr$ of type $[\mu]$ on $\Xcal_{\kappa}$ such that $\Autline(\Escr) \cong E_{\Xcal}([\mu])$. For this we may choose the standard $G$-bundle of type $\mu$ over $\Xcal_{\kappa}$ by Proposition~\ref{AutDecribeDisplayGroup}.
%Here $\Autline(\Escr)$ denotes the functor that sends a $\kappa$-algebra $R$ to the group of automorphisms of the $G$-bundle $\Escr_R$ over $\Xcal_R$, where $\Escr_R$ denotes the pullback of $\Escr$ under $\Xcal_R \to \Xcal_{\kappa}$.
%
%To define a $G$-bundle over $\Xcal_{\kappa}$ is the same as to define $\GG_m$-equivariant $G$-bundle $\tilde\Escr$ over $X_{\kappa}$. We define $\tilde\Escr$ as follows. The underlying $G$-bundle of $\tilde\Escr$ is the trivial $G$-bundle $G$ and the $\GG_m$-equivariant structure is given by the $\GG_m$-action \eqref{EqMuAction}. Then an automorphism of $\Escr_R$ is given by a point $g \in G(X_R) = \Hom(X_R,G)$ that is $\GG_m$-equivariant.
\end{proof}

%=====================================================================

\section{Frames and frame stacks}

\subsection{Frames}

\begin{definition}\label{DefFrame}
Let $A$ be a ring and let $\Bscr$ be a filtered $A$-algebra, given by $(B_{j+1} \ltoover{t} B_j)$ (Definition~\ref{DefFilteredAlg}), let $B_{-\infty}$ be its underlying ring, and let $\Re(B)$ be the associated Rees stack.
\begin{assertionlist}
\item
Let $\sigma$ be a morphism
\[
\sigma\colon \Spec B_{-\infty} \lto \Re(B).
\]
We will call such a pair $((B_{j+1} \ltoover{t} B_j), \sigma)$ a \emph{quasi-frame}. The composition
\[
\Spec B_0 \to \Spec B_{-\infty} \ltoover{\sigma} \Re(B) \lto \Spec B_0
\]
corresponds to a ring endomorphism of $B_0$, called the \emph{endomorphism of $B_0$ induced by $\sigma$}, which we always denote by $\phi$.
%
%If $B_0$ is endowed with the structure of an $O$-algebra for a ring $O$, then we say that $((B_{j+1} \ltoover{t} B_j), \sigma)$ is a \emph{quasi-frame over $O$}.
\item
A quasi-frame $((B_{j+1} \ltoover{t} B_j), \sigma)$ is called \emph{frame} if the following conditions are satisfied.
\begin{definitionlist}
\item
The generalized filtered $A$-algebra $(B_{j+1} \ltoover{t} B_j)$ is a filtered ring $(A,(\Fil^i_A))$.
\item
The map $\sigma\colon \Spec B_{\infty} = \Spec A \to \Re(A,(\Fil^i_A))$ factors through the open substack $\Re(A, (\Fil^i_A))^{\ne -}$.
\end{definitionlist}
\end{assertionlist}
\end{definition}

Let us make more concrete how to define a map $\sigma\colon \Spec A \to \Re(A)^{\ne-}$ for a filtered ring as above.

\begin{lemma}\label{ConstructFrame}
Let $(A,(\Fil^i))$ be a filtered ring and let $\phi\colon A \to A$ be a ring endomorphism. Then there is a bijection between the set of maps $\sigma\colon \Spec A \lto \Re(A)^{\ne-}$ inducing $\phi$ and pairs $(L,\sigma^*)$, where $L \in \Pic(A)$ and where $\sigma^*$ is a map of graded $A$-algebras
\[
\sigma^*\colon \phi^*\Rees(A) \to T(L) = \bigoplus_{i\in \ZZ}L^{\otimes i}
\]
such the graded ideal generated by the image of
\[
I^+ \otimes_{A,\phi} A \lto \phi^*\Rees(A) \ltoover{\sigma^*} T(L)\tag{*}
\]
contains $T(L)_i = L^{\otimes i}$ for some $i$.
\end{lemma}

\begin{proof}
A map $\sigma$ defining a frame and inducing $\phi$ is the same as a section of $\Re(A)^{\ne-} \otimes_{A,\phi} A$ over $\Spec A$. A section of $\Re(A) \otimes_{A,\phi} A$ is given by a map of graded $A$-algebras $\sigma^*\colon \phi^*\Rees(A) \to T(L) := \bigoplus_{i\in \ZZ}L^{\otimes i}$, where $L$ is a line bundle on $A$ (Remark~\ref{MapsToQuotientbyGGm}). The section factors through $\Re(A)^{\ne-} \otimes_{A,\phi} A$ if and only if the graded ideal generated by the image of
\[
I^+ \otimes_{A,\phi} A \lto \phi^*\Rees(A) \ltoover{\sigma^*} T(L)\tag{*}
\]
is all of $T(L)$. Indeed, more generally, if $\varphi\colon A \to B$ is a map of rings, $f\colon \Spec B \to \Spec A$ the corresponding map of affine schemes, and $I \subseteq A$ an ideal, then $f$ factors through the open subscheme $\Spec A \setminus V(I)$ if and only if $B = \varphi(I)B$.

Now if $J \subseteq T(L)$ is any graded ideal such that $J_i = T(L)_i$, then $J = T(L)$. Indeed, to see that the inclusion $J \to T(L)$ is surjective, we can work locally in $\Spec A$. Hence we can assume that $L = A$ and hence $T(L) = A[t,t^{-1}]$. Then multiplying by integral powers of $t$ shows that $J = A[t,t^{-1}]$. This shows that the graded ideal generated by the image of (*) is $T(L)$ if and only if it contains $T(L)_i$ for some $i$.
\end{proof}

\begin{defrem}\label{CategoryFrames}
Let $A$ be a ring. A \emph{morphism of quasi-frames over $A$} $((B_{j+1} \ltoover{t} B_j),\sigma_B) \to ((C_{j+1} \ltoover{t} C_j),\sigma_C)$ is a map $\psi\colon (B_{j+1} \ltoover{t} B_j) \to (C_{j+1} \ltoover{t} C_j)$ of filtered $A$-algebras (Definition~\ref{DefMapFilteredRings}) such that the following diagram commutes
\[\xymatrix{
\Spec B_{-\infty} \ar[r]^{\sigma_A} \ar[d] & \Re(B) \ar[d] \\
\Spec C_{-\infty} \ar[r]^{\sigma_B} & \Re(C),
}\]
where the vertical maps are induced by $\psi$. We obtain the category of quasi-frames and the full subcategory of frames.
\end{defrem}

\begin{remark}\label{FrameLiterature}
Our definition of frame generlizes the notion of frame defined by Lau in \cite{Lau_HigherFrames}.
%In \cite{GMM_TruncPrismGMuDis} an additional invertible $A$-module $J$ and an isomorphism $\phi^*J \otimes_A L \iso J$\footnote{Typo in loc.~cit.?} is fixed. This yields a map $\eta\colon \Re(A) \to \B{\GG_{m,A}}$, corresponding to the line bundle $\Oscr_{\Re(A)}(1) \otimes_A J$, such that $\eta \circ \tau = \eta \circ \sigma$ (note that $\tau^*\Oscr_{\Re(A)}(e) = A$ and $\sigma^*\Oscr_{\Re(A)}(e) = L^{\otimes e}$ by definition).
\end{remark}

\begin{remdef}\label{ExtensionOfFrames}
Let $(C,\phi_C)$ be a ring together with an endomorphism. Let $(A, (\Fil_A^i), \sigma_A)$ be a frame yielding a ring endomorphism $\phi_A$ of $A$, and let $\psi\colon (A,\phi_A) \to (C,\phi_C)$ be a ring homomorphism with $\psi \circ \phi_A = \phi_C \circ \psi$. Set $\Fil^i_C = \Fil^i_A \otimes_A C$. By Remark~\ref{ExtensionFilteredRing} we have $\Re(C)^{\ne -} = C \otimes_A \Re(A)^{\ne -}$ and hence we obtain a 2-cartesian diagram
\[\xymatrix{
C \otimes_{\phi_C,C} \Re(C)^{\ne -} \ar[r] \ar[d] & A \otimes_{\phi_A,A} \Re(A)^{\ne -}  \ar[d] \\
\Spec C \ar[r] & \Spec A.
}\]
Therefore the section $\sigma_A$ of $A \otimes_{\phi_A,A} \Re(A)^{\ne -} \to \Spec A$ induces by base change along $\psi$ a map $\sigma_C$ that makes $(C, (\Fil^i_C), \sigma_C)$ into a frame. Moreover $\psi$ defines then a morphism of frames $(C, (\Fil^i_C), \sigma_C) \to (A, (\Fil_A^i), \sigma_A)$.

In particular, we can endow this construction if $C = A/I$ for a $\phi_A$-invariant ideal. In his case, we call $(C, (\Fil^i_C), \sigma_C)$ the \emph{quotient frame} of $(A, (\Fil_A^i), \sigma_A)$ by $I$.
\end{remdef}

%If $z$ is regular in $A$, then Example~\ref{ZAdicReesFrobenius} is a special case of the following example.
%
%
%\begin{example}\label{DivisorFrame}
%Let $A$ be a ring amd let $I \subseteq A$ be an invertible ideal. Endow $A$ with the $I$-adic filtration and let $\Rees(A,I)$ be the corresponding Rees algebra.
%\end{example}

%---------------------------------------------------------------------

\subsection{The frame stack}

Let $(A,(\Fil^i)_i,\sigma)$ be a frame. By \eqref{EqReesStackNePlus} we have the open quasi-compact immersion
\begin{equation}\label{EqTau}
\tau\colon \Spec A \liso \Re(A)^{\ne+} \lto \Re(A)
\end{equation}
identifying $\Spec A$ with the locus of $\Re(A)$ where $t$ is invertible.

\begin{definition}\label{DefFrameStack}
The coequalizer
\[
\Fcal(A) := \Fcal(A, (\Fil^i)_i, \sigma) := \colim (\xymatrix{\Spec A
\ar@<.5ex>[r]^-(.4){\tau}\ar@<-.5ex>[r]_-(.4){\sigma} & \Re(A)}).
\]
is called the \emph{frame stack} of the frame. Since both $\sigma$ and $\tau$ factorize both over $\Re(A)^{\ne0} = \Re(A)^{\ne+} \cup \Re(A)^{\ne-}$, we can also define
\[
\Fcal(A)^{\ne0} := \Fcal(A, (\Fil^i)_i, \sigma)^{\ne0} := \colim (\xymatrix{\Spec A
\ar@<.5ex>[r]^-(.4){\tau}\ar@<-.5ex>[r]_-(.4){\sigma} & \Re(A)^{\ne0}}).
\]
In both cases we take the coequalizer in the 2-category of Adams stacks, see Appendix~\ref{APPADAMS}. Note that $\Re(A)$ and $\Re(A)^{\ne0}$ are both Adams stacks by Proposition~\ref{QuotientStackAdam}.
\end{definition}

\begin{remark}\label{VecFrameStack}
Let $G$ be a group scheme over $A$ such that $\B{G}$ is an Adams stack. It follows that the category of vector bundles (resp.~the groupoid of $G$-bundles) on $\Fcal(A)$ is equivalent to the category (resp.~groupoid) of pairs $(\Escr,\Phi)$ consisting of a vector bundle (resp.~$G$-bundle) $\Escr$ on $\Re(A)$ and an isomorphism of vector bundles (resp.~$G$-bundles) $\Phi\colon \sigma^*\Escr \to \tau^*\Escr$ over $\Spec A$.

A similar description holds for vector bundles and $G$-bundles over $\Fcal(A)^{\ne 0}$.
\end{remark}

\begin{remark}\label{DescribeCoequalizer}
With the notation above, we have a 2-commutative diagram of Adams stacks
\[\xymatrix{
\Spec A \amalg \Spec A \ar[r]^-{\tau \amalg \sigma} \ar[d]_{\textup{can}} & \Re(A)^{\ne0} \ar[d] \ar[r] & \Re(A) \ar[d] \\
\Spec A \ar[r] & \Fcal(A)^{\ne0} \ar[r] & \Fcal(A),
}\]
where the left square and the combined rectangle are pushout squares by definition. Therefore the right square is a pushout square.
\end{remark}

If $\Xcal$ is any Adams stack, then we obtain a 2-cartesian square of groupoids
\[\xymatrix{
\Hom(\Re(A)^{\ne0},\Xcal) & \Hom(\Re(A),\Xcal) \ar[l] \\
\Hom(\Fcal(A)^{\ne0},\Xcal) \ar[u] & \Hom(\Fcal(A),\Xcal). \ar[u] \ar[l]
}\]
In particular, if the restriction induces an equivalence (resp.~a fully faithful functor) $\Hom(\Re(A),\Xcal) \to \Hom(\Re(A)^{\ne0},\Xcal)$, then it follows formally that composition with $\Fcal(A)^{\ne0} \to \Fcal(A)$ also induces an equivalence (resp.~a fully faithful functor) $\Hom(\Fcal(A),\Xcal) \to \Hom(\Fcal(A)^{\ne0},\Xcal)$. In particular we have the following result.

\begin{lemma}\label{GBundlesFrameStack}
Suppose that $\Vec(\Re(A)) \to \Vec(\Re(A)^{\ne0})$ is an equivalence (resp.~fully faithful). Then $\Vec(\Fcal(A)) \to \Vec(\Fcal(A)^{\ne0})$ is an equivalence (resp.~fully faithful). Similarly, let $G$ be a group scheme over $A$ such that $\Bun_G(\Re(A)) \to \Bun_G(\Re(A)^{\ne0})$ is an equivalence (resp.~fully faithful). Then $\Bun_G(\Fcal(A)) \to \Bun_G(\Fcal(A)^{\ne0})$ is an equivalence (resp.~fully faithful).
\end{lemma}

%---------------------------------------------------------------------

\subsection{Classifying $G$-bundles on frame stacks}

We fix the following general setting. Let $O$ be a Dedekind domain and let $O \to \kappa$ be a surjective homomorphism of rings. Let $A\colon R \sends (A_R, (\Fil_{A_R}^i), \sigma_R, \alpha_R)$ be a presheaf on the category of $\kappa$-algebras with values in the category of tuples consisting of a frame over $O$ and an isomorphism $\alpha_R\colon A_R/t(\Fil_{A_R}^1) \iso R$ of $O$-algebras.

Let $\Re(A_R)$ be the Rees stack attached to the filtered ring $(A_R, (\Fil^i_{A_R}))$. Via the composition
\[
\Re(A_R) \to A_R^{\Fil} \to O^{\Fil} \to \B{\GG_{m,O}}
\]
we can consider $R \sends (\Re(A_R),\alpha_R)$ as a functor as described in the beginning of Section~\ref{Sec:DisplayGerbes}. In particular we obtain for every smooth affine group scheme $G$ and every conjugacy class $[\mu]$ of a cocharacter $\mu$ of $G$ the display group $E_{A}(G,[\mu])$ (Section~\ref{EGENERAL}). We denote by $\Bunline_G^{[\mu]}(A)$ the prestack $R \sends \Bun_G^{[\mu]}(\Re(A_R))$.

For every prestack $\Xcal$ over $O$ we define a prestack over $\kappa$ by
\[
L^A\Xcal(R) := \Xcal(A_R), \qquad\qquad \text{for every $\kappa$-algebra $R$}.
\]

We make the following assumptions.

\begin{assumption}\label{ConditionFrameFunctor}
\begin{definitionlist}
\item\label{ConditionFrameFunctora}
The functor $R \sends \Fil^i_{A_R}$ is representable for all $i \geq 0$ by an affine scheme.
\item\label{ConditionFrameFunctorb}
The pair $(A_R, t(\Fil^1_{A_R}))$ is henselian for every $\kappa$-algebra $R$.
\item\label{ConditionFrameFunctorc}
The canonical map of stacks $\B{(L^AG)} \to L^A(\B{G})$ is an equivalence.
\end{definitionlist}
Condition~\ref{ConditionFrameFunctora} ensures that the display group $E_{\Re(A)}(G,[\mu])$ is representable by an affine group scheme over $\kappa$ by Proposition~\ref{EMuRepresentable}. Using Hypothesis \ref{ConditionFrameFunctorb}, Theorem~\ref{GeneralBunlineGMu} yields an equivalence of stacks
\begin{equation}\label{EqDescribeBunGMuGenHecke}
\Bunline_G^{[\mu]}(A) \cong \B{E_{A}(G,[\mu])}.
\end{equation}
\end{assumption}

\begin{remark}\label{PropertiesLAG}
Let us look more closely at the functor of groups $L^AG$. As in the proof of Proposition~\ref{EMuRepresentable}, Condition~\ref{ConditionFrameFunctora} implies that $L^AG$ is representable by an affine group scheme, which is of finite type if $R \sends A_R$ is representable by a $\kappa$-scheme of finite type. Moreover, $L^AG$ is a smooth group scheme if $R \sends A_R$ is representable by a smooth affine scheme and $A_S \to A_R$ has nilpotent kernel for every surjective map of $\kappa$-algebras $S \to R$ with square zero kernel.
\end{remark}

Let $\Fcal(A_R)$ be the frame stack attached to the frame $(A_R, (\Fil_{A_R}^i), \sigma_R)$. Since $O$ is a Dedekind domain, $\B{G}$ is an Adams stack by Example~\ref{BGAdams} and hence we have by Remark~\ref{VecFrameStack} 
\begin{equation}\label{EqBunGFrame}
\Bun_G(\Fcal(A_R)) = \lim (\Bun_G(\Re(A_R)) \rightrightarrows \Bun_G(A_R)),
\end{equation}
where the two arrows are given via pullback with $\tau$ and $\sigma$. We set
\begin{equation}\label{EqBunGMuFrame}
\begin{aligned}
\Bun_G^{[\mu]}(\Fcal(A_R)) &:= \lim (\Bun^{[\mu]}_G(\Re(A_R)) \rightrightarrows \Bun_G(A_R)),\\
\Bunline_G^{[\mu]}(\Fcal(A))(R) &:= \Bun_G^{[\mu]}(\Fcal(A_R)).
\end{aligned}
\end{equation}
This defines a prestack $\Bunline_G^{[\mu]}(\Fcal(A))$ over $\kappa$.

We now come to our main general classification result. Consider the display group attached to the functor $R \sends A_R$ as in Example~\ref{LaurentE}. It is given by the affine group scheme $L^AG$. By Remark~\ref{EFunctorialInX}, the two functorial maps $\tau, \sigma\colon \Spec A_R \to \Re(A_R)$ induce maps of affine group schemes over $\kappa$
\begin{equation}\label{EqGenActionEonG}
E_{A}(G,[\mu]) \rightrightarrows L^AG,
\end{equation}
again denoted by $\tau$ and $\sigma$. This defines an action of the group scheme $E_A(G,[\mu])$ on the $\kappa$-scheme $L^AG$ and we can form the fpqc quotient $[E_A(G,[\mu])\backslash L^AG]$.

\begin{theorem}\label{DescribeBunlineFramStack}
There is a functorial equivalence of fpqc-stacks
\begin{equation}\label{EqDescribeGBundlesFrame}
\Bunline_G^{[\mu]}(\Fcal(A)) \cong [E_A(G,[\mu])\backslash L^AG].
\end{equation}
\end{theorem}

\begin{proof}
%The prestack $L^A\B{G}$ is the prestack on the category of $\kappa$-algebras $R \sends \Bun_G^{\mu}(A_R)$. Considering $A_R$ as a graded $O$-algebra concentrated in degree $0$
We have functorial identities
\begin{align*}
\Bunline_G^{[\mu]}(\Fcal(A)) &= \lim(\Bunline_G^{[\mu]}(A) \rightrightarrows L^{(N)}L^A(\B{G})) \\
&= \lim(\B{E_{A}(G,[\mu])} \rightrightarrows L^A(\B{G})) \\
&= \lim(\B{E_{A}(G,[\mu])} \rightrightarrows \B{L^AG}) \\
&= [E_A(G,[\mu]) \backslash L^A(G)].
\end{align*}
The first identity holds by definition, the second by \eqref{EqDescribeBunGMuGenHecke}, the third identity holds by Assumption~\ref{ConditionFrameFunctorc} above, and the last identity follows from Lemma~\ref{EqualizerClassifyingStack} below.

The quotient stack $[E_A(G,[\mu]) \backslash L^A(G)]$ has an affine diagonal and in particular is an fpqc-stack.
\end{proof}

In the proof of the theorem we have used the following general lemma.

\begin{lemma}\label{EqualizerClassifyingStack}
Let $\Ccal$ be a site, let $E$ and $H$ be sheaves of groups on $\Ccal$, and let $\tau,\sigma \colon E \to H$ be homomorphisms of sheaves of groups. Denote $\tgbar, \sgbar\colon \B{E} \to \B{H}$ be the induced maps on classifying spaces. Define a left action of $E$ one the underlying sheaf of sets of $H$ by
\[
E \times H \lto H, \qquad (e,h) \sends \tau(e)h\sigma(e^{-1}).
\]
Then one has a functorial equivalence of stacks
\[
\lim (\xymatrix{\B{E} \ar@<0.5ex>[r]^{\tgbar} \ar@<-0.5ex>[r]_{\sgbar} & \B{H}}) \cong [E \backslash H].
\]
\end{lemma}

\begin{proof}
It suffices to show the equivalence for the corresponding prestacks and then stackify, i.e., we may assume that $\Ccal$ is endowed with the trivial topology (where only isomorphisms define sieves). In this case this a straight forward calculation:

For $R \in \Ccal$, the objects of $\lim (\B{E} \rightrightarrows \B{H})(R)$ are given by pairs $(\eps,h)$, where $\eps \in \Ob(\B{E}(R)) = \{*\}$ and $h \in \Mor(\B{H}(R)) = H(R)$ is an isomorphism $\sgbar(\eps) \iso \tgbar(\eps)$. Hence the set of objects are the elements of $H(R)$. The morphisms $h \to h'$ are isomorphisms in $\B{E}(R)$, i.e., elements of $E(R)$, such that $\tau(e) h = h' \sigma(e)$. But this is the description of the prestack $[H \backslash E]$.
\end{proof}

%\begin{example}\label{QuotientBGGm}
%Let $O = \kappa$ be a finite field with $q$ elements and $O \to \kappa$ the identity. Let $A$ be the functor that attaches to a $\kappa$-algebra $R$ the tuple $(R, \Fil^i_{R}, \sigma_R, \alpha_R)$ defined as follows.
%\begin{assertionlist}
%\item
%We endow $R$ with the filtration given by $\Fil^i_R = 0$ for $i > 0$. Then the attached Rees stack is $R^{\Fil}$ \eqref{EqDefTheta}. Its attractor, repeller and fixed point locus is given by (Remark~\ref{ReesAttractor})
%\[
%(R^{\Fil})^0 = (R^{\Fil})^+ = \B{\GG_{m,R}}, \qquad (R^{\Fil})^- = R^{\Fil}.
%\]
%Then $\tau\colon \Spec R \to (R^{\Fil})^{\ne+} = \Spec R$ is the identity. 
%\end{assertionlist} 
%\end{example}
%

%---------------------------------------------------------------------

\subsection{Colimits of frame stacks}

In this section, we denote by  $(A_N, (\Fil_{A_N}^i), \sigma_N)_{N \geq 1}$ be an $\NN^{\opp}$-diagram of frames. We suppose that the following hypotheses are satisfied.
\begin{assertionlist}
\item
The maps $\Fil_{A_{N+1}}^i \to \Fil_{A_N}^i$ are surjective for all $N \geq 1$ and all $i \in \ZZ$, in particular $A_{N+1} \to A_N$ is surjective for all $N$.
\item
The kernel of $A_{N+1} \to A_N$ is contained in the Jacobson radical of $A_{N+1}$ for all $N \geq 1$.
\item
The $A_N$-modules $\Fil^i_{A_N}$ are finite projective for all $i$.
\end{assertionlist}
We obtain closed immersions of the attached Rees stacks $\Re(A_{N}) \to \Re(A_{N+1})$ for all $N \geq 1$. We set
\[
A := \lim_N A_N, \qquad \Fil_A^i := \lim_N \Fil_{A_N}^i
\]
and obtain a filtered ring $(A,(\Fil^i_A))$ with attached Rees stack $\Re(A)$. By Proposition~\ref{ColimitQuotientStack}, applied to $\ZZ$-graded $A[t]$-modules, we obtain an equivalence of Adams stacks
\begin{equation}\label{EqReesColimit}
\colim_N\Re(A_N) \liso \Re(A).
\end{equation}
Since we also have $\Spec A = \colim_N \Spec A_N$ by Example~\ref{AdicColimitAdam}, we also obtain formally
\begin{equation}\label{EqFrameColimit}
\colim_N\Fcal(A_N) \liso \Fcal(A)
\end{equation}
since coequalizers commute with colimits.

\section{$G$-bundles over the Rees stack attached to the $v$-adic filtration}\label{Sec:ReesVAdic}

In this chapter, we will study a type of Rees stack that will be used for (truncated) local shtukas and for prisms. Throughout the whole section let $A$ be a ring, let $L$ be an invertible $A$-module, and let $v\colon L \to A$ be an $A$-linear map. The pair $(L,v)$ corresponds to a section of $A^{\Fil} \to \Spec A$ which is also denoted by $v\colon \Spec A \to A^{\Fil}$.

%Recall that we denote by $T(L)$ the $\ZZ$-graded $A$-algebra $\bigoplus_{i\in \ZZ}L^{\otimes i}$ such that $\Spec T(L)$ is the underlying scheme of the $\GG_m$-bundle corresponding to $L$. We will also use the $A$-algebra $A[1/v]$ defined in Example~\ref{ExampleReesInvertible}. 

Define a filtration on $A$ by $\Fil^j := L^{\otimes j}$ for $j \geq 0$ and $\Fil^j := A$ for $j < 0$ and define $t\colon \Fil^{j+1} \to \Fil^{j}$ to be the multiplication with $v$ for $j \geq 0$ and the identity for $j < 0$. We obtain a filtered ring $(A,(L^{\otimes j})_j)$ with associated $\ZZ$-graded $A[t]$-algebra
\[
B := \Rees(A,v) := \bigoplus_{j \in \ZZ}\Fil^j
\]
and we denote by
\[
\Re(A,v) := [\GG_{m,A}\backslash \Spec B]
\]
the attached Rees stack.

\begin{remark}\label{VAdicSpecialCases}
There are the following two special cases.
\begin{assertionlist}
\item\label{VAdicSpecialCases1}
If $v$ is injective, then $L$ is identified via $v$ with an invertible ideal $I$ of $A$ and we also write $\Re(A,I)$ instead of $\Re(A,v)$.
\item\label{VAdicSpecialCases2}
If $L = A$, then $v$ is the multiplication with an element $z \in A$ and we also write $\Re(A,z)$ instead of $\Re(A,v)$. In this case we set $u := 1 \in \Fil^1$. Then $tu = z$ and we obtain a presentation of the underlying $\ZZ$-graded $A$-algebra of $B$ as
\begin{equation}\label{EqReesZAdic}
B = A[t,u]/(tu - z), \qquad \deg(t) = -1,\ \deg(u) = 1.
\end{equation}
Zariski locally on $A$, we can always choose an isomorphism $L \cong A$, i.e., Zariski locally on $A$, the $A[t]$-algebra $B$ is isomorphic to $A[t,u]/(tu-z)$ for some $z \in A$.
\end{assertionlist}
\end{remark}

%---------------------------------------------------------------------

\subsection{The Rees stack attached to the $v$-adic filtration}\label{Sec:ReesZAdic}

Let us describe the attractor, reppeller, and the fixed point locus of $\Re(A,v)$.

\begin{remark}\label{AttractorReAv}
By Remark~\ref{ReesAttractor} we have
\begin{equation}\label{EqAttractorLv}
\begin{aligned}
R := B^0 &= A/v(L),\\
B^- &= R[t],\\
B^+ &= B/tB = \bigoplus_{j \geq 0}L^{\otimes j} \otimes_A R = \Sym_R(L \otimes_A R).
\end{aligned}
\end{equation}
Hence the cartesian diagram~\eqref{EqCartesianAttractorRepellor} has the form
\begin{equation}\label{EqCartesianSpecialAttractor}
\begin{aligned}\xymatrix{
 & \Re(A,v)^{+} = [\GG_{m,R}\backslash \Spec \Sym_R(L \otimes_A R)] \ar[rd] \\
\Re(A,v)^0 = \B{\GG_{m,R}} \ar[ru] \ar[rd] & & \Re(A,v). \\
 & \Re(A,v)^- = R^{\Fil} \ar[ru]
}\end{aligned}
\end{equation}
The composition $\Re(A,v)^- = R^{\Fil} \to \Re(A,v)$ with the Rees map $\Re(A,v) \to A^{\Fil}$ is the natural closed embedding $R^{\Fil} \to A^{\Fil}$.

We have already seen in \eqref{EqReesStackNePlus} that the inclusion $\Re(A,v)^{\ne +} \to \Re(A,v)$ can be identified with
\[
\tau\colon \Re(A,v)^{\ne +} \cong \Spec A \lto \Re(A,v).
\]
Since $B^- = B/I^+$ and $I^+$ is the ideal generated by $L = B_1$, we find $\Re(A,v)^{\ne -} = [\GG_{m,A}\backslash B^{\ne -}]$ with $B^{\ne -} = B[L^{-1}] = T(L)$ (Remark~\ref{MapsToQuotientbyGGm}) which shows that
\[
\Re(A,v)^{\ne -} = [\GG_{m,A}\backslash T(L)] = \Spec A.
\]
Hence we find that
\[
\Re(A,v)^{\ne -} \cap \Re(A,v)^{\ne +} = \Spec A[1/v],
\]
with $A[1/v]$ as defined in Example~\ref{ExampleReesInvertible}. It follows that the algebraic stack
\[
\Re(A,v)^{\ne0} = \Re(A,v)^{\ne -} \cup \Re(A,v)^{\ne +}
\]
is the scheme obtained by gluing two copies of $\Spec A$ along $\Spec A[1/v]$. 
\end{remark}

Let us collect some properties of the scheme $\Re(A,v)^{\ne 0}$.

\begin{remark}\label{ReAne0ResolutionProp}
\begin{assertionlist}
\item
The scheme $\Re(A,v)^{\ne 0}$ is separated if and only if $\Spec A[1/v]$ is also a closed subscheme of $\Spec A$, for instance if $v$ is nilpotent and hence $A[1/v] = 0$.
\item
As the inclusion $\Spec A[1/v] \to \Spec A$ is affine, $\Re(A,v)^{\ne0}$ always has an affine diagonal.
\item
Let $U_1$ and $U_2$ be the two copies of $\Spec A$ in $\Re(A,v)^{\ne 0}$. For $n \in \ZZ$ define a line bundle $\Lscr_n$ on $\Re(A,v)^{\ne0}$ by gluing the line bundle defined by $L$ on $U_1$ and the the trivial line bundle on $U_2$ along $U_1 \cap U_2 = \Spec A[1/v]$ by the isomorphism $v^n$. We claim that $(\Lscr_1,\Lscr_{-1})$ is an ample pair of line bundles.

Indeed, global sections of $\Lscr_n$ are pairs $(b,a)$ with $b \in L^{\otimes n}$ and $a \in A$ such that $v^n(b) = a$ in $A[1/v]$. Choose $s_1,\dots,s_r \in L$ such that the non-vanishing loci of the $s_i$ are affine and cover $\Spec A$. Then the non-vanishing loci of the sections $(s_i,v(s_i))$ of $\Lscr_1$ are the open affine subscheme $(U_1)_{s_i}$, and these subschemes cover $U_1$. The non-vanishing locus of the global section $(v(1),1)$ of $\Lscr_{-1}$ is $U_2$. Hence $(\Lscr_1,\Lscr_{-1})$ is an ample pair.
\item
In particular, $\Re(A)^{\ne0}$ has the resolution property, i.e. every finite type quasi-coherent module over $\Re(A)^{\ne0}$ is the quotient of a vector bundle, even of a finite direct sum of tensor powers of $\Lscr_{-1}$ and $\Lscr_1$. Hence $\Re(A)^{\ne0}$ is an Adams stack.
\end{assertionlist}
\end{remark}

The construction of $\Re(A,v)$ is functorial in $(A,v)$ as follows.

\begin{remark}\label{FunctorialReesAv}
Let $A \to A'$ be a map of rings, let $L'$ be an invertible $A'$-module and $v'\colon L' \to A'$ be an $A'$-linear map. Suppose we are given a ring-homomorphism $f\colon A \to A'$ and an $A$-linear map $u\colon L \to L'$ such that $v' \circ u = f \circ v$. Then we obtain an induced map of algebraic stacks
\[
\Re(A,v) \lto \Re(A',v').
\]
\end{remark}

\begin{remark}\label{QCohSpecialRees}
Let $A$ and $v\colon L \to A$ be as above. Then the category of filtered modules over $(A, (L^{\otimes j})_{j\geq0})$, i.e. of quasi-coherent modules over $\Re(A,v)$, is equivalent to the category of $\ZZ$-graded $A$-modules $M = \bigoplus_j M_j$ together with $A$-linear maps
\[
t\colon M_{j+1} \to M_{j}, \qquad u\colon M_{j} \otimes_A L \to M_{j+1}
\]
such that $t \circ u$ is the multiplication with $v$ and $(u \otimes \id_{L^{\otimes-1}}) \circ t\colon M_{j+1} \to M_{j+1} \otimes L^{\otimes-1}$ is the multiplication with $v\vdual\colon A \to L\vdual = L ^{\otimes-1}$. 

For instance, for the free module $B$ of rank $1$, the maps $u\colon \Fil^j \otimes L \to \Fil^{j+1}$ are given by $v$ for $j < 0$ and by $\id_{L^{\otimes (j+1)}}$ for $j \geq 0$.

If $L = A$ and hence $v$ is given by multiplying with an element $z \in A$, then quasi-coherent modules over $\Re(A,z) = [\GG_{m,A}\backslash \Spec A[t,u]/(tu-z)]$ are graded $A$-modules $M = \bigoplus_j M_j$ together with $A$-linear maps $t\colon M_{j+1} \to M_{j}$ and $u\colon M_{j} \to M_{j+1}$ such that $t \circ u$ and $u \circ t$ is the multiplication with $z$.
\end{remark}

If $L = A$ and $z = p$ is a prime number, then a quasi-coherent module over $\Re(A,p)$ is nothing but a $p$-gauge in the sense of Fontaine and Jannsen \cite{FontaineJannsen_FrobeniusGauges}.

\begin{remark}\label{PullbackQCohReesSpecial}
As before let $R := A/v(L)$ and set $L_R := L \otimes_A R$. Let us describe the pullbacks of filtered modules over $(A, (L^{\otimes j}))$ given by $M = \bigoplus M_j$ as in Remark~\ref{QCohSpecialRees} to the various substacks. Since later on we are only interested in pullbacks of vector bundles, we describe only the non-derived pullback.
\begin{assertionlist}
\item\label{VectorbundlesSpecialRees1}
The pullback to the closed substack $\Re(A,v)^+ = \Sym_R(L_R)$ is given by the graded $R$-module $M/tM = \bigoplus_j M_j/t(M_{j+1})$ together with $R$-linear maps $u\colon M_{j-1}/tM_j \otimes_R L_R \to M_j/t(M_{j+1})$ (cf.~Example~\ref{MapsToAA1GGm}).
%, see Section~\ref{QCOHGGmSTACK} for the description of quasi-coherent modules on $[\GG_m\backslash \AA^1]$.

The pullback to $\Re(A,z)^- = R^{\Fil}$ is $M/u(M \otimes L) = \bigoplus_j M_j/u(M_{j-1} \otimes_A L)$ together with $R$-linear maps $t\colon M_{j+1}/u(M_j \otimes L) \to M_j/u(M_{j-1} \otimes L)$.

The pullback to the closed substack $\Re(A,z)^0 = \B{\GG_{m,R}} = \Re(A,v)^+ \cap \Re(A,v)^-$ is given by the graded $R$-module $\bigoplus_j M_j/(u(M_{j-1} \otimes L) + t(M_{j+1}))$.
\item\label{VectorbundlesSpecialRees2}
The pullback to the open substack $\Re(A,v)^{\ne+} = \Spec A$ is
\[
M_{-\infty} := \colim(\cdots \ltoover{t} M_{j+1} \ltoover{t} M_{j} \ltoover{t} M_{j-1} \ltoover{t} \cdots).
\]
The pullback to $\Re(A,v)^{\ne-} = \Spec A$ is 
\[
M_{\infty} := \colim(\cdots \ltoover{u} M_{j-1} \otimes L^{\otimes (1-j)} \ltoover{u} M_j \otimes L^{\otimes -j} \ltoover{u} M_{j+1} \otimes L^{-j-1} \ltoover{u} \cdots).
\]
Since the open substack $\Re(A,v)^{\ne0}$ is the scheme obtained by gluing two copies of $\Spec A$ along $\Spec A[1/v]$, quasi-coherent modules over $\Re(A,v)^{\ne0}$ are given by triples $(E,E',\Phi)$, where $E$ and $E'$ are $A$-modules and where $\Phi$ is an isomorphism $E[1/v] \iso E'[1/v]$. Then the pullback of $M$ to $\Re(A,v)^{\ne0}$ is given by the triple $(M_{-\infty}, M_{\infty}, \Phi)$, where $\Phi\colon M_{-\infty}[1/v] \iso M_{\infty}[1/v]$ is the canonical isomorphism induced by all isomorphisms $M_{j}[1/v] \iso M_{j+1}[1/v]$.
\end{assertionlist}
\end{remark}

\begin{example}\label{RestrictionTwistedLineBundle}
Set $\Xcal := \Re(A,v)$. Then the structure sheaf $\Oscr_{\Xcal}$ corresponds via Remark~\ref{QCohSpecialRees} to the $\ZZ$-graded $A[t]$-module $B = \bigoplus_{j\in \ZZ}\Fil^j$. The map $u\colon \Fil^j \otimes_A L \to \Fil^{j+1}$ is the multiplication with $v$ for $j < 0$ and the identity of $L^{j+1}$ for $j \geq 0$.

More generally, let $\Oscr_{\Xcal}(e)$ be the twisted line bundle for $e \in \ZZ$ (Section~\ref{QCOHGGmSTACK}). It corresponds to $M = \bigoplus M_j$ with $M_j = \Fil^{j+e}$.

Its restriction to $\Re(A,v)^{\ne 0}$ is given by $M_{-\infty} = A$, $M_{\infty} = L^{\otimes e}$ and $\Phi\colon A[1/v] \to L^{\otimes e}[1/v]$ the multiplication with $v^{-e}$.
\end{example}

We can also give the following description of vector bundles on $\Re(A,v)$.

\begin{proposition}\label{DescribeVBReesStack}
Let $A$, $v\colon L \to A$, and $R = A/v(L)$ as above with attached Rees algebra $B = \Rees(A,v)$ and attached Rees stack $\Re(A,v)$. Let $\Mscr$ be a quasi-coherent module over $\Re(A,v)$ corresponding to $M = \bigoplus_j M_j$ as in Remark~\ref{QCohSpecialRees}. Consider the following conditions on $M$.
\begin{definitionlist}
\item\label{DescribeVBReesStackb}
The $A$-modules $M_j$ are of finite presentation for all $j \in \ZZ$ and the maps $t\colon M_j \to M_{j-1}$ are isomorphisms for $j \ll 0$, and the maps $u\colon M_{j-1} \otimes L \to M_j$ are isomorphisms for $j \gg 0$.
\item\label{DescribeVBReesStacka}
The maps $t\colon M_{j+1}/u(M_{j} \otimes L) \to M_j/u(M_{j-1} \otimes L)$ and $u\colon M_{j-1}/tM_j \otimes_R L_R \to M_j/t(M_{j+1})$ are injective maps of finite projective $R$-modules with projective cokernels for all $j \in \ZZ$.
\item\label{DescribeVBReesStackc}
The $A$-modules $M_j$ are finite projective for all $j \in \ZZ$.
\item\label{DescribeVBReesStackd}
The maps $t\colon M_j \to M_{j-1}$ and $u\colon M_{j-1} \otimes L \to M_j$ are injective for all $j \in \ZZ$.
\end{definitionlist}
Then one has the following assertions.
\begin{assertionlist}
\item\label{DescribeVBReesStack1}
$M$ defines a module of finite presentation over $\Re(A,v)$ if and only if  \ref{DescribeVBReesStackb} holds.
\item\label{DescribeVBReesStack2}
Suppose that $M$ defines a vector bundle over $\Re(A,v)$. Then \ref{DescribeVBReesStacka}, \ref{DescribeVBReesStackb}, and \ref{DescribeVBReesStackc} hold. If $v$ is injective, then also \ref{DescribeVBReesStackd} holds.
\item\label{DescribeVBReesStack4}
Suppose that $v$ is injective. Then the following assertions are equivalent.
\begin{equivlist}
\item\label{DescribeVBReesStacki}
$M$ defines a vector bundle over $\Re(A,v)$.
\item\label{DescribeVBReesStackii}
All conditions \ref{DescribeVBReesStackb} -- \ref{DescribeVBReesStackd} hold.
\item\label{DescribeVBReesStackiii}
Condition~\ref{DescribeVBReesStackb} holds, the maps $u\colon M_{j-1}/tM_j \otimes_R L_R \to M_j/t(M_{j+1})$ are injective maps of finite projective $R$-modules with projective cokernels for all $j \in \ZZ$, $t\colon M_j \to M_{j-1}$ is injective for all $j$, and $M_{-\infty}$ is a finite projective $A$-module.
\item\label{DescribeVBReesStackiv}
Condition~\ref{DescribeVBReesStackb} holds, the maps $t\colon M_{j+1}/u(M_{j} \otimes L) \to M_j/u(M_{j-1} \otimes L)$ are injective maps of finite projective $R$-modules with projective cokernels for all $j \in \ZZ$, $u\colon M_j \otimes L \to M_{j+1}$ is injective for all $j$, and $M_{\infty}$ is a finite projective $A$-module.
\item\label{DescribeVBReesStackv}
Conditions~\ref{DescribeVBReesStackb}, \ref{DescribeVBReesStacka}, and \ref{DescribeVBReesStackc} hold, and $M_{-\infty}[1/v]$ (isomorphic to $M_{\infty}[1/v]$) is a finite projective $A[1/v]$-module.
\end{equivlist}
\item\label{DescribeVBReesStack5}
Suppose that $v(L)$ is contained in the Jacobson radical of $A$. Then $M$ defines a vector bundle over $\Re(A,z)$ if and only if Conditions~\ref{DescribeVBReesStackb}, \ref{DescribeVBReesStacka}, and \ref{DescribeVBReesStackc} hold.
\item\label{DescribeVBReesStack3}
Suppose that $v = 0$, then $M$ defines a vector bundle over $\Re(A,v)$ if and only if \ref{DescribeVBReesStacka} holds.
\end{assertionlist}
\end{proposition}

%The proof will show that $z$ is regular if and only if for every projective module $M = \bigoplus M_j$ all $M_j$ are projective $A$-modules. 

\begin{proof}
All assertions and conditions are Zariski locally on $A$, therefore we may assume $L =A$ and hence $B = A[t,u]/(tu-z)$ with $\deg(t) =-1$, $\deg(u) = 1$ and $z \in A$ by Remark~\ref{VAdicSpecialCases}. In this case $v$ is injective if and only if $z = v(1)$ is regular in $A$.

Let us show \ref{DescribeVBReesStack1}. Since $M$ is of finite presentation, it is the cokernel of a map of graded $B$-modules $F' \to F$ such that both $F'$ and $F$ are finite graded free. Since \ref{DescribeVBReesStackb} holds for $B(e)$, we deduce that it holds for $M$.
Conversely, suppose that \ref{DescribeVBReesStackb} holds. Let $a \leq b$ be integers such that $t\colon M_j \to M_{j-1}$ is an isomorphism for $j \leq a$ and such that $u\colon M_j \to M_{j+1}$ is an isomorphism for $j \geq b$. If $S$ is a homogeneous generating system for the $A$-module $M^a \oplus M^{a+1} \oplus \cdots \oplus M^b$, then $S$ is a generating system for the graded $B$-module $M$. It defines a surjection from a finite graded free module $F \to M$. Its kernel $N$ is a graded $B$-module such that still the maps $t\colon N_j \to N_{j-1}$ are isomorphisms for $j \ll 0$, and the maps $u\colon N_{j-1} \to N_j$ are isomorphisms for $j \gg 0$. Since $M_j$ is of finite presentation, $N_j$ is a finitely generated $A$-module. Now the same argument as before shows that we again find a surjection from a finite graded free module $F' \to N$. Therefore, we have constructed a finite presentation $F' \to F \to M \to 0$ by finite graded free modules. Hence $M$ is of finite presentation.

Next we show \ref{DescribeVBReesStack2}. Let $M$ define a finitely generated projective module. Then $M$ is of finite presentation and hence \ref{DescribeVBReesStackb} holds for $M$. Pulling back $M$, considered as a vector bundle over $\Re(A,z)$, to $\Re(A,z)^+$ and $\Re(A,z)^-$, it follows from Remark~\ref{PullbackQCohReesSpecial}~\ref{VectorbundlesSpecialRees1} and the description of vector bundles on $[\GG_m\backslash \AA^1]$ in Proposition~\ref{CharFilteredVB} that $M$ satisfies \ref{DescribeVBReesStacka}. Moreover, $M$ is a direct summand of a finite graded free $B$-module (Proposition~\ref{CharVBonQuotientByGGm}). Since each graded piece of $B$ is equal to $A$, it follows that $M_j$ is a finite projective $A$-module for all $j$. Finally, since $M$ is flat, the multiplication with $z$ on $M$ is injective. Since $z = tu = ut$, this implies that $t$ and $u$ are injective.

Let us show \ref{DescribeVBReesStack3}. In this case $B = A[t,u]/(tu)$ is the fiber product of graded rings $A[u] \times_A A[t]$, and therefore a quasi-coherent $\Re(A,0)$-module is a vector bundle if and only if its pullback to $\Re(A,0)^+$ and $\Re(A,0)^-$ is a vector bundle \cite[2.2]{Ferrand_Conducteur} which is equivalent to \ref{DescribeVBReesStacka}, again by Remark~\ref{PullbackQCohReesSpecial}~\ref{VectorbundlesSpecialRees1}. 

We now show \ref{DescribeVBReesStack4}, so we assume that $z$ is regular. We have already seen in \ref{DescribeVBReesStack2}, that \ref{DescribeVBReesStacki} implies \ref{DescribeVBReesStackii}. Clearly, \ref{DescribeVBReesStackii} implies \ref{DescribeVBReesStackiii}, \ref{DescribeVBReesStackiv}, and \ref{DescribeVBReesStackv}.

Let us show that \ref{DescribeVBReesStackiii} implies \ref{DescribeVBReesStacki}. We apply Proposition~\ref{FlatnessCrit} to $I = (t) \subseteq \Rees(A)$. Let $\Escr$ be the quasi-coherent module on $\Re(A,z)$ corresponding to $M$ such that $\Mtilde$ is the pullback of $\Escr$ to $\Spec B$. By \ref{DescribeVBReesStack1} we know that $M$ is a $B$-module of finite presentation. The restriction of $\Mtilde$ to $U = \Spec B \setminus V(t)$ is the pullback of $\Escr\rstr{\Re(A,z)^{\ne+}}$ which corresponds to the $A$-module $M_{-\infty}$ which by hypothesis is flat. The (non-dervied) restriction of $M$ to $V(t)$ is the pullback of the restriction of $\Escr$ to $\Re(A,z)^+ = [\GG_{m,R}\backslash \AA^1_R]$ which is a vector bundle by Proposition~\ref{CharFilteredVB}. It remain to show that $\Tor^B_1(B/tB,M) = 0$. Since $z = ut$, $t \in B$ is a regular element, i.e. one has $B/tB \cong (B \ltoover{t} B)$ in $D(B)$ and hence $\Tor^B_1(B/tB,M) = \Ker(M \ltoover{t} M) = 0$.

To see that \ref{DescribeVBReesStackiv} implies \ref{DescribeVBReesStacki} one argues in the same manner applying Proposition~\ref{FlatnessCrit} to $(u) \subseteq \Rees(A)$. Moreover, the same argument, now by applying Proposition~\ref{FlatnessCrit} to $(z) \subseteq B$ also shows that \ref{DescribeVBReesStackv} implies \ref{DescribeVBReesStacki}. Here we use that the flatness of $M/zM$ is equivalent to Condition~\ref{DescribeVBReesStacka}, which we have already proved, and that condition \ref{DescribeVBReesStackc} implies that $z\colon M_j \to M_j$ is injective for all $j$ and hence $\Tor_1^B(B/zB,M) = \Ker(z\colon M \to M) = 0$.
%
%Suppose that $M$ defines a flat module over $\Re(A)$, e.g. if $M$ defines a vector bundle. Since $t$ and $s$ are both regular in $\Rees(A)$, it follows that $t$ and $s$ are injective on $M$, i.e., \ref{DescribeVBReesStackd} holds.
%
%If \ref{DescribeVBReesStackb} and \ref{DescribeVBReesStackc} hold, then the pullbacks $M_{-\infty}$ to $\Re(A)^{\ne +}$ and $M_{\infty}$ to $\Re(A)^{\ne -}$ are both finite projective.
%???
%
%Since $zM_j = t(u(M_j)) \subseteq t(M_{j+1})$ we obtain a split exact sequence of finite projective $R$-modules
%\[
%0 \lto M_{j+1}/u(M_j) \ltoover{t} M_j/zM_j \lto M_j/t(M_{j+1}) \lto 0
%\]

It remains to show \ref{DescribeVBReesStack5}. Again we apply Proposition~\ref{FlatnessCrit} to $I = (z) \subseteq B$. The open subset $U = \Spec A[1/z]$ does not contain any closed point because $z$ is contained in the Jacobson radical of $A$, so the condition that $\Escr\rstr{U}$ is flat is unnecessary by Remark~\ref{RemFlatnessCrit}. The module $M$ corresponds to a module of finite presentation by \ref{DescribeVBReesStackb} and its restriction to the vanishing locus of $z$ is flat by \ref{DescribeVBReesStacka} as we have seen in \ref{DescribeVBReesStack3}. It remains to see that
\[
\Tor_1^{B}(B/zB, M) = \Tor_1^A(A/z,M) = 0,\tag{*}
\]
where the equality holds by Lemma~\ref{TorFlatExtension} below since $B$ is a flat $A$-algebra (\cite[14.22]{GW1}). But as an $A$-module, $M$ is by \ref{DescribeVBReesStackc} the direct sum of the finite projective $A$-modules $M_j$ which implies (*).
%
%Since $z$ is regular in $A$, a free resolution of $A_N/zA_N$ is given by the cochain complex
%\[
%\dots \ltoover{z^{N-1}} A_N \ltoover{z} A_N \ltoover{z^{N-1}} A_N \ltoover{z} A_N \lto 0,
%\]
%where the last $A_N$ sits in degree $0$. Since $B_N$ is a flat $A_N$-algebra, we see that $\Tor_1^{B_N}(B_N/zB_N, M)$ is $H^{-1}(\ )$ of the complex
%\[
%\dots \ltoover{z^{N-1}} M \ltoover{z} M \ltoover{z^{N-1}} M \ltoover{z} M \lto 0,
%\]
%which is just $\Ker(M \ltoover{z} M)/z^{N-1}M$. This module is zero if and only if $\Ker(M_j \ltoover{z} M_j)/z^{N-1}M_j = 0$, which is the case since the $M_j$ are finite projective $A_N$-modules by assumption.
\end{proof}

%We now come to the description of vector bundles over the Rees stack $\Re(W(R))$ (Example~\ref{ReesWitt}).
%
%\begin{remark}\label{VectorbundlesWittRees}
%
%\end{remark}

In the proof above we used to following easy fact.

\begin{lemma}\label{TorFlatExtension}
Let $A \to B$ and $B \to C$ maps of rings and suppose that $A \to B$ or $A \to C$ is flat\footnote{or, more generally, that $A \to B$ and $A \to C$ are tor-independent}. Let $M$ be a $B$-module. Then $\Tor^B_n(C \otimes_A B,M) = \Tor^A_n(C,M)$ for all $n$.
\end{lemma}

\begin{proof}
We have $(C \otimes_A B) \Lotimes_B M = C \Lotimes_A B \Lotimes_B M = C \Lotimes_A M$ since $A \to B$ and $A \to C$ are tor-independent.
\end{proof}

\subsection{Vector bundles and $G$-bundles on $\Re(A,v)$}

We keep the notation from Section~\ref{Sec:ReesZAdic} and we set $R := A/v(L)$ as before. The closed immersion $\B{\GG_{m,R}} = \Re(A,v)^0 \lto \Re(A,v)$ yields a map of groupoids
\begin{equation}\label{EqTypeMapPoints}
\Bun_G(\Re(A,v)) \lto \Bun_G(\B{\GG_{m,R}}).
\end{equation}

\begin{proposition}\label{LiftGBundlesReAv}
Let $G$ be a smooth and affine group scheme over $A$ and suppose that $(A,v(L))$ is a henselian pair. Then \eqref{EqTypeMapPoints} is full and essentially surjective and in particular we obtain an isomorphism
\[
H^1(G,\Re(A,v)) \iso H^1(G,\B{\GG_{m,R}}).
\]
\end{proposition}

\begin{proof}
This is a special case of \cite[2.9]{Wedhorn_ExtendBundles}.
\end{proof}

\begin{remark}\label{DescribeVBReAv}
If $(A,v(L))$ is a henselian pair, then every vector bundle $E$ over $R = A/v(L)$, lifts to a vector bundle $\Etilde$ over $A$, and $\Etilde$ is unique up to isomorphism. Moreover, every vector bundle over $A$ is of this form. Hence every vector bundle $M = \bigoplus_j M_j$ over $\Re(A,v)$ is isomorphic to a vector bundle of the form
\begin{equation}\label{EqDescribeVBReesStack}
\bigoplus_{j\in \ZZ}\Etilde_j \otimes_A B(j),
\end{equation}
where $(E_j)_{j\in \ZZ}$ is a family of finite projective $R$-modules which are almost all zero. Moreover the $E_j$ are uniquely determined up to isomorphism.

In particular, we see that Zariski locally on $\Spec R$ (not only on $\Spec A$), $M$ is a finite graded free $B$-module.
\end{remark}

\begin{proposition}\label{FullyFaithfulness0Complement}
Suppose that $v\colon L \to A$ is injective, set $I := v(L)$, which is an invertible ideal of $A$, and let $\Re(A,I)$ be the corresponding Rees stack.
\begin{assertionlist}
\item
The restriction functors
\begin{equation}\label{EqRestrictReNe0}
\begin{aligned}
\Vec(\Re(A,I)) &\lto \Vec(\Re(A,I)^{\ne 0}),\\
\Bun_G(\Re(A,I)) &\lto \Bun_G(\Re(A,I)^{\ne0})
\end{aligned}
\end{equation}
are fully faithful for every affine flat group scheme $G$ over $A$, .
\item
Suppose that $A$ is a Dedekind domain and that $I = \mfr_1\cdots\mfr_r$ for $r \geq 1$ pairwise different maximal ideals $\mfr_i$. Then the restrictions \ref{EqRestrictReNe0} are equivalences for every reductive group scheme $G$ over $A$.
\end{assertionlist}
\end{proposition}

\begin{proof}
Let $B$ be the underlying $\ZZ$-graded $A$-algebra of the filtered ring, set $X = \Spec B$ and let $U := X \setminus X^0$ the complement of the fixed point locus. To show (1), by Theorem~\ref{ExtendGBundle}~\ref{ExtendGBundle1} we have to show that $\Oscr_X \to j_*\Oscr_U$ is an isomorphism. This can be checked Zariski locally on $A$ and hence we can assume that $I = (z)$ for a regular element $z \in A$. Since $X$ is affine, it suffices to show that restriction yields an isomorphism $\Gamma(X,\Oscr_X) \iso \Gamma(U,\Oscr_X)$.

By Section~\ref{Sec:ReesZAdic} we have $B \cong A[t,u]/(tu-z)$ and $U$ is the open subscheme $D(t) \cup D(u)$ of $X = \Spec A[t,u]/(tu-z)$. Let $j\colon U \to X$ be the inclusion.

Using the open affine covering $U = D(t) \cup D(u)$ we see that $\Gamma(U,\Oscr_X) = A[t,t^{-1}] \cap A[u,u^{-1}]$ where the intersection takes place in $A[1/z][t,t^{-1}]$ which contains $A[u,u^{-1}]$ via $u \sends zt^{-1}$ (since $z$ is regular, $A \to A[1/z]$ is injective). But this intersection is equal to $A[t,u]/(tu-z)$, embedded into $A[1/z][t,t^{-1}]$ by $t \sends t$ and $u \sends zt^{-1}$.

To show Assertion~(2) we will use Theorem~\ref{ExtendGBundle}~\ref{ExtendGBundle2}. To show that $B$ is a regular ring of dimension $2$ we may again work locally on $A$ and hence we may assume that $B = A[t,u]/(tu-p_1\cdots p_r)$ for pairwise coprime prime elements $p_i \in A$. Then $(t,u)$ is a regular system in $B$ such that $B/(t,u) = A/(p_1\cdots p_r)$ is a product of fields. This shows that $B$ is regular of dimension $2$ and that $D(t) \cup D(u)$ contains every point of codimension $\leq 1$, allowing us to apply Theorem~\ref{ExtendGBundle}~\ref{ExtendGBundle2}.
\end{proof}

Let us describe the essential image of \eqref{EqRestrictReNe0} for arbitrary rings. We start with vector bundles. Recall that a vector bundle (resp~a quasi-coherent module) over $\Re(A,I)^{\ne 0}$ is given by a tripel $(E,E',\Phi)$, where $E$ and $E'$ are vector bundles (resp.~modules) over $A$ and where $\Phi\colon E[1/I] \iso E'[1/I]$ is an isomorphism over $A[1/I]$. Then
\begin{equation}\label{EqGloablSectionReAvne0}
\begin{aligned}
\Gamma(\Re(A,v)^{\ne 0},\Escr) &= \set{(e,e') \in E \times E'}{\Phi(e) = e'} \\
&= \set{e \in E}{\Phi(e) \in E'} = \Phi^{-1}(E') \cap E.
\end{aligned}
\end{equation}

\begin{lemma}\label{DirectImageVB}
Let $\iota\colon \Re(A,I)^{\ne 0} \to \Re(A,I)$ be the inclusion, which is a quasi-compact open immersion. Let $\Escr$ be a quasi-coherent module over $\Re(A,I)^{\ne 0}$ given by a triple $(E,E',\Phi)$ as above. Using the notation introduced in Remark~\ref{QCohSpecialRees}, the quasi-coherent module $\iota_*\Escr$ over $\Re(A,I)$ is given
\[
M_j = I^j\Phi^{-1}(E') \cap E, \qquad\text{intersection in $E[1/I]$},
\]
$t\colon M_j \to M_{j-1}$ the inclusion, $u\colon I \otimes A M_{j} \to M_{j+1}$ the multiplication.
\end{lemma}

\begin{proof}
Set $\Xcal := \Re(A,v)$. Since we know that $\iota^*$ is fully faithful, we have $\iota_*\iota^* \cong \id$. By \eqref{EqGradingGlobalSection} we have
\begin{align*}
M_j &= \Gamma(\Xcal,\iota_*(\Escr)(j)) \\
&= \Gamma(\Xcal, \iota_*(\Escr \otimes \iota^*\Oscr_{\Xcal}(j)) \\
&= \Gamma(\Re(A,v)^{\ne0}, \Escr \otimes \iota^*\Oscr_{\Xcal}(j)),
\end{align*}
where the second identity holds by the projection formula. By Example~\ref{RestrictionTwistedLineBundle}, $\Escr \otimes \iota^*\Oscr_{\Xcal}(j)$ corresponds to the triple $(E, E' \otimes I^j, \Phi \otimes v^{-j})$ and by \eqref{EqGloablSectionReAvne0} we therefore find
\[
M_j = I^j\Phi^{-1}(E') \cap E.
\]
For this note that if we identify $L$ with $I$, then $v$ is just the inclusion and hence $v^{-j}\colon I^{\otimes-j}[1/I] \cong A[1/I] \to A[1/I]$ is just the identity.

The map $t\colon M_{j} = \Gamma(\Xcal,\iota_*(\Escr)(j)) \to M_{j-1}$ is induced by the inclusion $\Oscr_{\Xcal}(j) \to \Oscr_{\Xcal}(j-1)$, given by $\Fil^d \to \Fil^{d-1}$ for $d \in \ZZ$, which shows that $t$ is indeed just the inclusion. Then $u$ has also the stated form since it is determined by $t$.
\end{proof}

\begin{proposition}\label{VBEssentialImage}
Suppose that $v\colon L \to A$ is injective and set $I := v(L)$. For a vector bundle on $\Re(A,v)^{\ne0}$ given by $(E,E',\Phi)$ the following assertions are equivalent.
\begin{equivlist}
\item\label{VBEssentialImage1}
The vector bundle $(E,E',\Phi)$ is in the essential image of \eqref{EqRestrictReNe0}.
\item\label{VBEssentialImage2}
The following conditions are satisfied.
\begin{definitionlist}
\item\label{VBEssentialImagea}
$M_j := I^j\Phi^{-1}(E') \cap E$ (intersection in $E[1/I]$) is a finite projective $A$-module.
\item\label{VBEssentialImageb}
$M_j = M_{j-1}$ for $j \ll 0$ and $M_j = IM_{j+1}$ for $j \gg 0$.
\item\label{VBEssentialImagec}
$M_j/(M_{j+1} + IM_{j-1})$ is a finite projective $R$-module for all $j \in \ZZ$.
\end{definitionlist}
\item\label{VBEssentialImage3}
The vector bundle $(E,E',\Phi)$ is Zariski locally on $\Spec A$ isomorphic to a triple of the form $(A^n, A^n, \Phi)$, where $\Phi$ is given by a diagonal matrix $\textup{diag}(z^{e_1},\dots,z^{e_n})$ for a generator $z$ of $I$, which exists Zariski locally, and $(e_1,\dots,e_n) \in \ZZ_+^n$, i.e. $e_i \in \ZZ$ with $e_1 \geq \dots \geq e_n$.
\end{equivlist}
\end{proposition}

\begin{proof}
Let $\iota$ be as in Lemma~\ref{DirectImageVB}. By \cite[24.67]{GW2}, which easily generalizes to stacks, a vector bundle $\Escr$ on $\Re(A,v)^{\ne0}$ is in the essential image of \eqref{EqRestrictReNe0} if and only if $\iota_*(\Escr)$ is a vector bundle. In Lemma~\ref{DirectImageVB} we have seen a description of $\iota_*(\Escr)$. To see that the conditions in \ref{VBEssentialImage2} just mean that $\iota_*(\Escr)$ is a vector bundle we apply Proposition~\ref{FlatnessCrit} with $I = (u,t)$. Conditions~\ref{VBEssentialImagea} and~\ref{VBEssentialImageb} imply that $\iota_*\Escr$ is of finite presentation. As the $M_j$ are finite projective, the maps $t$ and $u$ are injective which implies as in the proof of Proposition~\ref{DescribeVBReesStack} that $\Tor^B_1(B/IB,M) = 0$, since $(t,u)$ is a regular system. By hypothesis the restriction of $\iota_*\Escr$ to the non-vanishing locus $\Re(A,v)^{\ne 0}$ of $I$ is locally free. Hence Proposition~\ref{FlatnessCrit} shows that $\iota_*\Escr$ is a vector bundle. Hence~\ref{VBEssentialImage1} and \ref{VBEssentialImage2} are equivalent.

By Proposition~\ref{CharVBonQuotientByGGm}, every vector bundle on $\Xcal := \Re(A,v)$ is Zariski locally isomorphic to a direct sum of twisted line bundles $\Oscr_{\Xcal}(e)$. By Example~\ref{RestrictionTwistedLineBundle} we see therefore that \ref{VBEssentialImage1} and \ref{VBEssentialImage3} are equivalent.
\end{proof}

\begin{remark}\label{DescribeTrivialHeckeStack}
The vector bundle $(A^n,A^n,\Phi)$ in \ref{VBEssentialImage3} of Proposition~\ref{VBEssentialImage} has the following more conceptual description which also generalizes to $G$-bundles for more general smooth affine group schemes $G$: Let $P_0$ be the trivial $\GL_n$-bundle on $\Spec \Rees(A,v)$, let $(e_1,\dots,e_n) \in \ZZ^n_+$, and let $\mu\colon \GG_m \to GL_n$ by the corresponding cocharacter $t \sends \textup{diag}(t^{e_1},\dots,t^{e_n})$. Then $\mu$ defines a $\GG_m$-equivariant structure on $P_0$ and hence a $\GL_n$-bundle on $\Re(A,v)$. In other words, it is the trivial $\mu \backslash \GL_n$-bundle $\Escr_{\mu}$ on $\Re(A,v)$, see Definition~\ref{DefineGModMu}. It corresponds to the vector bundle $\bigoplus_{i=1}^n\Oscr_{\Re(A,v)}(e_i)$ whose restriction is the vector bundle $(A^n,A^n,\Phi)$ in \ref{VBEssentialImage3} of Proposition~\ref{VBEssentialImage}.
\end{remark}

As $\iota^*$ is monoidal and exact, the full subcategory of vector bundles satisfying the equivalent conditions of Proposition~\ref{VBEssentialImage} form an exact symmetric monoidal subcategory of the symmetric monoidal category of all vector bundles on $\Re(A,v)^{\ne 0}$. We denote this subcategory by $\Vec(\Re(a,v)^{\ne0})^0$.

Now assume that $A$ is an $O$-algebra, where $O$ is a Dedekind domain, and let $G$ be an affine smooth group scheme over $O$.

\begin{corollary}\label{RestrictGBunReAvne0}
For a $G$-bundle $\Escr$ on $\Re(A,v)^{\ne0}$ the following assertions are equivalent.
\begin{equivlist}
\item
There exists a $G$-bundle $\tilde\Escr$ on $\Re(A,v)$ such that $\Escr \cong \iota^*\tilde\Escr$.
\item
There exists an \'etale surjective map $\Spec O' \to \Spec O$ and a conjugacy class of cocharacters $[\mu]$ of $G$ defined over $O'$ such that $\Escr$ is \'etale locally isomorphic to the $G$-bundle $\Escr_{[\mu]}$.
\item
The exact monoidal functor $\Rep(G) \to \Vec(\Re(a,v)^{\ne0})$ corresponding to $\Escr$ factors through $\Vec(\Re(a,v)^{\ne0})^0$.
\end{equivlist}
\end{corollary}

%---------------------------------------------------------------------

\subsection{Colimits of truncations of $\Re(A,v)$}\label{COLIMITREES}

In this section let $A$ be a ring, $L$ a line bundle, and $v\colon L \to A$ be an $A$-linear map. We obtain a filtered ring $(A,(L^{\otimes j})_{j\geq0})$ as in Section~\ref{Sec:ReesZAdic} and the attached Rees stack $\Re(A,v)$. We fix also an ideal $J \subseteq A$ such that $A$ is $J$-adically complete.

For an integer $N \geq 1$ let $\pi_N\colon A \to A/J^N$ be the canonical map. Let $\pi_N^*(A,(L^{\otimes j})) = (A/J^N, (L^{\otimes j} \otimes_A A/J^N)_j)$ be its base change (Remark~\ref{ExtensionFilteredRing}). We obtain the attached graded Rees stack
\[
\Re(A/J^N,v) = \Re(A,v) \otimes_A A/J^N.
\]
We also set $A/J^{\infty} := A$.
%We have in mind the following two examples although we will focus on the first one and leave (a generalization of) the second one to \cite{???}.
%\begin{assertionlist}
%\item
%$A = R\psz$,
%\item
%$A = A_{\inf}(R) = W(R^{\flat})$ for a integral perfectoid ring $R$ and $z = \varphi^{-1}(d)$, where $(d)$ and $\varphi$ are defined by the natural structure of a prism on $A_{\inf}(R)$.
%\end{assertionlist}
%
For $1 \leq N \leq N' \leq \infty$, the canonical maps $A/J^{N'} \to A/J^N$ yield a closed immersion of algebraic stacks
\begin{equation}\label{EqTruncation}
\Re(A/J^N,v) \to \Re(A/J^{N'},v).
\end{equation}

By Proposition~\ref{ColimitQuotientStack} we have the following result. 

\begin{proposition}\label{InftyHeckeAsColimit}
The natural map of stacks $\colim_{N<\infty}\Re(A/J^N,v) \lto \Re(A,v)$, where we form the colimit in the 2-category of Adams stacks, is an isomorphism.
\end{proposition}

%This also follows from  (WHICH IS YET UNPROVED). Note that $\Rees(A) \to \lim_{N < \infty}\Rees(A_N)$ is not an isomorphism.
%
%\begin{proof}
%We apply Theorem~\ref{AdamsStackCocomplete}. Therefore it suffices to to show that
%\begin{assertionlist}
%\item
%$\Vec(\Re(A)) \lto \lim_{N<\infty}\Vec(\Re(A_N))$ is an equivalence and that
%\item
%if $w\colon \Escr \to \Fscr$ is a map of finitely presented modules over $\Re(A)$ such that its reduction module $z^N$ is surjective for some $N$, then $w$ is surjective.
%\end{assertionlist}
%Every quasi-coherent module over $\Re(A_N)$ for $N \leq \infty$ is given by a graded module $M = \bigoplus M_j$ together with homogeneous endomorphisms $u$ and $t$ of degree $1$ and $-1$, respectively, see Remark~\ref{QCohSpecialRees}. The description of vector bundles in Proposition~\ref{DescribeVBReesStack}~\ref{DescribeVBReesStack5} gives the first condition since $\Vec(A) = \lim_{N<\infty}\Vec(A/z^N)$ by $z$-adic completeness of $A$. For the second condition, it suffices to show that $w$ induces a surjective map $M_j \to N_j$ for all $j$ where $M = \bigoplus M_j$ corresponds to $\Escr$ and $N = \bigoplus N_j$ to $\Fscr$. But since $N_j$ is of finite presentation by Proposition~\ref{DescribeVBReesStack}~\ref{DescribeVBReesStack1}, this follows from Nakayama's lemma.
%\end{proof}
%

To apply these results to $G$-bundles on $\Re(A,v)$, we now assume that $A$ is an $O$-algebra for some noetherian regular ring $O$ of dimension $\leq 1$ and that $G$ is a smooth affine group scheme over $O$. For $1 \leq N \leq N' \leq \infty$ the closed immersion $\Re(A/J^N,v) \to \Re(A/J^{N'},v)$ yields by pullback a map of groupoids
\begin{equation}\label{EqReduceBunG}
\Bun_G(\Re(A/J^{N'},v)) \lto \Bun_G(\Re(A/J^N,v)).
\end{equation}

\begin{theorem}\label{LimitBunGTruncationofReesStacks}
\begin{assertionlist}
\item\label{LimitBunGTruncationofReesStacks1}
The maps \eqref{EqReduceBunG} are full and essentially surjective for all $1 \leq N \leq N' \leq \infty$. In particular they induce isomorphisms
\[
H^1(G,\Re(A/J^{N'},v)) \iso H^1(G,\Re(A/J^{N},v)).
\]
\item\label{LimitBunGTruncationofReesStacks2}
The maps \eqref{EqReduceBunG} yield an equivalence
\[
\Bun_G(\Re(A,v)) \liso \lim_{N<\infty}\Bun_G(\Re(A/J^N,v)).
\]
\end{assertionlist}
\end{theorem}

\begin{proof}
The fact that \eqref{EqTypeMapPoints} is full and essentially surjective is a special case of Theorem~\ref{LiftGBundles}. Since $A$ is $J$-adically complete, $(A,J)$ is a henselian pair and therefore $\Bun_G(\Re(A_{N'}) \lto \Bun_G(\Re(A_N))$ is full and essentially surjective by \cite[2.6]{Wedhorn_ExtendBundles}. This shows \ref{LimitBunGTruncationofReesStacks1}. 

Assertion~\ref{LimitBunGTruncationofReesStacks2} follows from Corollary~\ref{GeneralLimitGBundle}.
%For every stack $\Xcal$, the groupoid $\Bun_G(\Xcal)$ is by definition the groupoid $\Hom(\Xcal,\B{G})$. Moreover, the classifying stack $\B{G}$ is an Adams stack because of our assumption on $G$, see e.g., \cite[A.28]{Wedhorn_ExtendBundles}. So we obtain \ref{LimitBunGTruncationofReesStacks2} from Proposition~\ref{InftyHeckeAsColimit}.
\end{proof}

%---------------------------------------------------------------------

\subsection{The frame of $(A,v,\phi)$}\label{Sec:FrameAvphi}

We keep the notation $A$, $v\colon L \to A$, as before. Let $\phi\colon A \to A$ be a ring endomorphism. Then $\phi$ induces a map
\begin{equation}\label{EqDefSigmaVAdicFrame}
\sigma\colon \Spec A \ltoover{\phi} \Spec A \cong \Re(A,v)^{\ne 0} \mono \Re(A,v)
\end{equation}
and we obtain a frame attached to $(A,v,\phi)$ and hence the frame stack $\Fcal(A,v,\phi)$ and its open substack $\Fcal(A,v,\phi)^{\ne0}$ (Definition~\ref{DefFrameStack}).

%
%\begin{example}\label{ZAdicReesFrobenius}
%Let $A$ be a ring, $z \in A$, $\Rees(A) = \Spec A[t,u]/(tu-z)$, and $\Re(A) = \Re(A,z) := [\GG_{m,A}\backslash \Rees(A)]$ be the corresponding Rees stack. Let $\phi\colon A \to A$ be an endomorphism of $A$ with $\phi(z) = z$. We define $\sigma\colon \Spec A \to \Re(A,z)$ as the composition
%\[
%\Spec A \ltoover{\phi} \Spec A \cong [(\Spec A[u,u^{-1}])/\GG_{m,A}] = \Re(A)^{\ne -} \mono \Re(A).
%\]
%Since $\Re(A)^{\ne-} \to \Re(A)$ is given by $\Rees(A) \to A[u,u^{-1}]$, $t \sends zu^{-1}$, the map $\sigma$ is given by $\sigma^*\colon \phi^*\Rees(A) \to \bigoplus_{i\in \ZZ}A^{\otimes i}$ with $\sigma^*_{-1}(1 \otimes at) = \phi(a)z \in A^{\otimes-1}$ and $\sigma^*_1(1 \otimes au) = \phi(a)$.
%
%We call this frame the \emph{$z$-adic frame on $(A,\Phi)$}.
%\end{example}

\begin{corollary}\label{GBundlesZAdicFRameStack}
Let $(A,v, \phi)$ be as above. Suppose that $v$ is injective. Let $A$ be an $O$-algebra, where $O$ is a Dedekind domain, and let $G$ be a smooth affine group scheme over $O$.
\begin{assertionlist}
\item\label{GBundlesZAdicFRameStack1}
Pullback via $\Fcal(A,v,\phi)^{\ne0} \to \Fcal(A,v,\phi)$ yield fully faithful functors
\begin{equation}\label{EqFcalNe0}
\begin{aligned}
\Vec{\Fcal(A,v,\phi)} &\lto \Vec{\Fcal(A,v,\phi)^{\ne0}} \\
\Bun_G(\Fcal(A,v,\phi)) &\lto \Bun_G(\Fcal(A,v,\phi)^{\ne0}).
\end{aligned}
\end{equation}
\item\label{GBundlesZAdicFRameStack2}
A $G$-bundle on $\Fcal(A,v,\phi)^{\ne0}$ is in the essential image of \eqref{EqFcalNe0} if and only if it is \'etale locally isomorphic to $\Escr_{\mu}$.
\item\label{GBundlesZAdicFRameStack3}
Suppose that $A$ is a Dedekind domain, $v(L)$ a product of pairwise different maximal ideals, and $G$ is reductive. Then the functors \eqref{EqFcalNe0} are equivalences of categories.
\end{assertionlist} 
\end{corollary}

\begin{proof}
This follows by combining Proposition~\ref{FullyFaithfulness0Complement}, Corollary~\ref{RestrictGBunReAvne0}, and Lemma~\ref{GBundlesFrameStack}.
\end{proof}

Now fix as in Section~\ref{COLIMITREES} a $\phi$-invariant ideal $J \subseteq A$ such that $A$ is $J$-adically complete. Then we obtain for all $N \geq 1$ an induced frame on $A/J^N$ and its frame stack $\Fcal(A/J^N,v,\phi)$. Since colimits commute with coequalizers we obtain from Proposition~\ref{InftyHeckeAsColimit}, Example~\ref{AdicColimitAdam} and Theorem~\ref{LimitBunGTruncationofReesStacks} the following result, where $G$ is as in Section~\ref{COLIMITREES}.

\begin{corollary}\label{LimitBunGTruncFrame}
We have a natural isomorphism of Adams stacks
\[
\colim_{N<\infty}\Fcal(A/J^N,v,\phi) \liso \Fcal(A,v,\phi).
\]
In particular, one has a natural equivalence
\[
\Bun_G(\Fcal(A,v,\phi)) \liso \lim_{N<\infty}\Bun_G(\Fcal(A/J^N,v,\phi))
\]
for every smooth affine group scheme $G$ over $O$.
%\item
%For all $1 \leq N \leq N' \leq \infty$ the maps
%\[
%\Bun_G(\Fcal(A/J^{N'},v,\phi)) \lto \Bun_G(\Fcal(A/J^N,v,\phi))
%\]
%are full and essentially surjective.
%In particular they induce isomorphisms
%\[
%H^1(G,\Fcal(A/J^{N'},v,\phi)) \iso H^1(G,\Fcal(A/J^{N},v,\phi)).
%\]
\end{corollary}

%
%\begin{remark}\label{PullbackModule}
%Consider the following general construction. Let $M$ be an $A$-module with an $A$-linear map $v\colon M \to A$ and let $\psi\colon B \to A$ be a ring homomorphism. Then
%\[
%M \times_A B := M \times_{v,A,\psi} B := \set{(m,b) \in M \times B}{v(m) = \psi(b)}
%\]
%is a $B$-module via the scalar multiplication given by
%\[
%\beta\cdot (m,b) = (\psi(\beta)m, \beta b), \qquad \beta \in B.
%\]
%If $u\colon M' \to M$ is an $A$-linear map, it induces a $B$-linear map $M' \times_{v \circ u,A,\psi} B \lto M \times_{v,A,\psi} B$.
%\end{remark}
%
%Define the attached Nygaard filtration by
%\[
%\Fil^j(A,v,\phi) := L^{\otimes j} \times_{v^j,A,\phi} A.
%\]

%---------------------------------------------------------------------

\subsection{The case $v = 0$}\label{Sec:v0}

If $v = 0$, then $R = A$ and
\[
B = \Rees(A,0) = \Sym_A(L) \times_A A[t] = \set{(f,g) \in \Sym_A(L) \times A[t]}{f_0 = g(0)}.
\]
Hence $\Re(A,0)$ is the algebraic stack glued together from $\Re(A,0)^- = A^{\Fil}$ and $\Re(A,0)^+ = [\GG_{m,A}\backslash \Spec \Sym_A(L)]$ along the closed substack $\Re(A,0)^0 = \B{\GG_{m,A}}$.

Therefore a vector bundle over $\Re(A,0)$ is a triple $(M,N,\alpha)$ consisting of a vector bundle $M$ on $A$ with a decreasing filtration (Proposition~\ref{CharFilteredVB}), a vector bundle $N$ with an increasing $L$-filtration (see Remark~\ref{VBSymGGm}), and an isomorphism $\alpha\colon \gr(M) \iso \gr(N)$.

Now suppose that we are given a ring endomorphism $\phi\colon A \to A$ yielding a frame structure and an attached frame stack $\Fcal(A,0,\phi)$ as in Section~\ref{Sec:FrameAvphi}. In this case we obtain $\Fcal(A,0,\phi)$ by gluing the open substacks
\begin{align*}
\Re(A,0)^- \setminus \Re(A,0)^0 &= [\GG_{m,A}\backslash \Spec A[t,t^{-1}]] \cong \Spec A, \\
\Re(A,0)^+ \setminus \Re(A,0)^0 &= [\GG_{m,A}\backslash \Spec \Sym_A(L)] \cong \Spec A
\end{align*}
of $\Re(A,0)$ via $\phi$.

Therefore a vector bundle on $\Fcal(A,0,\phi)$ is given by a vector bundle $N$ with an increasing $L$-filtration $D$, a decreasing vector bundle filtration $C$ on $\phi^*N$ and an isomorphism $\alpha\colon \gr_C(\phi^*N) \iso \gr^D(N)$.

We can also consider the following variant of $\Fcal(A,0,\phi)$. Let $A^{\textup{$\phi$-BiZip}}$ the stack defined by the pushout square
\[\xymatrix{
\B{\GG_{m,A}} \ar[r] \ar[d]_{\phi} & A^{\Fil} \ar[d] \\
[\GG_{m,A}\backslash \Spec \Sym_A(\phi^*L)] \ar[r] & A^{\textup{$\phi$-BiZip}},
}\]
where the left vertical map is the composition of the map $\B{\GG_{m,A}} \to \B{\GG_{m,A}}$ induced by $\phi$ and the zero section $\B{\GG_{m,A}} \to [\GG_{m,A}\backslash \Spec \Sym_A(\phi^*L)]$. In $A^{\textup{$\phi$-BiZip}}$ there are two open substacks $\Spec A \subseteq A^{\Fil}$ and $\Spec A \subseteq [\GG_{m,A}\backslash \Spec \Sym_A(\phi^*L)]$ and we form the coequalizer of these two open immersions to obtain an Adams stack $A^{\textup{$\phi$-Zip}}$.

Then the canonical $\phi$-linear map $L \to \phi^*L$ induces a map $A^{\textup{$\phi$-BiZip}} \to \Re(A,0)$ and then by taking coequalizers a map
\begin{equation}\label{EqZipTrunc1}
A^{\textup{$\phi$-Zip}} \lto \Fcal(A,0,\phi).
\end{equation}
A vector bundle on $A^{\textup{$\phi$-Zip}}$ is given by a vector bundle $N$ endowed with an increasing $L$-filtration $D$ and a decreasing vector bundle filtration $C$ and an isomorphism $\alpha\colon \phi^*\gr_C(N) \iso \gr(N)$. Pullback under \eqref{EqZipTrunc1} sends $(N,C,D,\alpha)$ to $(N,\phi^*C,D,\alpha)$.

\section{Moduli of truncated equi-characteristic shtukas}\label{TRUNCHECKE}

In this section, we apply the above general result to study moduli spaces of truncated local shtukas. We fix the following notation. Let $\kappa$ be a field. From Section~\ref{Sec:TruncatedShtukaStack} on, $\kappa$ will be a finite field. We also fix a reductive group scheme $G$ over $O := \kappa\psz$. If $R$ is any $\kappa\psz$-algebra we denote by $G_R$ its base change to $R$. Note that $G$ is obtained via base change from a unique (up to isomorphism) reductive group over $\kappa$\footnote{Let $G' = G_{\kappa} \otimes_{\kappa} \kappa\psz$. Then $G$ and $G'$ have the same geometric based root datum therefore $G'$ is a form of $G$ for the \'etale topology. These forms of $G$ are classified by $H^1(\kappa\psz,\Autline(G))$, where $\Autline(G)$ is the group scheme of automorphisms of $G$. Since $\Autline(G)$ is ind-quasi-affine and smooth over $\kappa\psz$, we know that $H^1(\kappa\psz,\Autline(G)) \to H^1(\kappa,\Autline(G))$ is injective by \cite[2.1.7]{BC_Torsors} which shows $G' \cong G$.} which we also call $G$.

Let $N$ denote an element of $\ZZ_{\geq 0} \cup \{\infty\}$, which we endow with its usual total order. For every $\kappa$-algebra $R$ we set
\[
R_N = R\psz/(z^N) \quad\text{for $N < \infty$}, \qquad R_{\infty} := R\psz.
\]
Hence $R_0$ is the zero ring. For all $1 \leq N \leq \infty$ we endow $R_N$ with the $z$-adic filtration (Example~\ref{VAdicSpecialCases}~\ref{VAdicSpecialCases2}).

Define the positive loop group, resp.~the loop group as usually by
\begin{equation}\label{EqDefLoop}
L^+G(R) = G(R\psz), \qquad\qquad LG(R) = G(R\lsz).
\end{equation}
We have for their truncations
\[
L^{(N)}G(R) := G(R_N),
\]
i.e., $L^{(\infty)}G = L^+G$ and $L^{(N)}(G) = \Res_{\kappa_N/\kappa}(G)$ for $N < \infty$. Then
\[
L^+G = \lim_{N<\infty}L^{(N)}(G).
\]
We also set
\[
L_N(G) = \Ker(L^+G \to L^{(N)}(G)),
\]
i.e., $L_NG(R) = \set{g \in G(R\psz)}{g \sends 1 \in G(R_N)}$. In particular $L_{0}G := L^+G$. Hence we find $L^+G/L_N(G) = L^{(N)}(G)$.

%---------------------------------------------------------------------

\subsection{The $N$-truncated Hecke stack}

Let $R$ be a $\kappa$-algebra. We now apply the results of Section~\ref{Sec:ReesVAdic} to $v\colon R_N \ltoover{z} R_N$.

\begin{definition}\label{DefNTruncHecke}
For $1 \leq N \leq \infty$ we call
\[
R^{\Hck,N} := \Re(R_N, z) = [\GG_{m}\backslash \Spec (R_N[t,u]/(tu - z))]
\]
the \emph{$N$-truncated Hecke stack} over $R$. For $N = \infty$ we also set
\[
R^{\Hck} := \Re(R\psz,z)^{\ne0} \subseteq R^{\Hck,\infty}
\]
and call it the \emph{Hecke scheme} over $R$.
\end{definition}

\begin{remark}\label{DescribeGBunHck}
By Remark~\ref{AttractorReAv}, the scheme $R^{\Hck}$ is obtained by gluing two copies of $\Spec R\psz$ along $\Spec R\lsz$. Therefore the groupoid $\Bun_G(R^{\Hck})$ is equivalent to the groupoid of triples $(\Escr_1,\Escr_2,\alpha)$, where $\Escr_1$ and $\Escr_2$ are $G$-bundles on $R\psz$ and where $\Phi\colon \Escr_2[1/z] \iso \Escr_1[1/z]$ is an isomorphism of $G$-bundles over $R\lsz$. One has a similar description for the exact category of vector bundles on $R^{\Hck}$. This explains the terminology.

By Remark~\ref{ReAne0ResolutionProp}, $R^{\Hck}$ is semi-separated (but not separated) and it has an ample pair of line bundles. In particular, it satisfies the resolution property.
\end{remark}

Proposition~\ref{FullyFaithfulness0Complement} and Corollary~\ref{RestrictGBunReAvne0} yield the following result which will be essential for us in the sequel.

\begin{proposition}\label{BunGHckbar}
Restriction yields fully faithful functors
\begin{equation}\label{EqRestrictToHecke}
\begin{aligned}
\Vec(R^{\Hck,\infty}) &\lto \Vec(R^{\Hck}), \\
\Bun_G(R^{\Hck,\infty}) &\lto \Bun_G(R^{\Hck})
\end{aligned}
\end{equation}
which is an equivalence if $R$ is a field. The essential image for $G$-bundles of \eqref{EqRestrictToHecke} are those $G$-bundles $\Escr$ on $\Bun_G(R^{\Hck})$ for which there exists a finite separable extension $\kappa'$ of $\kappa$ and a conjugacy class of cocharacter $[\mu]$ of $G$ defined over $\kappa'\psz$ such that $\Escr$ is \'etale locally on $R$ isomorphic to $\Escr_{[\mu]}$ (Definition~\ref{DefineGModMu}).
\end{proposition}

See also Proposition~\ref{VBEssentialImage} for a description of the essential image of \eqref{EqRestrictToHecke} for vector bundles.

By Theorem~\ref{LimitBunGTruncationofReesStacks} we have
\begin{equation}\label{EqLimNBunGSht}
\Bun_G(R^{\Hck,\infty}) = \lim_{N<\infty} \Bun_G(R^{\Hck,N})
\end{equation}

If $R \to R'$ is a map of $\kappa$-algebras, it induces a map $R^{\prime\Hck,N} \to R^{\Hck,N}$ for all $1\leq N \leq \infty$ and we obtain a prestack over $\kappa$
\begin{equation}\label{EqDefBunlineHckbar}
\Bunline_G(\Hck,N)\colon \Alg{\kappa} \lto \Grpd, \qquad R \sends \Bun_G(R^{\Hck,N}).
\end{equation}

\begin{remark}\label{AlternativeBunline}
For $N < \infty$ we have $R_N \otimes_R R' = R'_N$ and hence $R^{\Hck,N} \otimes_R R' = R^{\prime\Hck,N}$. Therefore we find
\[
\Bunline_G(\Hck,N) = \Homline(\kappa^{\Hck,N},\B{G})
\]
in this case.
\end{remark}

Similarly, we obtain the \emph{(local) $G$-Hecke stack}
\[
\Bunline_G(\Hck)\colon \Alg{\kappa} \lto \Grpd, \qquad R \sends \Bun_G(R^{\Hck}).
\]
In fact, it is well known is a stack for the fpqc topology. This will also follow from Proposition~\ref{BunGHeckeAlgebraic}, where we will see that $\Bunline_G(\Hck,N)$ is a stack for the fpqc topology for all $N$.

The functors \eqref{EqRestrictToHecke} are functorial in $R$ and define a map of stacks
\begin{equation}\label{EqBunGHecke}
\Bunline_G({\Hck,\infty}) \lto \Bunline_G(\Hck).
\end{equation}
which is fully faithful on $R$-valued points. It is an equivalence of categories on $k$-valued points for every field extension $k$ of $\kappa$.

\begin{proposition}\label{BunGHeckeAlgebraic}
For $N < \infty$, the prestacks $\Bunline_G({\Hck,N})$ are algebraic stacks with affine diagonal. Moreover, $\Bunline_G({\Hck,N})$ are stacks for the fpqc topology for all $N \leq \infty$.
\end{proposition}

\begin{proof}
Let $N < \infty$. Recall that in this case $\Bunline_G({\Hck,N}) = \Homline(\kappa^{\Hck,N},\B{G})$. Now $\kappa^{\Hck,N}$ is an algebraic stack that is cohomologically projective over $\kappa$ by \cite[4.2.5]{HLP_MappingStacks} and hence formally proper over $\kappa$ by \cite[4.2.1]{HLP_MappingStacks}. Therefore \cite[5.1.1]{HLP_MappingStacks} implies that $\Bunline_G({\Hck,N})$ is an algebraic with affine diagonal. Since its diagonal is affine, it is a stack for the fpqc topology.

By \eqref{EqLimNBunGSht} we have $\Bunline_G({\Hck,\infty}) = \lim_{N<\infty} \Bunline_G({\Hck,N})$ which implies that $\Bunline_G({\Hck,\infty})$ is also a stack for the fpqc topology.
\end{proof}

Below in \eqref{EqDescribeBunGMuHckbar}, we will describe $\Bunline_G({\Hck,N})$ explicity as a disjoint sum of quotient stacks. This will give a different proof of Proposition~\ref{BunGHeckeAlgebraic}.

%---------------------------------------------------------------------

\subsection{$G$-bundles of type $[\mu]$ over the $R^{\Hck,N}$}\label{TYPEHECKE}

For $1 \leq N \leq \infty$, denote by $i\colon (R^{\Hck,N})^0 = \B{\GG_{m,R}} \to R^{\Hck,N}$ the closed embedding of the fixed point locus. Pullback with $i$ defines a map of groupoids $\Bun_G(R^{\Hck,N}) \to \Bun_G(\B{\GG_{m,R}})$ that is functorial in $R$. Hence we obtain a map of stacks over $\kappa$
\begin{equation}\label{EqTypeHecke}
\type\colon \Bunline_G({\Hck,N}) \lto \Bunline_G(\B{\GG_{m,\kappa}}).
\end{equation}
For every $\kappa$-algebra $R$ the induced map of groupoids
\[
\Bunline_G({\Hck,N})(R) \lto \Bunline_G(\B{\GG_{m,\kappa}})(R)
\]
is essentially surjective and full, in particular, it induces a bijection on isomorphism classes (Theorem~\ref{LimitBunGTruncationofReesStacks}). We obtain for $1 \leq N \leq N' \leq \infty$ a commutative diagram of fpqc stacks
\begin{equation}\label{EqTypeGerbe}
\begin{aligned}\xymatrix{
\Bunline_G(\Hck,N') \ar[rr] \ar[rd]_{\type} & & \Bunline_G({\Hck,N}) \ar[dl]^{\type} \\
& \Bunline_G(\B{\GG_{m,\kappa}}).
}\end{aligned}
\end{equation}

Let $\kgbar$ be an algebraic closure of $\kappa$, let $(X^*,\Phi,X_*,\Phi\vdual,\Delta)$ be the geometric based root datum of $G$, and let $W_G$ be its Weyl group (Section~\ref{Sec:StackTypesGBundles}). Then the decomposition \eqref{EqDecomposeBunGBGGm} of $\Bunline_G(\B{\GG_{m,\kappa}})$ yields via taking inverse images under the type morphism a decomposition into open and closed substacks
\[
\Bunline_G({\Hck,N})_{\kappa'} = \coprod_{[\mu] \in X_*/W_G}\Bunline^{[\mu]}_G({\Hck,N})
\]
over some finite separable extension $\kappa'$ of $\kappa$, for instance a splitting field of $G$. More precisely, each open and closed substack $\Bunline^{[\mu]}_G({\Hck,N})$ is defined over the field of definition $\kappa_{[\mu]} \subseteq \kgbar$ of the conjugacy class $[\mu]$.

Let us give a description of $\Bunline^{[\mu]}_G({\Hck,N})$ as a moduli problem. For this recall the standard $G$-bundle $\Escr_{[\mu]}$ of type $[\mu]$ on $\B{\GG_{m,\kappa_{[\mu]}}}$ defined in Definition~\ref{DefineGModMu}. By definition $\Bunline^{[\mu]}_G(\B{\GG_{m,\kappa_{[\mu]}}})$ consists of those $G$-bundles that are \'etale locally isomorphic to $\Escr_{[\mu]}$.
%By Proposition~\ref{BunlineMuStack} is a neutral gerbe over $\kappa_{[\mu]}$ which is isomorphic to the classifying stack $\B{\Cent_G(\mu)}$ over a finite Galois extension over which the conjugacy class of $[\mu]$ can be represented by a cocharacter $\mu\colon \GG_m \to G$.
Hence Theorem~\ref{LimitBunGTruncationofReesStacks} implies the following result.

\begin{proposition}\label{HeckeMu}
Let $R$ be a $\kappa_{[\mu]}$-algebra. Then a $G$-bundle $\Escr$ over $R^{\Hck,N}$ defines an $R$-valued point of $\Bunline^{[\mu]}_G({\Hck,N})$ if and only if it is \'etale locally on $\Spec R$ isomorphic to the pullback of $\Escr_{[\mu]}$ to $R^{\Hck,N}$ under the canonical map $R^{\Hck,N} \to \B{\GG_{m,R}} \to \B{\GG_{m,\kappa_{[\mu]}}}$.
\end{proposition}

\begin{remark}\label{SchubertCell}
We can also use the above formalism to give a moduli theoretic description of Schubert cells in the affine Grassmannian of $G$:

The monomorphism $\Bunline_G(\Hck,\infty) \lto \Bunline_G(\Hck)$ induces an isomorphism of $\Bunline^{[\mu]}_G(\Hck,\infty)$ with a locally closed substack $\Bunline^{[\mu]}_G(\Hck)$ whose $R$-valued points can be described as those $G$-bundles on $R^{\Hck}$ that are \'etale locally isomorphic to $\Escr_{[\mu]}$.

Then $\Bunline_G(\Hck)$ can also be described as the \'etale quotient stack of the loop group
\[
\Bunline_G(\Hck) = [L^+G\backslash LG/L^+G].
\]
The affine Grassmannian of $G$ is defined as the quotient
\[
\Gr_G := LG/L^+G.
\]
It parametrizes $G$-bundles $\Escr$ over $R\psz$ together with a trivialization of $\Escr$ over $R\lsz$. Equivalently, it parametrizes  $G$-bundles over $R^{\Hck}$, described as triples $(\Escr_0,\Escr,\alpha)$, where $\Escr_0$ is the trivial $G$-bundle.

The projection $\Gr_G \to \Bunline_G(\Hck)$ makes $\Gr_G$ into an $L^+G$-torsor over $\Bunline_G(\Hck)$ and we denote the preimage of $\Bunline^{[\mu]}_G(\Hck)$ under this projection by $\Gr_G^{[\mu]}$. It para\-metrizes $G$-bundles over $R^{\Hck}$ as above which are \'etale locally on $R$ isomorphic to $\Escr_{[\mu]}$. This gives a moduli theoretic description of the Schubert cell $\Gr_G^{[\mu]}$.
\end{remark}

\begin{remark}\label{DescribeHckGmuAlt}
Let $\kappa'$ be a finite separable extension of $\kappa$ such that the conjugacy class $[\mu]$ can be represented by a cocharacter $\mu$ of $G_{\kappa'\psz}$. For a $\kappa'$-algebra $R$, the groupoid $\Bun^{[\mu]}_G(\Hck)$ can then also be described as the full subgroupoids of triples $(\Escr_1,\Escr_2,\alpha) \in \Bun_G(R^{\Hck})$ in $\Bun_G(R^{\Hck})$ such that there exists a an \'etale
covering $\Spec R' \to \Spec R$ yielding $\pi\colon \Spec R'\psz \to \Spec R\psz$ such that $\pi^*\Escr_1$ and $\pi^*\Escr_2$ are trivial $G$-bundles and that via some (equivalently, any) trivializations $\pi^*\Escr_1 \cong G_{R'\psz} \cong \pi^*\Escr_2$, the isomorphism $\alpha_{R'\psz}$ is given by $g \sends hg$ for an element $h \in G({R'\psz})\mu(z)G({R'\psz}) \subseteq G(R'\lsz)$.
\end{remark}

We will now describe $\Bunline^{[\mu]}_G(\Hck,N)$ more explicitly as quotient stack. We use the constructions of Section~\ref{EGENERAL}. We define the display group, i.e., the group presheaf on the category of $\kappa_{[\mu]}$-algebras
\begin{equation}\label{EqDefENMuPreSheaf}
E_N(G,[\mu])\colon R \sends \Hom_{\B{\GG_{m}}}(R^{\Hck,N}, [\mu]\backslash G_{\kappa_{[\mu]}})
\end{equation}
see \eqref{EqDefEXcalMu}. 

\begin{remark}\label{EFunctorialN}
For $N \leq N' \leq \infty$, the closed immersions $R^{\Hck,N} \to R^{\Hck,N'}$ induces maps of group schemes
\begin{equation}\label{EqTruncationE}
E_{N'}(G,[\mu]) \lto E_N(G,[\mu]).
\end{equation}
\end{remark}

%More generally, let $\kappa'$ be a finite separable extension of $\kappa$ such that $[\mu]$ can be represented by a cocharacter $\mu$ of $G$ over $\kappa'\psz$. Then $\mu$ defines a $\GG_m$-action on $G$. Let $H$ be a $\GG_m$-invariant smooth affine subgroup scheme of $G_{\kappa'\psz}$. We obtain a corresponding display group presheaf $E_N(H,\mu)$.

%The $\GG_{m,\kappa'\psz}$-action on $H_{\kappa'\psz}$ defines a $\ZZ$-grading on $\Gamma(H,\Oscr_H)_{\kgbar\psz}$.
%
%By \eqref{EqAltDescribeEMu} we have
%\begin{equation}\label{EqDescribeENMu}
%E_N(H,\mu)(R) = \{\text{$\kbar\psz$-algebra homomorphisms $\Gamma(H,\Oscr_H)_{\kgbar\psz} \lto A_N[t,u]/(tu-z)$ of degree $0$}\}.
%\end{equation}

\begin{proposition}\label{ENMuRep}
\begin{assertionlist}
\item
For $N < \infty$, the presheaf $E_N(G,[\mu])$ is representable by a smooth affine group scheme over $\kappa_{[\mu]}$. 
\item
One has $E_{\infty}(G,[\mu]) = \lim_{N<\infty}E_N(G,[\mu])$, in particular $E_{\infty}(H,[\mu])$ is representable by an pro-smooth affine group scheme over $\kappa_{[\mu]}$.
\item
For all $N \leq \infty$ one has
\begin{equation}\label{EqDescribeBunGMuHckbar}
\Bunline^{[\mu]}_G(R^{\Hck,N}) \cong \B{E_N(G,[\mu])}
\end{equation}
\end{assertionlist}
\end{proposition}

\begin{proof}
We use Proposition~\ref{EMuRepresentable}. With the notation of loc.~cit.~one has $(B_R)_i = R\psz/(z^N)$ for all $i \in \ZZ$. The functor $R \sends R\psz/(z^N)$ is representable by a product of $N$-copies of $\AA_R^1$. Therefore $E_N(G,[\mu])$ is representable by an affine group scheme for all $N$ and it is of finite type for $N < \infty$ by Proposition~\ref{EMuRepresentable}~\ref{EMuRepresentable1}. Since $R^{\Hck,\infty} = \colim_{N<\infty}R^{\Hck,N}$ by Proposition~\ref{InftyHeckeAsColimit}, we have $E_{\infty}(G,[\mu]) = \lim_{N<\infty}E_N(G,[\mu])$. For $N < \infty$, $E_N(G,[\mu])$ is smooth by Proposition~\ref{EMuRepresentable}~\ref{EMuRepresentable2}. Finally, \eqref{EqDescribeBunGMuHckbar} is a special case of Proposition~\ref{GeneralBunlineGMu}.
\end{proof}

\begin{definition}\label{DefDisplayShtuka}
The smooth affine group scheme $E_N(G,[\mu])$ is called the \emph{$N$-truncated $(G,[\mu])$-display group}.
\end{definition}

Next we will describe the affine group schemes $E_N(G,[\mu])$ more explicitly. To simplify notation, we assume from now on that $[\mu]$ can be represented by a cocharacter $\mu$ of $G$ over the reflex field $\kappa_{[\mu]}$. This is always the case if $G$ is quasi-split, for instance if $\kappa$ is a finite field as will be the case starting with the next section anyway.

\begin{proposition}\label{DisplayGroupInfty}
\begin{assertionlist}
\item
For every $\kappa$-algebra $R$ one has
\begin{equation}\label{EqEInftyGMu}
E_{\infty}(G,\mu)(R) = \set{g \in G(R\psz)}{\mu(z)g\mu(z)^{-1} \in G(R\psz) \subseteq G(R\lsz)},
\end{equation}
i.e., $E_{\infty}(G,\mu) = L^+G \cap \mu^{-1}L^+G\mu \subseteq LG$.
\item
For $1 \leq N < \infty$ one has
\begin{equation}\label{EqENGMu}
E_N(G,\mu) = E_{\infty}(G,\mu)/(L_NG \cap \mu^{-1}L_NG\mu).
\end{equation}
\end{assertionlist}
\end{proposition}
%see \cite[5.2.3]{Ito_PrismaticGDisplays}.

\begin{proof}
Let $E'(G,\mu)$ be the display group associated to $R \sends \Spec R\lsz$. By Example~\ref{LaurentE} we have $E'(G,\mu) = LG$. The open immersions $\Spec R\lsz \to R^{\Hck} \to R^{\Hck,\infty}$ define by Remark~\ref{EFunctorialInX} a monomorphism $E_{\infty}(G,\mu) \to LG$. A straight forward calculation using the first equality of \eqref{EqAltDescribeEMu} shows that the image is precisely as described in \eqref{EqEInftyGMu}. The second assertion follows then from Proposition~\ref{EFunctorialSurjective}.
\end{proof}

There is also a different description of $E_1(G,\mu)$ using the groups $P^{\pm}(\mu)$ and $U^{\pm}(\mu)$ from Remark~\ref{GGmActionGroupScheme}. To ease notation, we simply write $\Cent_G(\mu)$ instead of $\Cent_{G_{\kappa_{[\mu]}}}(\mu)$.

\begin{proposition}\label{DescribeE1Mu}
The 1-truncated $(G,\mu)$-display group $E_1(G,\mu)$ is given by the smooth affine group scheme $\Cent_G(\mu) \rtimes (U^-(\mu) \times U^+(\mu))$. In particular, $E_1(G,\mu)$ is a smooth connected algebraic group over $\kappa_{[\mu]}$ of dimension $\dim G$.
\end{proposition}

This description can be deduced from \eqref{EqENGMu}, but we will give a direct proof here.

\begin{proof}
The stack $R^{\Hck,1}$ is obtained by glueing $[\GG_{m,R}\backslash (\AA_R^1)^-]$ and $[\GG_{m,R}\backslash (\AA_R^1)^+]$ along $\B{\GG_{m,R}} = [\GG_{m,R}\backslash\{0\}]$ (Section~\ref{Sec:v0}). Therefore we find by Example~\ref{BunGMuAA1GGm}
\[
\Bunline_G^{\mu}(\Hckbar{1}) = \B{P^-(\mu)} \times_{\B{\Cent_G(\mu)}} \B{P^+(\mu)} = \B{P^-(\mu) \times_{\Cent_G(\mu)} P^+(\mu)},
\]
where the last equality holds by Proposition~\ref{LimitClassifying} below. Since $P^{\pm}(\mu) = \Cent_G(\mu) \rtimes U^{\pm}(\mu)$, this implies the claim.
\end{proof}

\begin{proposition}\label{LimitClassifying}
Let $\Ccal$ be a site, let $\Ical$ be a finite category, and let $i \sends G_i$ be an $\Ical$-diagram of sheaves of groups in $\Ccal$ such that for every arrow $i \to j$ the map $G_i \to G_j$ is surjective. Set $G := \lim_i G_i$. Then the canonical map of stacks in groupoids
\[
\Psi\colon \B{G} \lto \lim_i \B{G_i}
\]
is an equivalence.
\end{proposition}

\begin{proof}
By \cite[2.1]{Bieker_IntegralModelsDeepBT}, $\Psi$ is fully faithful. Let $(E_i)_{i\in I}$ be in $\lim_i \B{G_i}$. By loc.~cit.~it suffices to show that the $G$-pseudo-torsor $\lim_i E_i$ has locally a section. But this follows easily from the hypothesis that all transition maps are epimorphisms: Let $U$ be an object in $\Ccal$. Since the number of objects in $\Ical$ is finite, we find a covering $(U_{\alpha} \to U)_{\alpha}$ such that $E_i(U_{\alpha}) \ne \emptyset$ for all $i$. Since the number of morphisms in $\Ical$ is finite and since all transition maps are epimorphisms, we find, after passing to a refinement of $(U_{\alpha} \to U)$, even an element in $E(U_{\alpha}) = \lim_i E_i(U_{\alpha})$ for all $\alpha$.
\end{proof}

%\begin{remark}\label{EFunctorial}
%Let $H \subseteq H'$ be $\GG_m$-invariant subschemes of $G$. Then $[\mu\backslash H] \to [\mu\backslash H']$ is a monomorphism and hence $E_N(H,\mu) \lto E_N(H',\mu)$ is a monomorphism (and hence a closed immersion) of group schemes.
%\end{remark}
%
%
%\begin{proposition}\label{DecomposeE}
%The multiplication map $E_{N}(P^-(\mu),\mu) \times E_N(U^+(\mu),\mu) \to E_N(G,\mu)$ is an isomorphism of $\kappa$-schemes.
%\end{proposition}
%
%Note that $E_N(U^+(\mu),\mu)$ does usually not normalize $E_{N}(P^-(\mu),\mu)$, so the isomorphism will not be an isomorphism of group schemes.
%
%\begin{proof}
%General fact: Holds for any Rees stack such that $\Ker(\Rees(A)_0 \to R)$ is contained in the Jacobson radical, same proof as in \cite[6.2.2]{Lau_HigherFrames}.
%\end{proof}
%

Let $\Escr$ be any $G$-bundle of type $[\mu]$ on $R^{\Hck,N}$. By Proposition~\ref{GeneralBunlineGMu} we see that $E_N(G,[\mu])$ is a strong inner form of $\Autline(\Escr)$. In particular $E_N(G,[\mu]) \cong \Autline(\Escr)$ if $R$ is a separably closed field or a finite field.

\begin{example}\label{EXGmuViaAutomorphism}
For $G = \GL_n$ the calculations in \cite[3.4]{Lau_HigherFrames} carry over verbatim: In this case $[\mu]$ is given by a tuple $(\mu_1,\dots,\mu_n) \in \ZZ^n$ with $\mu_1 \geq \mu_2 \geq \cdots \geq \mu_n$. Then for every local $\kappa$-algebra $R$ any $\GL_n$-bundle $\Escr$ of type $[\mu]$ on $R^{\Hck,N}$ is of the form
\[
\Escr = \bigoplus_i\Oscr_{R^{\Hck,N}}(-\mu_i)
\]
and $\Endline(\Escr)(R)$ can be identified with the ring of matrices $A = (a_{ij})$ with
\[
a_{ij} \in \Hom_{R^{\Hck,N}}(\Oscr(-\mu_j),\Oscr(-\mu_i)) \cong \Rees(R_N,z)_{\mu_j-\mu_i} \cong R\psz/(z^N).
\]
The entries with $\mu_i = \mu_j$ form a sequence of square matrices $A_1,\dots,A_r$ along the diagonal, say of size $n_1,\dots,n_r$ with entries in $\Rees(R_N,z)_0 = R_N$ and an endomorphism given by a matrix $A$ is an automorphism if and only if $A_i \in M_{n_i}(R_N)$ is invertible (or, equivalently, the image of $A_i$ in $M_{n_i}(R)$ is invertible).

This shows in particular that $E_{N+1}(\GL_n,[\mu]) \to E_{N}(\GL_n,[\mu]))$ is surjective and that its kernel is a vector group of dimension $n^2 = \dim(\GL_n)$.
\end{example}

The last assertion holds in general.

\begin{proposition}\label{KernelTruncationE}
For all $1 \leq N < \infty$, the truncation map $E_{N+1}(G,[\mu]) \to E_N(G,[\mu])$ is a surjective homomorphism of smooth algebraic groups. Its kernel is a vector group of dimension $\dim(G)$.
\end{proposition}

\begin{proof}
The truncation map $E_{N+1}(G,[\mu]) \to E_N(G,[\mu])$ is surjective on $R$-valued points for every $\kappa$-algebra $R$ by Proposition~\ref{EFunctorialSurjective}. By \eqref{EqENGMu}, its kernel is given by
\[
(L_NG \cap \mu^{-1}L_NG\mu)/(L_{N+1}G \cap \mu^{-1}L_{N+1}G\mu),
\]
which is a vector group of dimension $\dim(G)$.
\end{proof}

\begin{corollary}\label{DimE}
\begin{assertionlist}
\item\label{DimE1}
For all $1 \leq N < \infty$, the display group $E_N(G,[\mu])$ is a smooth connected affine group scheme of dimension $N\dim(G)$. Its maximal reductive quotient is isomorphic to $\Cent_{G}(\mu)$ and its unipotent radical is a successive extension of vector groups. 
%Therefore $E_{\infty}(G,[\mu]) = \lim_{N<\infty}E_N(G,[\mu])$ is a pro-smooth affine group scheme.
\item\label{DimE2}
For all $1 \leq N \leq \infty$, every $E_N(G,[\mu])$-bundles for the fpqc topology is already \'etale locally trivial.
\end{assertionlist}
\end{corollary}

\begin{proof}
In Proposition~\ref{DescribeE1Mu} we saw the results for $N=1$ and we conclude for $N < \infty$ by Proposition~\ref{KernelTruncationE}. The last assertion of \ref{DimE1} follows from Proposition~\ref{ENMuRep}.

Assertion~\ref{DimE2} holds for $N < \infty$ since $E_N(G,[\mu])$ is smooth. It holds then for $N = \infty$ by \cite[A.4.8]{RS_IntersectionMotive} using \ref{DimE1}.
\end{proof}

\subsection{The truncated shtuka stack}\label{Sec:TruncatedShtukaStack}

From now on let $p$ be a prime number, and let $\kappa = \FF_q$ be a finite field of characteristic $p$ with $q$ elements.

For every $\FF_q$-algebra $R$ and for all $1\leq N \leq \infty$ we define
\[
\phi\colon R_N \to R_N, \qquad \sum_{i\geq0}a_iz^i \sends \sum_{i\geq0}a_i^qz^i
\]
and let
\[
\sigma\colon \Spec R_N \to R^{\Hck,N}
\]
be the induced map as in \eqref{EqDefSigmaVAdicFrame}.

We obtain a frame $(R_N, z, \sigma)$ and morphisms
\begin{align*}
\tau\colon &\Spec R_N \iso (R^{\Hck,N})^{\ne+} \mono R^{\Hck,N}, \\
\sigma\colon &\Spec R_N \ltoover{\phi} \Spec R_N \iso (R^{\Hck,N})^{\ne-} \mono R^{\Hck,N}.
\end{align*}

\begin{definition}\label{DefTruncatedShtukaStack}
Let $R$ be an $\FF_q$-algebra and $1 \leq N \leq \infty$. The frame stack attached to the frame $(R_N,z,\sigma)$ is called the \emph{$N$-truncated shtuka stack over $R$}. It is denoted by $R^{\Sht,N}$, i.e.,
\[
R^{\Sht,N} = \colim (\xymatrix{\Spec R_N
\ar@<.5ex>[r]^-(.4){\tau}\ar@<-.5ex>[r]_-(.4){\sigma} & R^{\Hck,N}})
\]
We also set
\[
R^{\Sht} := (R^{\Sht,\infty})^{\ne 0} = \colim(\xymatrix{\Spec R\psz \ar@<0.5ex>[r]^{\tau} \ar@<-0.5ex>[r]_{\sigma} & R^{\Hck}}).
\]
and call it the \emph{shtuka stack over $R$}.
\end{definition}

\begin{remdef}\label{VecTruncShtukaStack}
By definition a vector bundle over $R^{\Sht,N}$ is the same as a vector bundle $\Escr$ over $R^{\Hck,N}$ together with an isomorphism $\sigma^*\Escr \iso \tau^*\Escr$ over $\Spec R_N$. A similar description holds for $G$-bundles, i.e., one has
\begin{equation}\label{EqBunGTrSht}
\Bun_G(R^{\Sht,N}) = \lim (\xymatrix{\Bun_G(R^{\Hck,N}) \ar@<0.5ex>[r]^{\tau^*} \ar@<-0.5ex>[r]_{\sigma^*} & \Bun_G(\Spec R_N)}).
\end{equation}
\end{remdef}

This yields indeed a geometrization (in the sense of the introduction) for local shtukas:

\begin{proposition}\label{VecShtukaStack}
The category of vector bundles on $\Sht_R$ is equivalent to the category of pairs $(\Escr,\Phi)$ consisting of a vector bundle $\Escr$ on $R\psz$ and an isomorphism of vector bundles $\Phi\colon \phi^*(\Escr)[1/z] \iso \Escr[1/z]$ over $R\lsz$.

One has an analogous description for the groupoid of $G$-bundles on $\Sht_R$. In other words, $\Bun_G(R^{\Sht})$ is the groupoid of local $G$-shtukas over $R$.
\end{proposition}

\begin{proof}
A vector bundle on $\Sht_R$ is the same as a vector bundle $\Escr'$ on $R^{\Hck}$ together with an isomorphism $\alpha\colon \sigma^*\Escr' \iso \tau^*\Escr'$. By Remark~\ref{DescribeGBunHck}, $\Escr'$ itself corresponds to a triple $(\Escr_1,\Escr_2,\Phi)$, where $\Escr_1$ and $\Escr_2$ are vector bundles on $\Spec R\psz$ and $\Phi$ is an isomorphism $\Escr_2[1/z] \iso \Escr_1[1/z]$. Then $\alpha$ is an isomorphism $\phi^*\Escr_2 \iso \Escr_1$ and sending $(\Escr,\Phi)$ to the tuple $(\phi^*\Escr_2,\Escr_2,\Phi,\id)$ defines the desired equivalence.

The proof for $G$-bundles is the same.
\end{proof}

For $1 \leq N \leq N' \leq \infty$ one has a commutative diagram
\[\xymatrix{
\Spec R_N \ar[r] \ar[d] & R^{\Hck,N}_R \ar[d] \\
\Spec R_{N'} \ar[r] & R^{\Hck,N'},
}\]
where the horizontal maps are either the maps $\sigma$ or the maps $\tau$ and where the vertical maps are the canonical closed immersions. Hence one obtains a map
\begin{equation}\label{EqTransitionShtuka}
R^{\Sht,N} \lto R^{\Sht,N'}.
\end{equation}

As a special case of Corollary~\ref{GBundlesZAdicFRameStack} and Corollary~\ref{LimitBunGTruncFrame} we obtain the following result.

\begin{theorem}\label{ColimitNTruncShtuka}
Let $R$ be a $\kappa$-algebra.
\begin{assertionlist}
\item\label{ColimitNTruncShtuka1}
In the 2-category of Adams stacks the maps \eqref{EqTransitionShtuka} induce an isomorphism
\[
\colim_{N<\infty}R^{\Sht,N} \liso R^{\Sht,\infty}.
\]
In particular we have
\[
\Bun_G(R^{\Sht,\infty}) \liso \lim_{N<\infty}\Bun_G(R^{\Sht,N})
\]
\item\label{ColimitNTruncShtuka2}
Pullback via the inclusion $R^{\Sht} \mono R^{\Sht,\infty}$ yields a fully faithful functor
\begin{equation}\label{EqGShtukasInfty}
\Bun_G(R^{\Sht,\infty}) \lto \Bun_G(R^{\Sht}).
\end{equation}
Its essential image consists of those $G$-bundles that are \'etale locally on $R$ isomorphic to the $G$-bundle $\Escr_{[\mu]}$ attached to the conjugacy class of $\mu$.

The functor \eqref{EqGShtukasInfty} is an equivalence if $R$ is a field.
\end{assertionlist}
\end{theorem}

%\begin{proof}
%Since colimits commute with coequalizers, this follows formally from the isomorphisms $\colim_{N<\infty}R^{\Hck,N}_R \iso R^{\Hck,\infty}_R$ (Proposition~\ref{InftyHeckeAsColimit}) and $\colim_{N<\infty}\Spec A_N \iso \Spec A_{\infty}$ (Example~\ref{AdicColimitAdam}).
%\end{proof}

We can now define the moduli stack of truncated $G$-bundles as follows.

\begin{definition}\label{DefModuliTruncShtukas}
For all $1 \leq N \leq \infty$ and every $\kappa$-algebra $R$ we define
\[
\Bunline_G(\Sht,N)(R) := \Bun_G(R^{\Sht,N}), \qquad \Bunline_G(\Sht)(R) := \Bun_G(R^{\Sht}).
\]
\end{definition}

The prestack $\Bunline_G(\Sht,N)$ is an fpqc-stack by \eqref{EqBunGTrSht} and Proposition~\ref{BunGHeckeAlgebraic}. By Proposition~\ref{VecShtukaStack}, $\Bunline_G(\Hck)$ is the moduli stack of local $G$-shtukas and hence also an fpqc-stack.

Again, we have a decomposition according to the type over some splitting extension $\kappa'$ of $\kappa$
\begin{equation}\label{EqDecompositionModuliShtuka}
\Bunline_G(\Sht,N)_{\kappa'} \cong \coprod_{[\mu] \in X_*/W} \Bunline^{[\mu]}_G(\Sht,N)_{\kappa'},
\end{equation}
and each open and closed substack $\Bunline^{[\mu]}_G(\Sht,N)_{\kappa'}$ is defined over the reflex field $\kappa_{[\mu]}$ of the conjugacy class $[\mu]$.

Let us describe the stack $\Bunline^{[\mu]}_G(\Sht,N)$ as a quotient stack. As in \eqref{EqGenActionEonG} we obtain for all $1 \leq N \leq \infty$ maps of affine group schemes over $\kappa$
\[
\tau_N,\sigma_N\colon E_N(G,[\mu]) \lto L^{(N)}G.
\]
Using the description in Proposition~\ref{DisplayGroupInfty}, $\tau_{\infty}$ is given by the inclusion $L^+G \cap \mu^{-1}L^+G\mu \mono L^+G$ and $\sigma_{\infty}$ is given by
\[
\sigma_{\infty}(g) = \phi(\mu(z)g\mu(z)^{-1}).
\]
For $1 \leq N < \infty$, functoriality shows that $\tau_N$ is induced by
\[
L^+G \cap \mu^{-1}L^+G\mu \ltoover{\tau_{\infty}} L^+G \lto L^+G/L_NG = L^{(N)}(G),
\]
and $\sigma_N$ is induced by
\[
L^+G \cap \mu^{-1}L^+G\mu \ltoover{\sigma_{\infty}} L^+G \lto L^+G/L_NG = L^{(N)}(G.
\]
For $1 \leq N \leq \infty$, we obtain an action of the affine flat group scheme $E_N(G,[\mu])$ on the affine flat scheme $L^{(N)}(G)$ by
\[
(e,g) \sends \tau_N(e)g\sigma_N(e).
\]
By our general results about $G$-bundles on frame stacks, we obtain now immediately the following description of the moduli space of $N$-truncated shtukas.

\begin{theorem}\label{DescribeModSpaceTruncShtukas}
For all $1 \leq N \leq \infty$ we have an equivalence of fpqc-stacks
\[
\Bunline^{[\mu]}_G(\Sht,N) \cong [E_N(G,[\mu])\backslash L^{(N)}(G)].
\]
For $N < \infty$, $\Bunline^{[\mu]}_G(\Sht,N)$ is a smooth algebraic stack of relative dimension $0$ over $\kappa_{[\mu]}$ with geometric connected components.
\end{theorem}

\begin{proof}
The description of $\Bunline^{[\mu]}_G(\Sht,N)$ as a quotient stack is a special case of Theorem~\ref{DescribeBunlineFramStack}. For $N < \infty$, $L^{(N)}(G)$ is a smooth (geometrically) connected affine scheme of dimension $N\dim(G)$ over $\kappa$. The same holds for the group schemes $E_N(G,[\mu])$ by Corollary~\ref{DimE}. Therefore $\Bunline^{[\mu]}_G(\Sht,N)$ is a geometrically connected smooth algebraic stack of relative dimension $0$.
\end{proof}

\begin{corollary}\label{DescribePointsShtuka}
For all $1 \leq N \leq \infty$ and for every strictly henselian local $\kappa$-algebra $R$ we have a functorial bijection
\[
\Bunline^{[\mu]}_G(\Sht,N)(R) \cong [E_N(G,[\mu])(R)\backslash G(R_N)].
\]
\end{corollary}

\begin{proof}
This follows from Theorem~\ref{DescribeModSpaceTruncShtukas} and Proposition~\ref{DimE}~\ref{DimE2}.
\end{proof}

We next relate $R^{\Sht,1}$ to the zip stack $R^{\Zip}$ defined in \cite{Yaylali_DerivedFZips}, where Yaylali proved that the groupoid of $G$-bundles of $R^{\Zip}$ is equivalent to the groupoid of $G$-zips over $R$, see \cite{PWZ2} for the notion of $G$-zips. Let us recall the definition $R^{\Zip}$.

\begin{defrem}\label{DefFZipStack}
For every $\kappa$-algebra $R$ let $\Frob_R$ the $q$-Frobenius $a \sends a^q$ on $R$. Define an $\FF_q$-scheme $Z_R$ by the pushout diagram
\[\xymatrix{
\Spec R \sqcup \Spec R \ar[r]^-{0 \sqcup \infty} \ar[d]_{\id_R \sqcup \Frob_R} & \PP^1_R \ar[d] \\
\Spec R \ar[r]^i & Z_R.
}\]
By \cite[7.1]{Ferrand_Conducteur} the pushout exists in the category of schemes. If we let act $\GG_{m,\FF_q}$ over $\FF_q$ on $\Spec R$ trivially and on $\PP^1_R$ by the usual action
\[
\GG_{m,\FF_q} \times_{\FF_q} \PP^1_R = \GG_{m,R} \times_R \PP^1_R \to \PP^1_R,
\]
then $\id_R \sqcup \Frob_R$ and $0 \sqcup \infty$ are $\GG_{m,\FF_q}$-equivariant and we can form the quotient stack over $\FF_q$
\begin{equation}\label{EqDefRFZip1}
R^{\Zip} := [\GG_{m,\FF_q}\backslash Z_R].
\end{equation}
The map $i\colon \Spec R \to Z_R$ is a closed immersion corresponding to an $R$-valued point $z_0 \in Z_R(R)$.

We can describe $R^{\Zip}$ also as follows. Let
\begin{equation}\label{EqReesDisp1}
D_R = \set{(f,g) \in R[t] \times R[u]}{f(0)^q = g(0)}.
\end{equation}
This is a $\ZZ$-graded $\FF_q$-algebra with $\deg(t) = -1$ and $\deg(u) = 1$, as before. Let $\tau',\sigma'\colon \Spec R \to [\GG_{m,\FF_q}\backslash \Spec D_R]$ be the morphisms induced by the $\ZZ$-graded maps $(f,g) \sends f \in R[x,x^{-1}]$ and $(f,g) \sends g \in R[x,x^{-1}]$. Then
\begin{equation}\label{EqDefRFZip2}
R^{\Zip} = \colim (\xymatrix{\Spec R \ar@<0.5ex>[r]^-{\tau'} \ar@<-0.5ex>[r]_-{\sigma'} & [\GG_{m,\FF_q}\backslash \Spec D_R]}).
\end{equation}
\end{defrem}

\begin{proposition}\label{CompareZip}
There are a universal homeomorphisms
\begin{equation}\label{EqCompareZip}
R^{\Sht,1} \lto R^{\Zip} \lto R^{\Sht,1}
\end{equation}
of algebraic stacks over $\FF_q$, functorial in $R$.
\end{proposition}

\begin{proof}
Set
\[
C_R := R[t,u]/(tu) = \set{(f,g) \in R[t] \times R[u]}{f(0) = g(0)}.
\]
Then $R^{\Sht,1} = [\GG_m\backslash \Spec C_R]$ and the maps \eqref{EqCompareZip} are induced by the map of $\ZZ$-graded $\FF_q$-algebras
\begin{align*}
C_R &\lto D_R, \qquad (f,g) \sends (f^{(q)},g),\\
D_R &\lto C_R, \qquad (f,g) \sends (f,g^{(q)}),
\end{align*}
where $(\sum_{i\geq0}a_it^i)^{(q)} = \sum_{i\geq0}a_i^qt^i$, and similarly for $g^{(q)}$.
%
%To construct \eqref{EqCompareZip} we have to construct a map $\Gamma\colon \Hckbar{1}_R \to \Zcal_R$ such that $\Gamma \circ \sigma = \Gamma \circ \tau$. To construct $\Gamma$, it suffices to construct a $\GG_m$-equivariant map $\Spec R[t,u]/(tu) \to Z_R$
\end{proof}

Let $G{-}\textup{Zip}^{\mu}$ be the stack of $G$-zips of type $\mu$ \cite{PWZ2}. Since we have $G{-}\textup{Zip}^{\mu}(R) = \Bun^{[\mu]}_G(R^{\Zip})$ by \cite[A.5]{Yaylali_DerivedFZips} we obtain the following corollary.

\begin{corollary}\label{CompareGZip}
Pullbacks via \eqref{EqCompareZip} yield universal homeomorphism of algebraic stacks
\[
\Bunline_G^{[\mu]}(\Sht,1) \lto G{-}\textup{Zip}^{\mu} \lto \Bunline_G^{[\mu]}(\Sht,1).
\]
\end{corollary}

\begin{remark}\label{RelateToYan}
One can view the above results as a reformulation of Theorem~B and Theorem~C from Qijun Yan \cite{Yan_RelationZipStackModuliStack}.
\end{remark}

\subsection{Traverso bounds for truncated $G$-shtukas}

We keep the above notation.

\begin{proposition}\label{TruncationSurjective}
For all $\infty \geq N' > N \geq 1$ and for every $\kappa$-algebra $R$, the truncation map $\Bunline_G^{[\mu]}(\Sht,N')(R) \lto \Bunline_G^{[\mu]}(\Sht,N)(R)$ is locally for the \'etale topology on $R$ essentially surjective.
\end{proposition}

It is easy to give examples that the truncation map is usually not surjective on morphisms.

\begin{proof}
By Theorem~\ref{DescribeModSpaceTruncShtukas} we have $\Bunline_G^{[\mu]}(\Sht,N) = [E_N(G,\mu)\backslash L^{(N)}(G)]$, hence any $R$-valued point of $\Bunline_G^{[\mu]}(\Sht,N)$ lifts \'etale locally on $R$ to an $R$-valued point of $L^{(N)}(G)$ because $E_N(G,\mu)$ is a smooth group scheme by Corollary~\ref{DimE}. Since $G$ is smooth and $(R_{N'},(z^N))$ is henselian, any $R$-valued point of $L^{(N)}G$ lifts to an $R$-valued point of $L^{(N')}G$.
\end{proof}

\begin{remark}\label{TruncationSurjectivePrecise}
The proof of Proposition~\ref{TruncationSurjective} shows that the obstruction to lift an $R$-valued point of $\Bunline_G^{[\mu]}(\Sht,N)$ lies in $H^1(R,E_N(G,[\mu])) \cong H^1(R,\Cent_G(\mu))$, where the identity follows from Corollary~\ref{DimE}~\ref{DimE1}.
\end{remark}

\begin{theorem}\label{ExistTraverso}
There exists an $N_0$ such the truncation map $\Bunline_G^{[\mu]}(\Sht,N') \lto \Bunline_G^{[\mu]}(\Sht,N)$ is bijective on isomorphism classes of $k$-valued points for every $\infty \geq N'\geq N \geq N_0$ and every algebraically closed extension $k$ of $\kappa$.
\end{theorem}

Since $\Bunline_G(\Sht,\infty)(k) = \Bunline_G(\Sht)(k)$ by Theorem~\ref{ColimitNTruncShtuka}~\ref{ColimitNTruncShtuka2} we obtain in particular:

\begin{corollary}\label{ExistTraversoCar}
There exists an $N_0$ such that for all $N \geq N_0$ and every algebraically closed extension $k$ of $\kappa$ truncation yields a bijection between the set of isomorphism classes of local $G$ shtukas of type $\mu$ over $k$ and the set of isomorphism classes of $N$-truncated local $G$-shtukas of type $\mu$ over $k$.
\end{corollary}

We first recall some ingredients for the proof of Theorem~\ref{ExistTraverso}. We choose a Borel pair $T\subset B\subset G$ of $G$ and a conjugacy class $[\mu]\in X_*/W_G$. Let $\mu\in X_*(T)$ be the dominant representative of $[\mu]$, defined over a finite extension of $\kappa$. Let $b\in LG(\bar \kappa)$. We consider the associated affine Deligne-Lusztig variety, that is the locally closed reduced subscheme of the affine Grassmannian of $G$ defined over $\bar \kappa$ whose $\bar \kappa$-valued points are given by

$$X_{\mu}(b)(\bar \kappa)=\{g\in LG(\bar \kappa)/L^+G(\bar \kappa)\mid g^{-1}b\phi(g)\in L^+G(\bar \kappa)\mu(z)L^+G(\bar \kappa)\}.$$

It is non-empty if and only if the $\phi$-conjugacy class $[b]$ of $b$ is contained in some finite subset $B(G,\mu)$ of the Kottwitz set $B(G)$ of $\phi$-conjugacy classes. The closed affine Deligne-Lusztig variety is defined as $X_{\leq\mu}(b)=\bigcup_{\mu'\leq\mu}X_{\mu'}(b)$. It is a closed subscheme of the affine Grassmannian. Let $J_b$ be the automorphism group of $b$, which is a reductive group over $\kappa\lsz$ with $J_b(\kappa\lsz)$ acting on $X_{\mu}(b)$ and $X_{\leq\mu}(b)$ for every $\mu$. We write $K=L^+G(\bar\kappa)$.
\begin{thm}[\cite{Cornut_Nicole} Prop. 8, \cite{HamacherViehmann_fin} Lemma 5.3]
Let $b\in K\mu(z)K$. Then there is a finite subset $M_0(b,\mu)\subset X_*(T)_{\dom}$ such that 
$$X_{\mu}(b)\subseteq J_b(\kappa\lsz)\cdot \coprod_{\xi\in M_0(b,\mu)}K\xi(z)K.$$
\end{thm}

Taking the union over the finitely many classes $[b]\in B(G,\mu)$ this implies the next corollary.  
\begin{cor}\label{cor_bound_ADLV}
For every $[b]\in B(G,\mu)$ choose some representative $b\in [b]$. Then there is a finite subset $S\subset X_*(T)_{\dom}$ such that for every $g\in K\mu(z)K$, there is an $h\in K\xi(z)K$ for some $\xi\in S$ such that $h^{-1}g\phi(h)=b$ for one of the finitely many elements $b$ chosen above.
\end{cor}
\begin{proof}
Let $S$ be the union over all $[b]\in B(G,\mu)$ of the sets $M_0(b,\mu_b)$ where $\mu_b$ is such that $b\in K\mu_b(z)K$. Let $g$ be as above, in particular $[g]\in B(G,\mu)$. Let $b$ be the fixed representative of $[g]$. There is an $h\in LG(\bar \kappa)$ with $h^{-1}b\phi(h)=g$, in particular, $h\in X_{\mu}(b)$. We may replace $h$ by another element of its $J_b(\kappa\lsz)$-orbit without changing these properties. Thus the theorem shows that we may choose $h\in K\xi(z)K$ for some $\xi\in S$.
\end{proof}

Let $\widetilde W=N_G(T)(\bar \kappa(\!(z)\!))/T(\bar \kappa[\![z]\!])$ be the extended affine Weyl group of $G$. An element $x\in\widetilde W$ is called fundamental if the Iwahori double coset $IxI$ (where $I$ is the inverse image of $B$ under the projection map $L^+G\rightarrow G$) is $I$-$\phi$-conjugate to $x$, compare \cite{Nie} for details and equivalent characterizations. By \cite{Nie}, Thm. 1.4, we may choose the representatives $b$ in the corollary to be at the same time representatives of fundamental elements of $\widetilde W$. 

Consider $I$ as a parahoric group scheme over $O = \kappa\psz$, and for every $n\in\mathbb N$ let $I_n$ denote the kernel $I(O)\rightarrow I(O/(z^n))$.

\begin{lemma}\label{lemfundalc}
Let $x$ be a fundamental element in $\widetilde W$, let $M$ be the centralizer of its Newton point and choose a representative in $LM(k)$ of $x$, also denoted by $x$. Then for every $n\in\mathbb N$, the map $I_n\rightarrow I_n x I_n$ with $i\mapsto i x\phi(i)$ is surjective. 
\end{lemma}
\begin{proof}
For $n=0$, this is a well-known property of fundamental alcoves, see for example \cite{GoertzHeNie_P-alcoves}, Theorem 3.3.1. Assume that $n>0$. In particular, we then have $[I_n,I_n]\subseteq I_{n+1}$. Using induction, it is enough to show that $I_n\cdot_{\phi} I_{n+1}xI_{n+1}\rightarrow I_nxI_n$ is surjective, where $I_n$ acts by $\phi$-conjugation.

Being fundamental and basic in $M$, the element $x\in LM(k)$ is a representative of a length 0 element of the extended affine Weyl group of $M$ (compare \cite{Nie}). Let $P$ be the parabolic subgroup of $G$ defined by the Newton point of $x$ and let $N$ be its unipotent radical. Let $\bar N$ be the unipotent radical of the opposite parabolic. Let $I_{N,n}, I_{\bar N,n}$ be the intersections with $I_n$. Since $I_n/I_{n+1}$ is commutative, it is enough to show that for $I'=I_{M,n},I_{N,n}$, or $I_{\bar N,n}$, the map $I'\cdot_{\phi} I_{n+1}xI_{n+1}\rightarrow I'I_{n+1}xI_{n+1}I'$  with $i\mapsto i x\phi(i)$ is surjective. For $I'=I_{M,n}$ recall that $x$ is of length $0$ in $M$ and thus, $I_M=I\cap M$ is stable under conjugation by $x$. Hence the same also holds for each $I_{M,n}=I_n\cap M$. Using the commutativity of $I_n/I_{n+1}$, it is enough to show that $$I_{M,n}/I_{M,n+1}\rightarrow I_{M,n}xI_{M,n}/I_{M,n+1}=xI_{M,n}/I_{M,n+1}$$ with $g\mapsto g^{-1}x\phi(g)$ is surjective. This in turn is a consequence of Lang's theorem. Now consider $I'=I_{N,n}$. Since $x$ is fundamental, and $I$ and $N$ are $\phi$-stable, we have $x\phi(I_{N})x^{-1}\subseteq I_N$, and thus also $x\phi(I_{N,n})x^{-1}\subseteq I_{N,n}$ for every $n$. Similarly, $I_{N,n}xI_{N,n}=I_{N,n}x$. Furthermore, $N$ is the unipotent radical of the parabolic defined by the Newton point of $x$, or equivalently by any positive power of $x\phi$ that is of the form $x'\phi^c$ where $x'$ is a $\phi$-stable representative of an element of $X_*(T)$. Hence $x'\phi^c(I_{N,n})(x')^{-1}\subseteq I_{N,n+1}$ for any such $c$. Let $i\in I_{N,n}$ and let $\phi_x:g\mapsto x\phi(g)x^{-1}$. Let $\delta=i\phi_x(i)\dotsm \phi_x^{c-1}(i)$. Then $\delta\in I_{N,n}$ and $\delta^{-1} ix\phi(\delta)=x\phi(\phi_x^{c-1}(i))=\phi_x^c(i)x=x'\phi^c(i)(x')^{-1}x\in I_{N,n+1}x$ which proves the assertion for $I'=I_{N,n}$. The argument for $I'=I_{\bar N,n}$ is analogous, using that $xI_{\bar N}x^{-1}\supseteq I_{\bar N}$.
\end{proof}

\begin{proof}[Proof of Theorem \ref{ExistTraverso}]
By Proposition \ref{TruncationSurjective} we only need to show injectivity. More precisely, it is enough to show that there is an $N_0$ such that
\[
\Bunline_G^{[\mu]}(\Sht) \lto \Bunline_G^{[\mu]}(\Sht,N_0)
\]
is injective on isomorphism classes of $k$-valued points. Let $\mathcal G$ and $\mathcal G'$ be $k$-valued points of the left hand side. We trivialize the $G$-torsors underlying $\mathcal G$ and $\mathcal G'$, thus they are of the form $(LG_{k}, b\phi)$ and $(LG_{k}, b'\phi)$ for some $b,b'\in LG(k)$ bounded by $\mu$. Choosing the trivializations compatibly with the isomorphism of truncated $G$-shtukas, we may assume that $b'=h_1bh_2$ with $h_1,h_2\in L^+_{N}G(k)=K_N$ for some $N$, yet to be determined. Let $b_0$ be a representative of a fundamental element in $[b]$. Let $S$ be as in Corollary \ref{cor_bound_ADLV}. Then there is a $g\in G(\breve F)$ with $g^{-1}b_0\phi(g)=b$, i.e., $g\in X_{\leq \mu}(b_0)$. By the corollary, we may choose $g\in \coprod_{\xi\in S}G(\mathcal O_{\breve F})\xi(t)G(\mathcal O_{\breve F})$. We have
\[
gb'\phi(g)^{-1}=(gh_1g^{-1})b_0\phi(gh_2g^{-1})^{-1}.
\]
Define
\begin{equation}\label{EqDefForIsomCutOff}
C := \max_{\alpha,\xi \in S} \langle \alpha,\xi\rangle, 
\end{equation}
where $\alpha$ runs over all roots. For every $C<n\in \mathbb N$, every  $g\in \coprod_{\xi\in S}L^+G(k)\xi(t)L^+G(k)$ and every $h\in K_n$ we have $ghg^{-1}\in K_{n-C}$. Thus if we choose $N$ larger than $C$, we have $gb'\phi(g)^{-1}\in gK_Nb K_N g^{-1}\subseteq  K_{N-C}b_0K_{N-C}$. In particular, from Lemma \ref{lemfundalc} for $n=1$ we obtain that $\mathcal G,\mathcal G'$ are isogenous as soon as $N>C$ since then the right hand side is contained in $[b_0]$. Assume now that $N>2C$. By the lemma, $gb'\phi(g)^{-1}\in I_{N-C-1}b_0 I_{N-C-1}$ can be written as $ib_0\phi(i)^{-1}$ for some $i\in I_{N-C-1}\subseteq K_{N-C-1}$. Hence 
$b'=(g^{-1}ig)b\phi(g^{-1}i^{-1}g)$ and $g^{-1}ig$ is in $K_{N-2C-1}\subseteq K$. Thus $b'$ is $K$-$\phi$-conjugate to $b$, and the local $G$-shtukas $\mathcal G$ and $\mathcal G'$ are isomorphic.
\end{proof}

\begin{remark}
The least possible $N_1$ such that any two $\mathcal G,\mathcal G'\in \Bunline_G^{[\mu]}(\Sht)(k)$ mapping to the same point in $\Bunline_G^{[\mu]}(\Sht,N_1)$ are isogenous is called the isogeny cutoff for the datum $(G,\mu)$. Similarly, the least possible $N_0$ as in Theorem \ref{ExistTraverso} is called the isomorphism cutoff for $(G,\mu)$. From the proof of the theorem, we obtain explicit bounds $C+1$ for the isogeny cutoff and $2C+1$ for the isomorphism cutoff in terms of the set $S$ in Corollary \ref{cor_bound_ADLV}.
\end{remark}

\section{Truncated displays}\label{Chapter:TruncatedDisplays}

In this section, we apply the general results to displays and their truncations. Many of the results of this section just rephrase results from Lau in \cite{Lau_HigherFrames} and Daniels in \cite{Daniels_TannakianFrameWork} in our language. Exceptions are the geometrization of the theory obtained by defining and using the various stacks $R^{\Disp}$, $R^{\Disp,N}$, or $R^{\Disp,N,n}$. 

In this section, set $\kappa := \FF_p$ the field with $p$-elements and $O = \ZZ_p$. Choose an algebraic closure $\kgbar$ of $\FF_p$ and let $\ZZ_p^{\ur}$ be the ring of integers in the corresponding maximal unramified extension of $\QQ_p$.

Let $R$ denote a $p$-adically complete ring, let $W(R)$ be the ring of $p$-adic Witt vectors. For every $N \geq 1$ we denote by $W_N(R)$ its truncation, so that $W_1(R) = R$. We also set $W_{\infty}(R) := W(R)$. For $1 \leq N \leq \infty$ we denote by $V\colon W_N(R) \to W_N(R)$ the Verschiebung and set $I_N(R) := V(W_N(R))$ which is an ideal of $W_N(R)$ such that $W_N(R)/I_N(R) \cong R$.

We will often consider $p$-adically formal stacks. These are defined as stacks on the site of bounded $p$-adically complete rings endowed with the $p$-completely flat topology.

\subsection{The truncated display stacks}\label{Section:TruncatedDisplayStack}

Let $1 \leq N \leq \infty$. Every element of $I_N(R)$ is of the form $V(x)$ for a unique $x \in W_N(R)$. Define a filtration on $W_N(R)$ (in the sense of Definition~\ref{DefFilteredRing}) by $\Fil^i := I_{N+1}(R)$ for $i > 0$, where we set $\infty + 1 := \infty$, with maps
\[
\ldots \ltoover{p} \Fil^{i+1} \ltoover{p} \Fil^{i} \ltoover{p} \ldots \ltoover{p} \Fil^1 \lto \Fil^0 = W_N(R)
\]
defined as follows. For $i \geq 1$ the maps $\Fil^{i+1} \to \Fil^i$ are given by multiplication with $p$. The map $\Fil^1 \to W_N(R)$ is the truncation map $I_{N+1}(R) \to I_N(R)$ followed by the inclusion $I_N(R) \to W(R)$. So the image of $\Fil^1 \to W_N(R)$ is $I_N(R)$.
The multiplication $\Fil^i \times \Fil^j \to \Fil^{i+j}$ for $i,j \geq 1$ is given by
\[
I_{N+1}(R) \times I_{N+1}(R) \lto I_{N+1}(R), \qquad V(x)\cdot V(y) := V(xy), \quad\text{for $x,y \in W(R)$.}
\]
This defines a filtration since one has $V(x)V(y) = pV(xy)$ for all $x,y \in W_N(R)$. 

We get the attached $\ZZ$-graded $W_N(R)$-algebra $\Rees(W_N(R))$ and the attached Rees stack
\[
\Re(W_N(R)) = [\GG_{m,W(R)}\backslash \Spec \Rees (W_N(R))],
\]
which we call the \emph{$N$-truncated Rees--Witt stack}. For $N = \infty$, it it simply called the \emph{Rees--Witt stack}. For all $1 \leq N \leq N' \leq \infty$ we obtain a closed immersion of algebraic stacks
\[
\Re(W_N(R)) \lto \Re(W_{N'}(R)).
\]

Next we endow the filtered ring $(W_N(R), (\Fil^i)_i)$ with the structure of a frame (Definition~\ref{DefFrame}). For this, we start with the following remark.

\begin{remark}\label{RemarkReesWittRep}
The quotient algebra $\Rees(W_N(R))^-$, which defines the repeller locus, is defined in $\Rees(W_N(R))$ by the ideal $I^+$ with $I^+_i = p^{-i}I_N(R) = I_N(R)^{-i+1}$ for $i \leq 0$ by Remark~\ref{ReesAttractor} and because we have $I_N(R)^i = p^{i-1}I_N(R)$ for all $i \geq 1$.
\end{remark}
 
Let us now define a frame first for $N = \infty$. Let $\phi$ be the Frobenius endomorphism on $W(R)$ and define $\sigma\colon \Spec W(R) \to \Re(W(R))$ inducing $\phi$ by Lemma~\ref{ConstructFrame} as follows. Let $\sigma^*\colon \phi^*\Rees(W(R)) \to \bigoplus_{i\in \ZZ}W(R)^{\otimes i}$ with $\sigma^*_{-1}(1 \otimes t) = p \in W(R)^{\otimes-1}$, $\sigma^*_0(1 \otimes a) = \phi(a)$, and $\sigma^*_i(1 \otimes V(a)) = a$ for $a \in W(R)$ and $i \geq 1$. For $i \geq 1$, the image of $\Rees(A)_i \subseteq I^+$ under $\sigma^*_i$ generates $W(R)$ and hence $\sigma$ factors through a morphism $\sigma\colon \Spec W(R) \to \Re(W(r))^{\ne-}$ by Lemma~\ref{ConstructFrame} using Remark~\ref{RemarkReesWittRep}.

Hence we obtain the \emph{Witt frame} $(W(R), (\Fil^i)_i, \sigma)$, reproducing Lau's definition in \cite{Lau_HigherFrames} of the Witt frame in our language. We get the attached frame stack (Definition~\ref{DefFrameStack})
\begin{equation}\label{EqDefRDisp}
R^{\Disp} := \Fcal(W(R), (\Fil^i)_i, \sigma)
\end{equation}
which we call the \emph{displayfication of $R$}.

Now let us consider the truncated case, i.e., $N < \infty$. In this case, the Frobenius $\phi$ induces a ring endomorphism $\phi\colon W_N(R) \to W_N(R)$ only if $R$ is of characteristic $p$. Hence let us assume this for now, see also Remark~\ref{TruncatedWittPAdic} below for a variant for $p$-adically complete rings.

Let $\sigma\colon \Spec W_N(R) \to \Re(W_N(R))$ be defined as in the non-truncated case above. In other words, it is the unique map of stacks making the diagram
\[\xymatrix{
\Spec W_N(R) \ar[r]^{\sigma} \ar[d] & \Re(W_N(R)) \ar[d] \\
\Spec W(R) \ar[r]^{\sigma} & \Re(W(R))
}\]
commutative. We obtain a frame structure on $\Re(W_N)$ which we call the \emph{$N$-truncated Witt frame}. The attached frame stack is denoted by
\begin{equation}\label{EqDefRDispN}
R^{\Disp,N} := \Fcal(W_N(R), (\Fil^i)_i, \sigma)
\end{equation}
and is called the \emph{$N$-truncated displayfication of $R$}. We also set $R^{\Disp,\infty} := R^{\Disp}$.
Note that the truncated Witt frame is not a quotient frame of the Witt frame in the sense of Definition~\ref{ExtensionOfFrames}. But still we have for all $1 \leq N \leq N' \leq \infty$ a commutative diagram
\[\xymatrix{
\Re(W_N(R)) \ar[r] \ar[d] & R^{\Disp,N} \ar[d] \\
\Re(W_{N'}(R)) \ar[r] & R^{\Disp,N'}.
}\]

\begin{remark}\label{TruncatedWittPAdic}
We can also define a variant of $R^{\Disp,N}$ if $R$ is $p$-adically complete following ideas from Hoff \cite{Hoff_EKORSiegel} and \cite{Hoff_EKORGeneral}. In this case, the Frobenius $\phi$ induces maps $W_N(R) \to W_n(R)$ for all $n < N$. We denote this induced map again by $\phi$. Thus we fix integers $1 \leq n < N < \infty$. In this case we obtain as above a map $\sigma\colon W_n(R) \to \Re(W_N(R))$. We also denote by $\tau$ the composition of $\Spec W_n(R) \to W_N(R)$ with the canonical map $\Spec W_N(R) \to \Re(W_N(R))$ \eqref{EqDefineTau2} and then define a stack $R^{\Disp,N,n}$ as the coequalizer of $\sigma$ and $\tau$.

Let us endow
\[
\Ncal := \set{(N,n) \in \ZZ^2}{1 \leq n < N}
\]
with the product partial order
\[
(N,n) \leq (N',n') {:}\iff N \leq N' \text{\ and\ } n \leq n'.
\]
Then one obtains for all pairs $(N,n) \leq (N',n')$ a map
\[
R^{\Disp,N,n} \lto R^{\Disp,N',n'}
\]
\end{remark}

The following result describes $1$-truncated displays. It is a stacky reformulation of a result of Lau \cite[2.1.7]{Lau_HigherFrames}.

\begin{proposition}\label{1TruncDispFZip}
Let $R$ be a ring of characteristic $p$.
\begin{assertionlist}
\item\label{1TruncDispFZip1}
One has a pushout diagram of Adams stacks
\[\xymatrix{
\B{\GG_{m,R}} \sqcup \B{\GG_{m,R}} \ar[r]^-{0 \sqcup \infty} \ar[d]_{\id_R \sqcup \Frob_R} & [\GG_{m,R}\backslash\PP^1_R] \ar[d] \\
\B{\GG_{m,R}} \ar[r]^i & \Re(W_1(R)).
}\]
\item\label{1TruncDispFZip2}
There is an isomorphism of stacks, functorial in $R$,
\begin{equation}\label{EqIdentifyZipDisp1}
R^{\Zip} \liso R^{\Disp,1}.
\end{equation}
\end{assertionlist}
\end{proposition}

Pulling back $\GL_n$-bundles via \eqref{EqIdentifyZipDisp1} gives an alternative proof of \cite[3.6.4]{Lau_HigherFrames}. 

\begin{proof}
Recall the $\ZZ$-graded $\FF_p$-algebra $D_R$ from \eqref{EqReesDisp1}. It is straight forward to see that one has an isomorphism of $\ZZ$-graded algebras $D_R \cong \Rees(W_1(R))$ which implies \ref{1TruncDispFZip1} after forming quotients by $\GG_m$. This is compatible with pushouts, since forming quotients can be itself expressed as a colimit. The isomorphism $D_R \cong \Rees(W_1(R))$ also shows \ref{1TruncDispFZip2} by \eqref{EqDefRFZip2} and the definition of $R^{\Disp,1}$.
\end{proof}

%---------------------------------------------------------------------

\subsection{Classification of truncated $G$-displays}

From now on let $G$ be a smooth affin group scheme over $\ZZ_p$. Our next goal is to apply the general results to describe moduli spaces of $G$-bundles on the stacks defined above.

Since $R$ is a $p$-adically complete ring, $W_N(R)$ is $I_N(R)$-adically complete and $p$-adically complete \cite[Prop.~3]{Zink_Displays}. In particular $(W_N(R),I_N(R))$ is a henselian pair and hence $\Rees(W_N(R))$ is henselian (Definition~\ref{DefHenselianGradedRing}). Therefore Corollary~\ref{LiftingGBundlesReesStack} gives the following result, which for $G = \GL_n$ can be seen as a reformulation for the existence of normal decompositions for displays as in \cite[3.3]{Lau_HigherFrames}.

\begin{proposition}\label{GBundlesReesWitt}
There is an isomorphism of pointed sets
\[
H^1(\Re(W_N(R)),G) \iso H^1(\B{\GG_{m,R}},G).
\]
%for every smooth affine group scheme $G$ over $W_N(R)$.
\end{proposition}

Let us now study the following prestacks of $G$-bundles.

\begin{definition}\label{DefGDisplayStack}
For every $p$-adically complete ring $R$ we set
\begin{align*}
\Bunline_G(W)(R) &:= \Bun_G(\Re(W(R))), \\
\Bunline_G(W_N)(R) &:= \Bun_G(\Re(W_N(R))), \qquad 1 \leq N \leq \infty,\\
\Bunline_G(\Disp)(R) &:= \Bun_G(R^{\Disp}), \\
\Bunline_G(\Disp,N,n)(R) &:= \Bun_G(R^{\Disp,N,n}), \qquad 1 \leq n < N < \infty, \\
\intertext{if $pR = 0$, we also set}
\Bunline_G(\Disp,N)(R) &:= \Bun_G(R^{\Disp,N}), \qquad 1 \leq N \leq \infty.
\end{align*}
\end{definition}

\begin{lemma}\label{BunlineDispStack}
All these prestacks are stacks for the $p$-completely flat topology.
\end{lemma}

For $\Bunline_G(\Disp,N)$ this just means that it is an fpqc-stack on the catagory of all $\FF_p$-algebras.

\begin{proof}
The prestacks $\Bunline_G(W_N)$ are stacks: For $N = \infty$ this follows from \cite[3.5]{Daniels_TannakianFrameWork} since one has a Tannakian description of $G$-bundles by \cite[A.30]{Wedhorn_ExtendBundles}. The truncated case can be shown similarly, only that the used descent result by Lau \cite[4.3.2]{Lau_HigherFrames} is easier to prove in the truncated case. Now $R \sends \Bun_G(W_N(R))$ is also a stack for all $1 \leq N < \infty$ (use for instance \cite[2.12]{BH_GMuWindows}). This shows that the prestacks $\Bunline_G(\Disp)$, $\Bunline_G(\Disp,N,n)$, $\Bunline_G(\Disp,N)$ are all equalizers of stacks, hence they are also stacks.
\end{proof}

Now fix a conjugacy class $[\mu]$ of cocharacters of $G$ defined over some finite flat extension dvr $O_{[\mu]}$ of $\ZZ_p$. If $G$ is reductive, we let $O_{[\mu]}$ its ring of definition, a finite \'etale extension of $O$ by Remark~\ref{SplittingNormal}.

\begin{definition}\label{DefGDisplay}
Let $R$ be a $p$-adically complete ring. A \emph{$G$-display over $R$} is defined as a $G$-bundle over $R^{\Disp}$. An \emph{$(N,n)$-truncated $G$-display over $R$} is a $G$-bundle over $R^{\Disp,N,n}$. If $R$ is of characteristic $p$, then an \emph{$N$-truncated $G$-display over $R$} is a $G$-bundle over $R^{\Disp,N}$.

We say that a (truncated or not) display $\Escr$ is of type $[\mu]$ if it is \'etale locally isomorphic to $\Escr_{[\mu]}$ (Definition~\ref{DefineGModMu}).
\end{definition}

We obtain substacks $\Bunline_G^{[\mu]}(\bigstar)$ for $\bigstar \in \{W,W_N, \Disp,(\Disp,N,n), (\Disp,N)\}$ over $O_{[\mu]}$. If $G$ is reductive, these are open and closed substacks \eqref{EqDecomposeBunGBGGm}.

\begin{remark}\label{RelateToHoff}
As already mentioned (without proof) by Hoff in \cite[2.8]{Hoff_EKORSiegel}, the groupoid of pairs of type $(h,d)$ over $R$ in the sense of loc.~cit.\ Definition~2.2 is equivalent to the groupoid of $\GL_h$-bundles over $\Re(W(R))$ of type $[\mu_d]$, where $[\mu_d]$ is the conjugacy class of the standard minuscule cocharacter
\[
\mu_d\colon \GG_m \to \GL_h, \qquad t \sends \twomatrix{tI_d}{0}{0}{I_{h-d}}.
\]
This implies easily that the groupoid of displays of type $(h,d)$ over $R$ is equivalent to the groupoid of $\GL_h$-bundles of type $[\mu_d]$ over $R^{\Disp}$.
Christopher Lang gives a proof of these statements in Appendix~\ref{APPLang}.

A similar result with the same proof holds for $N$-truncated pairs, $(N,n)$-truncated displays, and $N$-truncated displays.
\end{remark}

Next we give descriptions of all these stacks as quotient stacks, similarly as in the shtuka case in Chapter~\ref{TRUNCHECKE}. First recall the notion of truncated Witt loop groups. For each $p$-adically complete $R$-algebra we set
\begin{align*}
L^WG(R) &:= G(W(R)[1/p]), \\
L^{W,+}G(R) &:= G(W(R)), \\
L^{W,(N)}G(R) &:= G(W_N(R)), \qquad 1 \leq N \leq \infty.
\end{align*}
Hence $L^{W,(\infty)}G = L^{W,+}(G)$. Then $L^{W,(N)}G$ is an affine group scheme. It is of finite type for $N < \infty$.

As in Section~\ref{EGENERAL} we also have a display group $E_N^W(G,[\mu])$ attached to $R \sends \Re(W_N(R))$ for all $1 \leq N \leq \infty$ which by Proposition~\ref{EMuRepresentable} is a formal affine group scheme over $\Spf \ZZ_p$ which is smooth if $N < \infty$. Moreover $E_{\infty}(G,[\mu]) = \lim_{N<\infty}E_N(G,[\mu])$. The general results Theorem~\ref{GeneralBunlineGMu} and Theorem~\ref{DescribeBunlineFramStack} then specialize to the following description of the various stacks of $G$-bundles.

\begin{theorem}\label{DescribeGDisplays}
One has equivalences of stacks for the $p$-completely flat topology on the category of bounded $p$-adically complete rings
\begin{align*}
\Bunline^{[\mu]}_G(W_N) &\cong \B{E_N^W(G,[\mu])}, \\
\Bunline^{[\mu]}_G(\Disp) &\cong [E_{\infty}^W(G,[\mu])\backslash L^{W,+}G],\\
\Bunline^{[\mu]}_G(\Disp,{N,n}) &\cong [E_{N}^W(G,[\mu])\backslash L^{W,(n)}G].\\
\intertext{Further, one has equivalences of stacks for the fpqc-topology on the category of $\FF_p$-algebras}
\Bunline^{[\mu]}_G(\Disp,{N}) &\cong [E_{N}^W(G,[\mu])\backslash L^{W,(N)}G].
\end{align*}
\end{theorem}

\begin{proof}
We have to check that the Assumptions~\ref{ConditionFrameFunctor} are satisfied. For \ref{ConditionFrameFunctora} this is clear in the truncated case since $R \sends W_N(R)$, $R \sends I_N(R)$ are represented as schemes by an affine space for $N < \infty$. For $N = \infty$ they are represented by a cofiltered limit of affine spaces which is an affine scheme. We have already remarked above that Condition~\ref{ConditionFrameFunctorb} is satisfied since $W_N(R)$ is $I_N(R)$-adically complete. Finally, Condition~\ref{ConditionFrameFunctorc} is satisfied by \cite[2.12]{BH_GMuWindows}).
\end{proof}

\begin{corollary}\label{LimitTruncatedDisplays}
One has
\begin{align*}
\Bunline^{[\mu]}_G(W) &= \lim_{N<\infty}\Bunline^{[\mu]}_G(W_N), \\
\Bunline^{[\mu]}_G(\Disp) &= \lim_{(N,n) \in \Ncal^{\opp}}\Bunline^{[\mu]}_G(\Disp,{N,n}),\\
\Bunline^{[\mu]}_G(\Disp) &= \lim_{N<\infty} \Bunline^{[\mu]}_G(\Disp,{N})
\end{align*}
\end{corollary}

The same arguments as in the equi-characteristic shtuka case show that one has the following description of $E_{N}^W(G,\mu)$ if $[\mu]$ is represented by a cocharacter $\mu$. Let $R$ be a bounded $p$-complete ring such that $W(R)$ is $p$-torsion free. This implies that $G(W(R)) \to G(W(R)[1/p])$ is injective. Then
\begin{equation}\label{EqEInftyGWMu}
E^W_{\infty}(G,\mu)(R) = L^{W,+}G(R) \cap \mu^{-1}L^{+,W}G(R)\mu \subseteq L^WG(R).
\end{equation}
If $R$ is in addition an $\FF_p$-algebra, then one has for $1 \leq N < \infty$ that
\begin{equation}\label{EqENWGMu}
E_N^W(G,\mu)(R) = E^W_{\infty}(G,\mu)(R)/(L^W_NG(R) \cap \mu^{-1}L^W_NG(R)\mu),
\end{equation}
where $L^W_NG := \Ker(L^{W,+}G \to L^{W,(N)}G)$.

For a characterization of rings $R$ such that $W(R)$ is $p$-torsion free, see \cite[4.1]{MunozBertrand_CharPropWitt}. All reduced rings $R$ have this property. In characteristic $p$, also the converse holds.

\begin{proposition}\label{TruncationWittDisplayGroup}
Let $1 \leq N < \infty$. Then the canonical map
\[
E_{N+1}^W(G,\mu) \to E_N^W(G,\mu)
\]
is a surjective homomorphism of smooth algebraic groups. Its kernel is a vector group of dimension $\dim G$.
\end{proposition}

As in the equi-characteristc shtuka case (Corollary~\ref{DescribePointsShtuka} and Proposition~\ref{TruncationSurjective}) one deduces via Theorem~\ref{DescribeGDisplays} the following results.

\begin{corollary}\label{ClassifyIsomClassesTruncDisplay}
For every strictly henselian local $\FF_p$-algebra $R$ and for every $1 \leq N \leq \infty$, one has an equivalence of groupoids
\[
\Bunline^{[\mu]}_G(\Disp,{N})(R) \cong [E_{N}^W(G,\mu)(R)\backslash G(W_N(R))].
\]
\end{corollary}

\begin{corollary}\label{TruncationDisplaysSurjective}
For all $\infty \geq N' > N \geq 1$ and for every $\kappa$-algebra $R$, the truncation map $\Bunline_G^{[\mu]}(\Disp,N')(R) \lto \Bunline_G^{\mu}(\Disp,N)(R)$ is locally for the \'etale topology on $R$ essentially surjective. 
\end{corollary}

Now the same proof as for Theorem~\ref{ExistTraverso}, replacing $k\psz$ by $W(k)$, shows the existence of Traverso bounds for truncated $G$-display of type $\mu$:

\begin{theorem}\label{ExistTraversoDisplay}
Let $G$ be reductive over $\ZZ_p$ and let $[\mu]$ be a conjugacy class of cocharacter of $G$, defined over a finite \'etale dvr extension $O_{[\mu]}$ of $\ZZ_p$. Let $k$ be an algebraically closed extension of the residue filed of $O_{[\mu]}$. Then there exists an $N_0 < \infty$ such the truncation map $\Bunline_G^{\mu}(\Disp,N') \lto \Bunline_G^{\mu}(\Disp,N)$ is bijective on isomorphism classes of $k$-valued points for every $\infty \geq N'\geq N \geq N_0$.
\end{theorem}

%
%
%Then we obtain a formal stack $\Bunline_G(\Disp)$ over $\Spf(\ZZ_p)$, where
%\begin{equation}\label{EqDefBunlineDisp}
%\Bunline_G(\Disp)(R) := \Bun_G(R^{\Disp})
%\end{equation}
%for every $p$-adically complete $\ZZ_p$-algebra $R$.
%
%One has the map
%\[
%\iota\colon \B{\GG_{m,R}} = \Re(W(R))^0 \mono \Re(W(R)) \lto R^{\Disp},
%\]
%which yields a pullback map
%\[
%\iota^*\colon \Bun_G(R^{\Disp}) \lto \Bun_G(\B{\GG_{m,R}}).
%\]
%Let $G$ be a reductive group scheme over $\ZZ_p$. Then $G$ splits over some finite \'etale $\ZZ_p$-algebra. We obtain from \eqref{EqDecomposeBunGBGGm} a decomposition into open and closed formal substacks
%\[
%\Bunline_G(\Disp)_{\ZZ_p^{\ur}} = \coprod_{[\mu]}\Bunline_G^{[\mu]}(\Disp),
%\]
%where $\Bunline_G^{[\mu]}(\Disp)$ is defined over the ring of definition $O_{[\mu]}$ of the conjugacy class $[\mu]$ of cocharacters, which is a finite \'etale $\ZZ_p$-algebra.

%=====================================================================

\section{Prismatic displays, prismatic Breuil--Kisin modules, and prismatic $F$-gauges}\label{Sec:PrismaticDisplays}

In this section, let $\Aline := (A,\delta,I)$ be a bounded prism and let $\phi\colon A \to A$ be the attached Frobenius endomorphism.

We attach two frames to $\Aline$ which we call the prismatic frame and the Nygaard frame. First, we will study the prismatic frame. Most results in these two sections are rephrasements of results of Kazuhiro Ito \cite{Ito_PrismaticGDisplaysDescent} and \cite{Ito_DefPrismaticGDisplays} or slight generalizations thereof in our language using the principle of geometrization. 

Moreover, we will study truncations in this setting only in a later paper.

\subsection{Prismatic displays and prismatic Breuil--Kisin modules}

\begin{void}\textbf{The prismatic frame.}
We can attach a frame to every bounded prism as follows. Let $\Re(A,I)$ be the attached Rees stack defined in Section~\ref{Sec:ReesVAdic}, see in particular Remark~\ref{VAdicSpecialCases}. It contains the open substacks $\Re(A,I)^{\ne-}$ and $\Re(A,I)^{\ne+}$, which are both isomorphic to $\Spec A$ and whose intersection is $\Spec A[1/I]$. Their union is the open substack $\Re(A,I)^{\ne0}$, which is a non-separated scheme carrying an ample pair of line bundles (Remark~\ref{ReAne0ResolutionProp}).

We define $\sigma$ as in \eqref{EqDefSigmaVAdicFrame}. This indeed recovers the description in \cite[5.2.3]{Ito_PrismaticGDisplaysDescent} if one can and does choose a generator $d$ of $I$: Then one has an isomorphism of $\ZZ$-graded $A$-algebras $T(I) \iso T(A)$, given in degree $i$ by $a \sends a/d^i$, which induces after taking quotients by $\GG_m$ the identity $\Spec A \iso \Spec A$. Therefore, we $\sigma$ is given by
\[
\sigma_d^*\colon \phi^*\Rees(A,I) \lto T(A),
\]
where $\sigma^*_{d,i}$ is given for $i \geq 0$ by the $\phi$-linear map $I^i \to A$, $a \sends \phi(a/d^i)$.

%Using Lemma~\ref{ConstructFrame} one sees that we obtain a frame $(A,I,\sigma_d)$, 
We obtain a frame stack $\Fcal(\Aline)$, and its open substack $\Fcal(\Aline)^{\ne0}$ (Definition~\ref{DefFrameStack}), which we denote by
\[
\Aline^{\BK} \subset \Aline^{\Disp}.
\]
We call the frame stack $\Aline^{\Disp}$ the \emph{prism display stack} of $\Aline$ and its open substack $\Aline^{\BK}$ the \emph{Breuil-Kisin stack} of $\Aline$.
If $\Aline$ can be oriented by a generator $d \in I$, the underlying Rees algebra is isomorphic to $A[t,u]/(tu-d)$ (Remark~\ref{VAdicSpecialCases}~\ref{VAdicSpecialCases2}).

%If we choose another generator $d'$ of $I$, then there exists a unique unit $u \in A^{\times}$ such that $d' = du$. Then multiplication with $u^i$ on the $i$-th filtration step yields an isomorphism of filtered rings $(A,d') \iso (A,d)$ which induces an isomorphism of frames $(A',d',\delta') \to (A,d,\delta)$.
%
%For general prisms we define the prismatic frame stack by a colimit as follows. Let $\Ccal$ be the category of oriented bounded prisms $(B,d,\delta)$ together with a map of bounded prisms $(A,I,\delta) \to (B,(d),\delta)$. Maps in $\Ccal$ are maps of prisms (not necessarily preserving the chosen generators) over $(A,I,\delta)$. For every map $\psi\colon (B',(d'),\delta') \to (B,(d),\delta)$ in $\Ccal$ let $u \in B^{\times}$ be the unique element satisfying $\psi(d') = ud$. Then $\psi$ and multiplication with $u$ yield a map of frames $(B', (d'), \sigma_{d'}) \lto (B,(d), \sigma_{d})$ and hence a map of frame stacks
%\[
%\Fcal(B,d,\delta) \lto \Fcal(B',d',\delta).
%\]
%This defines a functor from $\Ccal$ into the bicategory of Adams stacks.
%We call
%\[
%\Aline^{\Disp} := \colim_{(B,d,\delta) \in \Ccal}\Fcal(B,d,\delta)
%\]
%the \emph{prism frame stack} of $(A,I,\delta)$. The open substack
%\[
%\Aline^{\BK} := \colim_{(B,d,\delta) \in \Ccal}\Fcal(B,d,\delta)^{\ne0}
%\]
%is called the \emph{Breuil-Kisin stack} attached to the prism $\Aline$.
\end{void}

\begin{remark}\label{FunctorialInPrism}
Let $f\colon \Aline = (A,I,\delta) \to \Aline' = (A',I',\delta')$ be a map of prisms. Then by functoriality (Remark~\ref{FunctorialReesAv}) we obtain an induced map of $\Re(A,I) \to \Re(A',I')$ which induces maps
\[
\Aline^{\Disp} \lto \Aline^{\prime\Disp}, \qquad \Aline^{\BK} \lto \Aline^{\prime\BK}.
\]
\end{remark}

As a special case of Proposition~\ref{LiftGBundlesReAv} we have for every smooth affine group scheme $G$ over $A$ a bijection of pointed sets
\begin{equation}\label{EqGBunReesPrism}
H^1(\Re(A,I),G) \iso H^1(\B{\GG_{m,A/I}},G),
\end{equation}
which may be viewed as the existence of a prismatic normal decomposition.

We have the following definitions, which are reformulations of definitions due to Ito \cite[5.2.1]{Ito_PrismaticGDisplaysDescent}, \cite[3.1.2]{Ito_PrismaticGDisplaysDescent}.

\begin{definition}\label{DefGDisplayBKModule}
Let $\Aline = (A,I,\delta)$ be a bounded prism and let $G$ be an affine flat group scheme over $A$.
\begin{assertionlist}
\item
A \emph{$G$-display over $\Aline$} is a $G$-bundle over $\Aline^{\Disp}$.
\item
A \emph{Breuil-Kisin $G$-bundle over $\Aline$} is a $G$-bundle over $\Aline^{\BK}$.
\end{assertionlist}
\end{definition}

The following result reproves and generalizes \cite[5.3.8]{Ito_PrismaticGDisplaysDescent}. It is a special case of Corollary~\ref{GBundlesZAdicFRameStack}.

\begin{proposition}\label{BKDisplay}
Let $G$ be an affine flat group scheme over a Dedekind domain $O$ and let $\Aline$ be a bounded prism such that $A$ is an $O$-algebra.
\begin{assertionlist}
\item
Pullback via the open immersion $\Aline^{\BK} \lto \Aline^{\Disp}$ yields a fully faithful functor from the groupoid of $G$-displays of $\Aline$ to the groupoid of Breuil-Kisin $G$-bundles over $\Aline$.
\item
Let $G$ be in addition smooth. The essential image consists of those Breuil-Kisin $G$-bundles that have a type $[\mu]$ (in the sense of Definition~\ref{DefineGModMu}~\ref{DefineGModMu3}).
\end{assertionlist}
\end{proposition}

If $\mu$ is a cocharacter of $G$, then $G$-bundles on $\Aline^{\BK}$ that are \'etale locally isomorphic to $\Escr_{\mu}$ are the same as $G$-Breuil--Kisin modules of type $\mu$ as defined in \cite[5.1.4]{Ito_PrismaticGDisplaysDescent}.

\begin{example}\label{DisplayPerfectRing}
Let $R$ be a perfect ring of characteristic $p$. Then $\underline{W(R)} := (W(R), (p), \delta)$ is a prism, where $\delta$ is given by $p\delta(x) = \phi(x) - x^p$. Then $\underline{W(R)}^{\Disp} = R^{\Disp}$, where the right hand side is defined in \eqref{EqDefRDisp}. In particular a $G$-display over $\underline{W(R)}$ is the same as a $G$-display over $R$.

If $R = k$ is a perfect field, then we deduce from Corollary~\ref{GBundlesZAdicFRameStack}~\ref{GBundlesZAdicFRameStack3} that pullback via the open immersion $\underline{W(k)}^{\BK} \to \underline{W(k)}^{\Disp}$ yields an equivalence of $G$-displays and Breuil--Kisin $G$-bundles over $\underline{W(k)}$ if $G$ is reductive.
\end{example}

%---------------------------------------------------------------------

\subsection{Prismatic $F$-gauges}

\begin{void} \textbf{The Nygaard frame.}
Let $\Aline = (A,I,\delta)$ be a bounded prism. For $i \geq 0$ define the \emph{Nygaard filtration}
\[
\Fil_{\Ncal}^i := \set{a \in A}{\phi(a) \in I^i}
\]
which is a filtration by ideals of $A$. We denote by $\Rees_{\Ncal}\Aline$ the corresponding Rees algebra and by $\Aline^{\Ncal}$ the corresponding Rees stack. Using Remark~\ref{ConstructFrame} we define $\sigma\colon \Spec A \to \Aline^{\Ncal}$ by
\[
\sigma^*\colon \phi^*\Rees_{\Ncal}\Aline \lto \bigoplus_{i\in \ZZ}I^{\otimes i},
\]
where $\sigma^*_i\colon \phi^*\Fil_{\Ncal}^i \to I^i$ for $i \geq 0$ is the $A$-linear map corresponding to $\phi\colon \Fil_{\Ncal}^i \to I^i$ and where $\sigma^*_{-1}(1 \otimes t) = 1 \in A \subseteq I^{-1}$. We obtain a frame
\[
(A, (\Fil^i_{\Ncal}), \sigma)
\]
which we call the \emph{Nygaard frame} of the prism $(A,I,\delta)$. We denote the attached frame stack by 
\[
\Aline^{\Syn}
\]
and call it the \emph{syntomic frame stack} of $\Aline$.
\end{void}

\begin{example}\label{ReesQRSP}
Let $R$ be a quasi-regular semiperfectoid (or short qrsp) ring and let $(\Prism_R,I)$ be the initial object in the prismatic site of $R$. By definition, there exists a perfectoid ring mapping to $R$ and hence this prism is orientable, i.e. $I = (d)$ for a regular element $d \in \Prism_R$. Let $\varphi\colon \Prism_R \to \Prism_R$ be its Frobenius. Endowing $\Prism_R$ with the Nygaard filtration, we obtain the corresponding Rees stack
\[
R^{\Ncal} := \Re(\Prism_R).
\]
By \cite[12.2]{BhattScholze_PrismaticCohomology}, we have $\Prism_R/\Fil^1_{\Ncal} \cong R$, i.e. $(R^{\Ncal})^0 = \B{\GG_{m,R}}$. We write
\[
R^{\Syn} := (\Prism_R,I)^{\Syn}.
\]
By \cite[5.5.11]{Bhatt_PrismaticFGauges}, the $R^{\Ncal}$ and the $R^{\Syn}$ defined here are the algebraic versions of the stack $R^{\Ncal}$ and $R^{\Syn}$ defined in loc.~cit., where the formal spectrum (instead of the algebraic spectrum) of the Rees algebra is used in the definition.
\end{example}

\begin{example}\label{ReesPerfectoid}
A special case of Example~\ref{ReesQRSP} is the Rees stack of an integral perfectoid ring $R$. In this case $(\Prism_R,(d)) = (A_{\inf}(R),(d))$ and $\varphi$ is an isomorphism. Here $A_{\inf}(R) := W(R^{\flat})$ is the ring of Witt vectors of the tilt $R^{\flat}$ of $R$. Hence the Nygaard filtration is simply the $\varphi^{-1}(d)$-adic filtration on $A_{\inf}(R)$. Hence we are in the situation of Chapter~\ref{Sec:ReesVAdic} with $v$ the inclusion $(\varphi^{-1}(d)) \to W(R^{\flat})$. In particular $R^{\Ncal} = \Re(W(R^{\flat})[t,u]/(tu - \varphi^{-1}(d)))$ and Corollary~\ref{GBundlesZAdicFRameStack} holds (note that $W(R^{\flat})$ is a Dedekind domain if and only if $R$ is a field).
\end{example}

\begin{definition}\label{DefPrismaticFGauges}
Let $O$ be a Dedekind domain, let $G$ be a smooth affine group scheme over $O$, and let $\Aline$ be a prism over $O$. A \emph{prismatic $G$-gauge over $\Aline$} is a $G$-bundle on $\Aline^{\Syn}$. If $R$ is a qrsp ring we call a $G$-bundle over $R^{\Syn}$ a \emph{prismatic $G$-gauge over $R$}.
\end{definition}

\begin{remark}\label{NygaardPrismaticFrame}
Let $\Aline = (A,I,\delta)$ be a prism. There is a natural map of frames from the Nygaard frame to the prismatic frame of $\Aline$ which is defined as follows.

We claim that $\phi\colon A \to A$ induces such a map. By definition, $\phi$ sends $\Fil^i_{\Ncal}$ to $I^i$ for all $i$. Hence we obtain a map of Rees stacks $\Re(A,I) \to \Aline^{\Ncal}$. It remains to check that $\phi$ is compatible with $\sigma$. For this we can work locally and assume that $I = (d)$. Then the commutativity of the following diagram shows the desired compatability since the horizontal maps induce the respective maps $\sigma$
\[\xymatrix{
\Fil^i_{\Ncal} \ar[r]^{\phi} \ar[d]_{\phi} & I^i \ar[r]^{a \sends a/d^i} & A \ar[d]^{\phi} \\
I^i \ar[r]^{a \sends a/d^i} & A \ar[r]^{\phi} & A
}\]
for $i \geq 0$.

This yields a map
\begin{equation}\label{EqMapDispSyn}
\Aline^{\Disp} \lto \Aline^{\Syn}
\end{equation}
which is functorial in $\Aline$.

By composition with the open immersion $\Aline^{\BK} \to \Aline^{\Disp}$ we obtain a map $\Aline^{\BK} \to \Aline^{\Syn}$. Pullback via this map yields a functor from the groupoid of prismatic $G$-gauges over $\Aline$ to the groupoid of Breuil--Kisin $G$-bundles. If $\Aline$ is the prism attached to a quasi-regular semiperfectoid ring, then this is the functor constructed in \cite[8.2.11]{Ito_PrismaticGDisplaysDescent} and applying the general criterion \cite[1.7]{Wedhorn_ExtendBundles} to the Rees stacks and then using Corollary~\ref{GBundlesFrameStack} one sees that this functor is fully faithful for flat affine group schemes $G$, recovering (and slightly generalizing) Ito's result.
\end{remark}

\appendix

\section{Colimits of Adams stacks}\label{APPADAMS}

In this appendix we recall the complete and cocomplete 2-category of Adams stacks introduced by Sch\"appi. We prove that a Ferrand pushout of Adams schemes is a pushout in the category of Adams stacks. Recall the definition of an Adams stack.

\begin{definition}\label{DefAdamsStack}
Let $R$ be a ring. A qcqs fpqc-stack $\Xcal$ on $\Affrel{R}$ is called \emph{Adams stack} if it satisfies the following conditions.
\begin{definitionlist}
\item\label{DefAdamsStacka}
$\Xcal$ has an affine diagonal.
\item\label{DefAdamsStackb}
$\Xcal$ admits a faithfully flat morphism from an affine scheme
\item\label{DefAdamsStackc}
$\Xcal$ has the resolution property (i.e., every quasi-coherent $\Oscr_{\Xcal}$-module is the quotient of a filtered colimit of vector bundles).
\end{definitionlist}

An $R$-scheme is called \emph{Adams scheme} if it is an Adams stack if considered as fpqc-stack.
\end{definition}

\begin{remark}\label{AdamsStackHopfAlgebroid}
Let $R$ be a ring. Let $\Xcal$ be an fpqc-stack on $\Affrel{R}$ that satisfies Conditions~\ref{DefAdamsStacka} and~\ref{DefAdamsStackb} from Definition~\ref{DefAdamsStack}. Choose a faithfully flat morphism $X_0 = \Spec A \to \Xcal$ and set $X_1 = X_0 \times_{\Xcal} X_0$. Then $X_1$ is an affine scheme by hypothesis and $(X_0,X_1)$ has a canonical structure of a groupoid with flat source and target maps $X_1 \to X_0$, given by the projections. Conversely, every groupoid consisting of affine schemes with flat source and target maps defines an fpqc-stack satisfying Conditions~\ref{DefAdamsStacka} and~\ref{DefAdamsStackb} \cite[Section 3.3]{Naumann_StackFormalGroups}.
\end{remark}

\begin{example}\label{AmpleAdams}
Every quasi-compact scheme that has an ample family of line bundles is an Adams scheme as such schemes are automatically quasi-separated \cite[0FXS]{Stacks} and have the resolution property \cite[0GMM]{Stacks}.
\end{example}

The following two results give plenty of examples of Adams stacks.

\begin{proposition}\label{QuotientStackAdam}
Let $S$ be an affine scheme and let $X$ be a quasi-affine scheme over $S$. Let $G$ be a geometrically reductive group scheme over $S$ that admits a closed embedding into $\GL_m$ for some $m$. Then the quotient stack $[G \backslash X]$ is an Adams stack over $S$.
\end{proposition}

\begin{proof}
Since $X$ is quasi-affine, it has an affine diagonal. As $G \to S$ is affine, $[G \backslash X]$ has an affine diagonal and is a stack for the fpqc topology by \cite[Section A.3]{Wedhorn_ExtendBundles}. It has the resolution property by \cite[Example A.29]{Wedhorn_ExtendBundles}.
\end{proof}

\begin{example}(\cite[Example A.28]{Wedhorn_ExtendBundles})\label{BGAdams}
Let $R$ be a regular ring of dimension $\leq 1$ and let $G$ be an affine flat group scheme of finite type over $R$. Then $\B{G}$ is an Adams stack over $R$.
\end{example}

We learned the following lemma from Eike Lau.

\begin{lemma}\label{AffineAdamsStack}
Let $f\colon \Xcal \to \Ycal$ be a relatively affine morphisms of prestacks and suppose that $\Ycal$ is an Adams stack. Then $\Xcal$ is an Adams stack.
\end{lemma}

\begin{proof}
Choose a faithfully flat morphism $Y_0 = \Spec A \to \Ycal$ and let $Y_1 = Y_0 \times_{\Ycal} Y_0$. Since $f$ is affine, $X_0 := Y_0 \times_{\Ycal} \Xcal$ and
\[
X_1 := Y_1 \times_{\Ycal} \Xcal = X_0 \times_{\Xcal} X_0
\]
are affine schemes and $(X_0,X_1)$ has the structure of a groupoid with flat source and target map such that $\Xcal$ is the associated fpqc-stack (Remark~\ref{AdamsStackHopfAlgebroid}). Therefore $\Xcal$ satisfies Conditions~\ref{DefAdamsStacka} and~\ref{DefAdamsStackb} from Definition~\ref{DefAdamsStack}.

It remains to check that $\Xcal$ has the resolution property. Let $\Fscr$ be a quasi-coherent module on $\Xcal$. By our hypothesis on $\Ycal$, there exists a surjective homomorphism $\bigoplus \Escr_i \to f_*\Fscr$, where $(\Escr_i)_i$ is a family of vector bundles on $\Ycal$. We obtain a surjective homomorphism $\bigoplus f^*\Escr_i \to f^*f_*\Fscr \to \Fscr$, where the second map is surjective since $f$ is affine.
\end{proof}

We have the following theorem by Sch\"appi.

\begin{theorem}(\cite[1.1.1]{Schaeppi_ConstructingColimits} and \cite[3.2.1, 3.2.6, 3.1.7]{Schaeppi_DescentTannaka})\label{AdamsStackCocomplete}
Let $R$ be a ring. The 2-category of Adams stacks over $R$ is complete and cocomplete. Limits are formed as limits of prestacks.

Colimits can be recognized as follows. Let $\Ical$ be a small connected category, let $\Xcal\colon i \sends \Xcal_i$ be a diagram of non-empty Adams stacks indexed by $\Ical$, and let $(f_i\colon \Xcal_i \to \Wcal)$ be a pseudococone\footnote{We suppress 2-morphisms in the notation.} in the 2-category of Adams stacks on the diagram $\Xcal$. Suppose that the following conditions hold.
\begin{definitionlist}
\item\label{AdamsStackCocomplete1}
The pseudocone obtained by applying $\Vec(-)$ to the pseudococone $(f_i)$ induces an equivalence of $R$-linear categories
\[
\Vec(\Wcal) \liso \lim_i \Vec(\Xcal_i).
\]
\item\label{AdamsStackCocomplete2}
There exists a subset $I_0$ of $\Ob(\Ical)$ such that for any $\Ical$-diagram $i \sends \Ccal_i$ of categories the functor $\lim \Ccal_i \to \prod_{i\in I_0}\Ccal_i$ reflects isomorphisms and such that for any module $\Escr$ of finite presentation over $\Wcal$ with $f_i^*\Escr = 0$ for all $i \in I_0$ one has $\Escr = 0$.
\end{definitionlist}
Then the pseudocone induces an equivalence $\colim_i \Xcal_i \liso \Wcal$. Moreover, Condition~\ref{AdamsStackCocomplete1} is also necessary.
\end{theorem}

%---------------------------------------------------------------------
%
%\subsection{Ferrand pushouts and Adams stacks}

\begin{corollary}\label{FerrandAdamsPushout}
Let
\[\xymatrix{
Z \ar[r]^i \ar[d]_g & X \\
Z'
}\]
be a diagram of Adams schemes. Suppose that $i$ is a closed immersion and that $g$ is an integral morphism such that every fiber of $g$ is contained in an open affine subscheme of $Z$. Consider the Ferrand pushout of this diagram \cite[7.1]{Ferrand_Conducteur}
\[\xymatrix{
Z \ar[r]^i \ar[d]_g & X \ar[d]^f \\
Z' \ar[r]^{i'} & X'.
}\]
Suppose that $X'$ is an Adams scheme. Then this is a pushout diagram in the 2-category of Adams stacks.
\end{corollary}

It is not clear to us whether the other hypotheses already imply that $X'$ is an Adams scheme.

\begin{proof}
We use Theorem~\ref{AdamsStackCocomplete}. Condition~\ref{AdamsStackCocomplete1} is satisfied by \cite[6.3.3]{TemkinTyomkin_Ferrand} and Condition~\ref{AdamsStackCocomplete1} holds by a globalization of \cite[2.2, ii)]{Ferrand_Conducteur}, see  \cite[7.4]{Ferrand_Conducteur}.
\end{proof}

%=====================================================================

\section{A flatness criterion}

We will also use the following flatness criterion for which we could not find a reference. It is of course well known for noetherian rings.

\begin{proposition}\label{FlatnessCrit}
Let $A$ be a ring, $I \subseteq A$ an ideal, let $M$ be an $A$-module of finite presentation, and let $\Escr$ be the corresponding quasi-coherent module on $X := \Spec A$. Set $U := X \setminus V(I)$. Then $\Escr$ is a vector bundle if and only if the following conditions are satisfied.
\begin{definitionlist}
\item\label{FlatnessCritb}
$M/IM$ is a flat $A/I$-module and $\Tor^A_1(A/I,M) = 0$.
\item\label{FlatnessCritc}
$\Escr\rstr{U}$ is a flat $\Oscr_U$-module.
\end{definitionlist}
%\begin{equivlist}
%\item\label{FlatnessCriti}
%$M$ is a flat $A$-module.
%\item\label{FlatnessCritii}
%$M/IM$ is a flat $A/I$-module and $\Tor^A_1(A/I,M) = 0$.
%\item\label{FlatnessCritiii}
%$M/I^nM$ is a flat $A/I^n$-module for all $n \geq 1$.
%\end{equivlist}
%Then the implications ``\ref{FlatnessCriti} $\implies$ \ref{FlatnessCritii} $\implies$ \ref{FlatnessCritiii}'' hold. If $M$ is of finite presentation, then \ref{FlatnessCriti} and \ref{FlatnessCritii} are equivalent.
\end{proposition}

\begin{proof}
Recall that $\Escr$ is a vector bundle if and only if $\Escr$ is flat and of finite presentation. Hence the conditions are necessary by standard flatness criterion (e.g., \cite[III,5.2]{BouACI-VII}). Now suppose that the two conditions hold. We have to show that for every maximal ideal $\mfr \in V(I)$ the $A_{\mfr}$-module $M_{\mfr}$ is flat. As Tor commutes with flat base change we can replace $A$ by $A_{\mfr}$, $I$ by $I_{\mfr}$, and $M$ by $M_{\mfr}$. Therefore we may assume that $A$ is local and that $I \ne A$. In this case the result follows from \cite[0471]{Stacks}.
\end{proof}

\begin{remark}\label{RemFlatnessCrit}
The proof shows that one can replace Condition~\ref{FlatnessCritc} by the condition that $\Escr_x$ is a flat $\Oscr_{X,x}$-module for every closed point $x$ of $\Spec A$ contained in $U$. This conditions is for instance empty if $I$ is contained in the Jacobson radical of $A$.
\end{remark}

\begin{remark}\label{GeneralFlatnessCrit}
The result \cite[0471]{Stacks} used in the proof is much more general than the version needed to prove Proposition~\ref{FlatnessCrit}. In fact a similar proof shows the following more general flatness criterion, which we will not need in the paper.

Let $A$ be a ring, $I \subseteq A$ an ideal, and let $A \to B$ be an $A$-algebra of finite presentation. Set $X := \Spec A$, $Y := \Spec B$, and let $f\colon Y \to X$ be the corresponding map of affine schemes. Let $M$ be a $B$-module of finite presentation, and let $\Escr$ be the corresponding quasi-coherent module on $Y$. Set $U := X \setminus V(I)$. Suppose that the following conditions are satisfied.
\begin{definitionlist}
\item
$M/IM$ is a flat $A/I$-module and $\Tor^A_1(A/I,M) = 0$.
\item\label{FlatnessCritc2}
$\Escr\rstr{f^{-1}(U)}$ is a flat $\Oscr_U$-module.
\end{definitionlist}
Then $\Escr$ is flat over $X$.
\end{remark}

%
%Let $S$ be an algebraic space and let $\Xcal$ be a stack (not necessarily algebraic) over $S$ that admits an fpqc covering $\pi\colon X \to \Xcal$ where $X = \Spec A$ is an affine scheme. Let $i\colon \Zcal \to \Xcal$ be a map of stack that is representable by a closed immersion and let $I \subseteq A$ be the ideal defined by the cartesian diagram
%\[\xymatrix{
%\Spec(A/I) \ar[r] \ar[d] & \Spec A \ar[d]^{\pi} \\
%\Zcal \ar[r]^i & \Xcal.
%}\]
%Let $\Escr$ be a quasi-coherent $\Oscr_{\Xcal}$-module of finite presentation and let $M$ be the $A$-module of finite presentation corresponding to $\pi^*(

%\input{AppBTruncatedShtukas.tex}
\section{Vector bundles on the display stack and {$\GL_h$}-displays\\by Christopher Lang}\label{APPLang}

In this appendix, $R$ denotes a $p$-complete ring. Let $0 \leq d \leq h$ be integers and consider the standard minuscule cocharacter $\mu$ of $\GL_h$ given by
\begin{equation}\label{EqDefStandardMuAppC}
\mu\colon t \sends \twomatrix{tI_d}{}{}{I_{h-d}}
\end{equation}
Our goal is to show that the category of vector bundles of rank $h$ of type $\mu$ on $R^{\Disp}$ in the sense of \ref{DefGDisplay} is equivalent to the category of $(\GL_h,\mu)$-displays defined by B\"ultel and Pappas in \cite{BP_DisplaysRZSpaces}, see Proposition~\ref{CompareBueltelPappas}. The latter category is very concisely described by Hoff in \cite{Hoff_EKORSiegel}. Let us recall the definitions first. For any stack $\Xcal$ we denote by $\Vec^h(\Xcal)$ the category of vector bundles of rank $h$.

%---------------------------------------------------------------------

\subsection{Pairs and Displays of type $(h,d)$}

\begin{definition}[{\cite[2.2]{Hoff_EKORSiegel}}]
A \emph{pair} of type \((h,d)\) over \(R\) is a tuple \((M,M_1)\), where \(M\) is a finite projective \(W(R)\)-module of rank \(h\), and \(M_1\subseteq M\) is a \(W(R)\)-submodule with \(I_R M\subseteq M_1\) and such that \(M_1/I_R M \subseteq M/I_R M\) is a rank \(d\) direct summand.

We can make this into a category \(\mathrm{Pair}^{h,d}(R)\) by setting morphisms \((M,M_1)\to(M',M'_1)\) to be morphisms \(f:M\to M'\) with \(f(M_1)\subseteq M'_1\). (Equivalently: \(\overline f(M_1/I_R M)\subseteq M'_1/I_R M'\))
\end{definition}

For such a pair one has the following construction.

\begin{proposition}[{\cite[2.7]{Hoff_EKORSiegel}}]
\label{prop_hoff_def_M1_tilde}
There is a functor 
\begin{align*}
    \mathrm{Pair}^{h,d}(R) &\to \mathrm{Vec}^h\big(W(R)\big)\\
    (M,M_1)\cong (L\oplus T,L\oplus I_R T) &\mapsto \widetilde M_1 \colon= L^\phi \oplus T^\phi\\
    \begin{pmatrix}a&b\\c&d\end{pmatrix} &\mapsto
    \begin{pmatrix}a^\phi &p b^\phi\\\dot c&d^\phi\end{pmatrix}
\end{align*}
with \(\dot c: L^\phi \xrightarrow{c^\phi} I_R^\phi \underset{W(R)}{\otimes} T'^\phi \xrightarrow{\phi^\mathrm{div}\otimes\id} T'^\phi\). Here $M = L \oplus T$ is a normal decomposition of the pair $(M,M_1)$, see \cite[2.6]{Hoff_EKORSiegel}.
\end{proposition}

\begin{definition}[{\cite[2.9]{Hoff_EKORSiegel}}]
A \emph{display} \((M,M_1,\Psi)\) over \(R\) of type \((h,d)\) consists of a pair \((M,M_1)\) of type \((h,d)\) and an isomorphism \(\Psi: \widetilde M_1 \to M\). A morphism of displays is a morphism of pairs such that 
\[\begin{tikzcd}
\widetilde M_1 \rar["\Psi"] \dar & M \dar \\
\widetilde M_1' \rar["\Psi'"] & M'
\end{tikzcd}\]
commutes. We call the resulting category \(\mathrm{Disp}^{h,d}(R)\).
\end{definition}

The goal of this appendix is now to show that these constructions can be also described via the coequalizer diagram of stacks
\[
\xymatrix{\Spec W(R)
\ar@<.5ex>[r]^-(.4){\tau}\ar@<-.5ex>[r]_-(.4){\sigma} & \Re(W(R))} \lto R^{\Disp}
\]
constructed in Section~\ref{Chapter:TruncatedDisplays}.

%---------------------------------------------------------------------

\subsection{Comparison of different notions of displays}

Denote by $\Vec^h\big(\Re(W(R))\big)^\mu$ the category of vector bundles of rank $h$ and of type $\mu$ over $\Re(W(R))$ with $\mu$ as defined in \eqref{EqDefStandardMuAppC}. Similarly define $\Vec^h\big(R^{\Disp})\big)^\mu$.

\begin{proposition}\label{CompareBueltelPappas}
There are equivalences of categories which are functorial in $R$.
\begin{align*}
\Vec^h\big(\Re(W(R))\big)^\mu &\liso \mathrm{Pair}^{h,d}(R),\\
\Vec^h\big(R^{\Disp})\big)^\mu &\liso \mathrm{Disp}^{h,d}(R).
\end{align*}
Via these equivalences
\begin{assertionlist}
\item
pullback of vector bundles via $\Re(W(R)) \to R^{\Disp}$ is identified with the functor sending a display to its underlying pair,
\item
pullback via $\tau$ is identified with sending a pair $(M,M_1)$ to its underlying $W(R)$-module $M$,
\item
pullback via $\sigma$ is identified with sending a pair $(M,M_1)$ to the $W(R)$-module $\Mtilde_1$.
\end{assertionlist}
\end{proposition}

To prove the proposition, we first spell out what it means for a vector bundle of $\Re(W(R))$ to be of type $\mu$.

\begin{remark}\label{DseribeVBOfTypeMuWitt}
By \cite[2.6]{Wedhorn_ExtendBundles}, we have 
\[
H^1(\Re(W(R)),\GL_h) \xrightarrow{\sim} H^1(B\GG_{m,R},\GL_h),
\]
and objects on the right are (classes of) finite projective graded modules on \(R\) (where \(R\) has the trivial grading), and these are isomorphic to \(\bigoplus_{i\in\ZZ} \overline L_i \otimes_R R(i)\), where the \(\overline L_i\) are vector bundles on \(R\), with ranks adding to \(h\). Now we only take such objects with \(\rank\overline L_1 = h-d\) and \(\rank\overline L_0 = d\).  The vector bundle on \(\Re(W(R))\) will then be isomorphic to \(\bigoplus_{i\in\ZZ} L_i \otimes_{W(R)} \Rees\bigl(W(R)\bigr)(i)\), where \(L_i\) is a lift of \(\overline L_i\) to a \(W(R)\)-module, unique up to isomorphism, using that \((W(R),I_R)\) is Henselian.

Let us elaborate on the notation \(L \otimes_{W(R)} \Rees\bigl(W(R)\bigr)(i)\), as we will need it often in the sequel. We first make \(L\) into a \(\Rees\bigl(W(R)\bigr)\)-module and then shift by \(i\). In particular, we have
\[
    \left(L\tensor{W(R)}\Rees\bigl(W(R)\bigr)(i)\right)_j 
    = \begin{cases}
        L, & j \le -i\\
        I_R L, & j > -i,
    \end{cases}
\]
where multiplication with \(t\) is given by the identity for \(j\le -i\),  the inclusion \(I_R L\hookrightarrow L\) for \(j=-i+1\) and multiplication with \(p\) otherwise. Multiplication with \(V(x)t^{-k}\) is given by \(V(x)t^{-k}\cdot l = V(x)l\) for \(j\le-i\) and by \(V(x)t^{-k}\cdot V(y)l = V(xy)l\) otherwise.
%The morphisms of vector bundles are morphisms of graded modules, i.e.\ \(f(N_i)\subseteq N'_i\). Transition maps for \(R\to S\) are given by 
%\begin{align*}
%    \mathrm{VB}^h\big(\Re(W(R))\big)^\mu &\to \mathrm{VB}^h\big(\Re(W(S))\big)^\mu \\
%    N &\mapsto N \underset{\Rees(W(R))}{\otimes} \Rees\bigl(W(S)\bigr)
%    \text{and}
%    \mathrm{Pair}^{h,d}(R) &\to \mathrm{Pair}^{h,d}(S)\\
%    (M,M_1) &\mapsto (M', M'_1),
%\end{align*}
%where the pair is characterized by 
%\[M' =  M\underset{W(R)}{\otimes} W(S)\qquad\text{and}\qquad M_1'/I_S M' = (M_1/I_R M)\underset{R}{\otimes} S,\]
%see \cite[1.4]{hoff2023}, and a morphism \(f\) is simply mapped to \(f\underset{W(R)}{\otimes}W(S)\).
\end{remark}

Next we define a functor
\begin{equation}\label{EqFunctorPairs}
\begin{aligned}
\mathrm{Vec}^h\big(\Re(W(R))\big)^\mu &\to \mathrm{Pair}^{h,d}(R), \\
N &\mapsto (N_{-1},N_0)\\
f &\mapsto f_{-1}=f|_{N_{-1}}
\end{aligned}
\end{equation}

Let us first check that this is well defined: We first show that \(N_{-1}\) is a projective \(W(R)\)-module of rank $h$. Note that it is a \(W(R)\)-module as \(W(R)\) is the degree 0 part of \(\Rees\bigl(W(R)\bigr)\). It is finite projective because we can identify 
\[
N \cong \left(L_0 \underset{W(R)}{\otimes} \Rees\bigl(W(R)\bigr)\right) \oplus \left(L_1 \underset{W(R)}{\otimes}\Rees\bigl(W(R)\bigr)(1)\right),
\]
where \(L_0\) and \(L_1\) are projective \(W(R)\)-modules of rank \(d\) resp \(h-d\). Then we get 
\[
N_{-1} \cong L_0 \oplus L_1
\]
as \(W(R)\)-modules, and this is projective of rank \(h\).

Next we need to identify \(N_0\) as a submodule of \(N_{-1}\). This is done by the multiplication with \(t\) map. This map is injective, since with the identification above we have
\[
N_0 \cong L_0 \oplus I_R L_1,
\]
as \(I_R L_1 = L_1\tensor{W(R)}\Rees\bigl(W(R)\bigr)_1\), so \(L_0\) and \(L_1\) form a normal decomposition of \((N_{-1},N_0)\).

The condition \(I_R N_{-1} \subseteq N_0\) follows because \(I_R\) naturally lives in degree 1 inside \(\Rees\bigl(W(R)\bigr)\). Alternatively, this also immediately follows from the identifications above. We still need that \(N_0/I_R N_{-1} \subseteq N_{-1}/I_R N_{-1}\) is a direct summand of rank \(d\). But this we get from the identification, as 
\[N_0/I_R N_{-1} \cong L_0/I_R L_0 \qquad\text{and}\qquad N_{-1}/I_R N_{-1} \cong L_0/I_R L_0 \oplus L_1/I_R L_1.\]
Well-definedness on morphisms is clear, using that \(f_{-1}(t\cdot N_0) = t\cdot f_0(N_0) \subseteq t\cdot N_0'\).

It is easy to see that the functor \eqref{EqFunctorPairs} is functorial in $R$. Next we show:

\begin{lemma}\label{IdentifyPairs}
The functor \eqref{EqFunctorPairs} is an equivalence.
\end{lemma}

\begin{proof}
Let us first show that the functor is essentially surjective. Given a pair \((M,M_1)\), we can take a normal decomposition \(M=L\oplus T\) and \(M_1=L\oplus I_R T\). Then
\[
N \colon= \left(L \underset{W(R)}{\otimes} \Rees\bigl(W(R)\bigr)\right) \oplus \left(T \underset{W(R)}{\otimes}\Rees\bigl(W(R)\bigr)(1)\right)
\]
is a preimage of this pair.
\begin{comment}
Given a pair \((N_{-1},N_0)\), we make it into a graded \(\Rees\bigl(W(R)\bigr)\)-module \(N\) in the following way: We set \(N_i=N_{-1}\) for \(i<0\) and \(N_i =I_R N_{-1}\) for \(i>0\). Then \(\bigoplus N_i\) clearly is a \(W(R)\)-module, so to make it into a \(\Rees\bigl(W(R)\bigr)\)-module, we need to say how to multiply with \(t\) and with \(V(x)t^{-j}\in\Rees\bigl(W(R)\bigr)_j\) for \(j>0\). 
\begin{align*}
    \cdot t &= 
    \begin{cases}
    N_i = N_{-1} \to N_{i-1} = N_{-1},\; n\mapsto n, & i<0\\
    N_0\hookrightarrow N_{-1}, & i=0\\
    N_1 = I_R N_{-1} \hookrightarrow N_0, & i=1\\
    N_i = I_R N_{-1} \to N_{i-1} = I_R N_{-1},\; n\mapsto pn, & i>0
    \end{cases}\\[10pt]
    \cdot V(x)t^{-j} &= 
    \begin{cases}
    N_i \to N_{i+j},\; n \mapsto V(x)n\,(\in I_R N_{-1}\subseteq N_0\subseteq N_{-1}), & i\le0\\
    N_i=I_R N_{-1} \to N_{i+j}=I_R N_{-1},\; V(y)n \mapsto V(xy)n, & i>0\\
    \end{cases}
\end{align*}
It remains to check that \(N\) is a vector bundle (of type \(\mu\)).
\end{comment}

Next we show that the functor is full. Let's take two graded modules \(N, N'\) of type $\mu$, with decompositions
\begin{alignat*}{3}
    &N &&\cong \left(L \underset{W(R)}{\otimes} \Rees\bigl(W(R)\bigr)\right) &&\oplus \left(T \underset{W(R)}{\otimes}\Rees\bigl(W(R)\bigr)(1)\right),\\
    &N' &&\cong \left(L' \underset{W(R)}{\otimes} \Rees\bigl(W(R)\bigr)\right) &&\oplus \left(T' \underset{W(R)}{\otimes}\Rees\bigl(W(R)\bigr)(1)\right).
\end{alignat*}
A map of pairs corresponds to a quadruple
\[a:L\to L',\quad b:T\to L',\quad c:L\to I_R T',\quad d:T\to T',\]
and this data can be made into a map \(N\to N'\):
\begin{alignat*}{3}
    &a\otimes \id &&: L \underset{W(R)}{\otimes} \Rees\bigl(W(R)\bigr) &&\to L' \underset{W(R)}{\otimes} \Rees\bigl(W(R)\bigr)\\
    &d\otimes \id &&: T \underset{W(R)}{\otimes}\Rees\bigl(W(R)\bigr)(1) &&\to T' \underset{W(R)}{\otimes}\Rees\bigl(W(R)\bigr)(1)\\
    &b\otimes \cdot t &&: T \underset{W(R)}{\otimes} \Rees\bigl(W(R)\bigr)(1) &&\to L' \underset{W(R)}{\otimes} \Rees\bigl(W(R)\bigr),
\intertext{and for the last map we note that each graded piece of \(T'\underset{W(R)}{\otimes} \Rees\bigl(W(R)\bigr)(1)\) is either \(T'\) (in degrees \(\le-1\)), or \(I_R T'\). Therefore we can define a map} 
&\tilde c &&: L \underset{W(R)}{\otimes} \Rees\bigl(W(R)\bigr) &&\to T' \underset{W(R)}{\otimes}\Rees\bigl(W(R)\bigr)(1),
\end{alignat*}
which is induced by \(c\) in degrees \(\le 0\), and in degrees \(\ge 1\) by 
\[I_R L \to I_R T',\qquad V(x)\cdot l \mapsto \sum_{i=1}^{n_l} V(x y_i)\cdot a_i,\qquad(\text{with }c(l)=\sum_{i=1}^{n_l} V(y_i)\cdot a_i\in I_R T').\]
We need to check that this is well-defined, namely that two different representations of \(c(l)\) yield the same result. By subtraction, it is enough to show that if \(c(l)=0\), then \(V(x)\cdot l\) maps to \(0\). If \(p\) is not a zero-divisor in \(W(R)\), we can multiply by \(p\), use \(pV(xy_i)=V(x)V(y_i)\) and pull \(V(x)\) outside the sum. Otherwise, the situation seems to be more complicated. We would like to use \(V(y_i)a_i=V(y_i F(a_i))\), but we only have the Frobenius in \(W(R)\), not on \(T'\), so we first need to embed \(T'\) into a free module: We can look at 
\[\begin{tikzcd}
    \hat c:\quad L \rar["c"] & I_R T' \rar[hook,"\iota"] & I_R^m = I_R T' \oplus I_R \tilde T',
\end{tikzcd}\]
where \(T'\oplus\tilde T'\) is free. We can write the \(j\)-th component of this map as \(l\mapsto V(c_j(l))\). We can make this into a map 
\[\tilde c': I_R L\to I_R^m,\quad V(x)\cdot l\mapsto \biggl(V\bigl(x c_j(l)\bigr)\biggr)_{1\le j\le m},\]
which is clearly well-defined. Once we show that \(\tilde c' = \iota\circ \tilde c\), we know that \(\tilde c\) is also well-defined. This equation holds because we have
\[\hat c(l)=\sum_i V(y_i)a_i=\left(\sum_i V(y_i)a_{ij}\right)_j = \biggl( V\bigl(\underbrace{\sum_i y_i F(a_{ij})}_{=c_j(l)}\bigr)\biggr)_j,\]
where \(a_{ij}\) are the components of \(a_i\) in \(I_R^m\), so we can compute
\begin{align*}
    \tilde c'(V(x)l)_j 
    &= V(xc_j(l)) = V\left(\sum_i x y_i F(a_{ij})\right)\\
    \iota\circ\tilde c(V(x)l)_j
    &= \iota\left(\sum_{i} V(x y_i)\cdot a_i\right)_j 
    =\sum_{i} V(x y_i)\cdot a_{ij}
    = V\left(\sum_i x y_i F(a_{ij})\right).
\end{align*}
This shows that the functor is full.

It remains to see that the functor is faithful. We need to show that the restriction map
\[g: \Hom_{\Rees(W(R))}(N,N') \to \Hom_{W(R)}(N_{-1},N'_{-1})\]
is injective. 
For this we take \(W(R)\)-modules \(\tilde L, \tilde L', \tilde T, \tilde T'\) such that \(L\oplus \tilde L\) etc are all free (of rank \(n_L,n_{L'},n_T,n_{T'}\)). Then we look at the map
\[g\oplus\tilde g: \Hom_{\Rees(W(R))}(N\oplus \tilde N, N'\oplus \tilde N') \to \Hom_{W(R)}((N\oplus \tilde N)_{-1},(N'\oplus \tilde N')_{-1}),\]
where \(\tilde N\) resp.\ \(\tilde N'\) are defined with \(\tilde L, \tilde T\) resp.\ \(\tilde L', \tilde T'\), and observe that 
\[N\oplus \tilde N = \Rees\bigl(W(R)\bigr)^{n_L}\oplus \Rees\bigl(W(R)\bigr)^{n_T}(1).\]
Using \(\Hom_R(A^n,B^m)=\Hom_R(A,B)^{n\times m}\), we can eliminate the products, as \(g\oplus\tilde g\) maps each matrix entry separately. So we need to check that maps of the form
%Höchstwahrscheinlich quatsch:
% We interpret both sides as \(W(R)\)-modules and want to test this locally. For this we take \(f_1,\dots,f_n\in R\) that generate \(R\) such that the \(R\)-modules corresponding to the \(W(R)\)-modules \(T,T',L,L'\) are all free on \(D(f_i)\). By the henselian property of \(W(R_{f_i})\), the corresponding \(W(R_{f_i})\)-modules are free. 
%
% Then we use the fact that \(W(R_{f_i}) = W(R)_{[f_i]}\), where \([f_i]\in W(R)\) is the Teichmüller lift of \(f_i\in R\). {\color{blue}(Why? Zink claims this in the proof of Prop 35 of The Display of a Formal p-Divisible Group.)} Then the \([f_i]\) generate \(W(R)\), as the linear combination of \(1\) with the \(f_i\) Teichmüller lifts to a linear combination of the \([f_i]\) to an element, which has as the first Witt component \(1\), hence a unit {\color{blue}(???)}. Hence, the \(W(R)_{[f_i]}\)-modules  \(L_{[f_i]}\), \(T_{[f_i]}\), \(L'_{[f_i]}\) and \(T'_{[f_i]}\) are free, so we have
% \[N_{[f_i]}=\Rees\big(W(R_{f_i})\big)^{n_1}\oplus \Rees\big(W(R_{f_i})\big)^{n_2}(1)\quad\text{(as \(\Rees\big(W(R_{f_i})\big)\)-modules)},\]
% so it is enough to show that maps of the form
\begin{align*}
\Hom_{\Rees(W(R))}\Big(\Rees\bigl(W(R)\bigr)(\epsilon),\Rees\bigl(W(R)\bigr)(\delta)\Big) &\to \Hom_{W(R)}\big(W(R),W(R)\big)=W(R),\\
f&\mapsto f|_{-1}
\end{align*}
are injective for \(\epsilon,\delta\in\{0,1\}\). But maps on the left are determined by the image \(f(1)\) of the element \(1\in\Rees\bigl(W(R)\bigr)\), and \(f(1)\) has to be a homogeneous of degree \(\epsilon-\delta\), i.e.\ \(-1\), \(0\) or \(1\). Hence the map is in fact the natural map from the degree \(-1\), \(0\) or \(1\) of \(\Rees\bigl(W(R)\bigr)\) to \(W(R)\), which is indeed injective. Hence \(g\oplus \tilde g\) is injective, and therefore \(g\) as well.
\end{proof}

To finish the proof of Proposition~\ref{CompareBueltelPappas} it now suffices to show the following lemma.

\begin{lemma}\label{IdentifySigmaTau}
Via the equivalence \eqref{EqFunctorPairs} the pullback via $\tau$ is identified with the functor $(M,M_1) \sends M$ and the pullback via $\sigma$ is identified with $(M,M_1) \sends \Mtilde_1$.
\end{lemma}

Let us recall the definition of $\tau, \sigma\colon \Spec W(R) \to \Re(W(R))$ from Section~\ref{Section:TruncatedDisplayStack}. By the description of quotient stacks, both maps correspond to line bundles \(L\) on \(\Spec W(R)\) together with equivariant maps \(L\to \Spec\Rees\bigl(W(R)\bigr)\), or in the language of modules, a rank 1 projective module \(A\) over \(W(R)\) and a graded map \(\Rees\bigl(W(R)\bigr)\to \bigoplus_{i\in\ZZ}A^{\otimes i}\cdot t^{-i}\). For our two special points we choose \(A=W(R)\) both times, and for the graded maps we have 
\[\begin{NiceArray}{rr@{\;}c@{\;}lw{r}{1.7cm}@{\;}ll}
\Block{2-1}{\tilde\tau:}&\Block{2-1}{\Rees\bigl(W(R)\bigr)} &\Block{2-1}{\to}& \Block{2-1}{W(R)[t^{\pm1}],}& xt^{-i} &\mapsto xt^{-i},& i\le 1\\
& & & & xt^{-i} &\mapsto pxt^{-i}, & i\ge2\\
& \\
\Block{2-1}{\tilde\sigma:}&\Block{2-1}{\phi^*\Rees\bigl(W(R)\bigr)} &\Block{2-1}{\to}& \Block{2-1}{W(R)[t^{\pm1}],}& xt^{-i} &\mapsto \phi(x)p^{-i} t^{-i}, & i\le 0 \\
& & & & V(x)t^{-i} &\mapsto xt^{-i}, & i\ge1
\CodeAfter
\SubMatrix{\{}{1-5}{2-5}{.}
\SubMatrix{\{}{4-5}{5-5}{.}
\end{NiceArray}\]
In Section~\ref{Section:TruncatedDisplayStack}, $\tilde{\sigma}$ is denoted by $\sigma^*$.
%The corresponding \(W(R)\)-valued points of \(\Re\big(W(R)\big)\) are called \(\tau\) and \(\sigma\). Then we want to consider the coequalizer \(\Coeq(\tau,\sigma)\). We have
%\[\mathrm{VB}^h\big(\Coeq(\tau,\sigma)\big)^\mu = \set{(N,\psi)}{N\in\mathrm{VB}^h\big(\Re(W(R))\big)^\mu,\, \psi:\sigma^*N \xrightarrow{\sim}\tau^*N},\]
%where the morphisms are the morphisms of vector bundles that make the obvious diagram commute.
%
%\begin{theorem}
%\label{thm_equiv_disp_vbRees}
%    We have an equivalence of categories
%    \[\mathrm{VB}^h\bigl(\Coeq(\tau,\sigma)\bigr)^\mu \cong \mathrm{Disp}^{h,d}(R).\]
%\end{theorem}

\begin{proof}
If \(N = \left(L \underset{W(R)}{\otimes} \Rees\bigl(W(R)\bigr)\right) \oplus \left(T \underset{W(R)}{\otimes}\Rees\bigl(W(R)\bigr)(1)\right)\), then we have to show
\begin{align*}
    \tau^*N &\quad\text{corresponds to}\quad L\oplus T \quad\text{and}\\
    \sigma^*N &\quad\text{corresponds to}\quad L^\phi \oplus T^\phi.
\end{align*}
To see this, we look at the diagram 
\[\begin{tikzcd}
    & \QCoh{[\Spec W(R)[t^{\pm1}]/\GG_m]} \\
    \QCoh{\Spec W(R)} \ar[ur,"M\mapsto \bigoplus_{i\in\ZZ}M\cdot t^i","\sim"' sloped] & & {\QCoh{\Re(W(R))}.} \ar[ul,"{N\otimes \tilde\tau \mapsfrom N}", sloped, near start] \ar[ll,"\tau^*N \mapsfrom N"]
\end{tikzcd}\]
So in order to compute \(\tau^* N\), we only need to compute any degree of \(N\otimes \tilde\tau^*\). But given the decomposition into \(L\) and \(T\), this is an easy task, as 
\[\left(L\tensor{W(R)}\Rees\bigl(W(R)\bigr)\right)\tensor{\Rees(W(R)),\tilde\tau}W(R)[t^{\pm1}] = L\tensor{W(R)}W(R)[t^{\pm1}] = \bigoplus_{i\in\ZZ} L\cdot t^{-i}.\]
(This uses that \(W(R)\to\Rees\bigl(W(R)\bigr)\xrightarrow{\tilde\tau}W(R)[t^{\pm1}]\) is the canonical map \(x\mapsto xt^0\).) For the \(T\)-component basically the same calculation holds true, just the degrees are shifted by 1. This shows \(\tau^*N=L\oplus T\). It remains to show \(\sigma^*N=L^\phi \oplus T^\phi\).

For \(\sigma^* N\) we do the same computation, but with \(\tau\) replaced by \(\sigma\) everywhere.
\[\left(L\tensor{W(R)}\Rees\bigl(W(R)\bigr)\right)\tensor{\Rees(W(R)),\tilde\sigma}W(R)[t^{\pm1}] = L \tensor{W(R),\phi}W(R)[t^{\pm1}] = \bigoplus_{i\in\ZZ} L^\phi\cdot t^{-i}\]
with \(\phi: W(R)\to\Rees\bigl(W(R)\bigr)\xrightarrow{\tilde\sigma}W(R)[t^{\pm1}],\;\;x\mapsto \phi(x)t^0\). Again, up to shift, the same holds for \(T\). This shows the claim, and hence the map \(\Psi\) from displays and \(\psi\) from Rees stacks can be identified with each other through the choice of a normal decomposition. 

We still need to check compatibility on morphisms, i.e.\ that the diagram 
\[\begin{tikzcd}[ampersand replacement = \&]
    L^\phi \oplus T^\phi \rar["\Psi"] \dar["{\left(\begin{smallmatrix}
        a^\phi & pb^\phi \\
        \dot c & d^\phi
    \end{smallmatrix}\right)}"'] \& L \oplus T \dar["{\left(\begin{smallmatrix}
        a&b\\c&d
    \end{smallmatrix}\right)}"] \\[10pt]
    L'^\phi \oplus T'^\phi \rar["\Psi'"] \& L' \oplus T'
\end{tikzcd}\]
commutes if and only if 
\[\begin{tikzcd}[ampersand replacement = \&]
    \sigma^* N \rar["\psi"] \dar["{\sigma^*\left(\begin{smallmatrix}
        a\otimes\id & b\otimes \cdot t \\
        \tilde c & d\otimes \id
    \end{smallmatrix}\right)}"'] \& \tau^* N \dar["{\tau^*\left(\begin{smallmatrix}
        a\otimes\id & b\otimes \cdot t \\
        \tilde c & d\otimes \id
    \end{smallmatrix}\right)}"] \\[10pt]
    \sigma^* N' \rar["\psi'"] \& \tau^* N'
\end{tikzcd}\]
commutes. For this we check that both 
\[\begin{tikzcd}[ampersand replacement = \&]
    L\oplus T \rar["\sim"] \dar["{\left(\begin{smallmatrix}
        a&b\\c&d
    \end{smallmatrix}\right)}"'] \& \tau^* N \dar["{\tau^*\left(\begin{smallmatrix}
        a\otimes\id & b\otimes \cdot t \\
        \tilde c & d\otimes \id
    \end{smallmatrix}\right)}"] \\[10pt]
    L'\oplus T' \rar["\sim"] \& \tau^* N'
\end{tikzcd}\qquad\text{and}\qquad
\begin{tikzcd}[ampersand replacement = \&]
    L^\phi \oplus T^\phi \rar["\sim"] \dar["{\left(\begin{smallmatrix}
        a^\phi & pb^\phi \\
        \dot c & d^\phi
    \end{smallmatrix}\right)}"'] \& \sigma^* N \dar["{\sigma^*\left(\begin{smallmatrix}
        a\otimes\id & b\otimes \cdot t \\
        \tilde c & d\otimes \id
    \end{smallmatrix}\right)}"] \\[10pt]
    L'^\phi \oplus T'^\phi \rar["\sim"] \& \sigma^* N'
\end{tikzcd}\]
commute. The essential step is checking the "\(c\)"-component on the right. This we do by looking at the following diagram, where we use that \(c(l)\) can be written as (a sum of) \(V(y_l)\cdot a_l\), for \(l\in L\). The index 0 or 1 stands for the homogeneous part of the corresponding degree of the graded module. For \(T'\), we have degree 1 due to the shift.

\[\scalebox{0.588}{
\begin{tikzcd}[ampersand replacement = \&]
l\otimes x \arrow[rrr, maps to] \arrow[dd, maps to]       \& \& \& l\otimes x \arrow[dd, maps to] \\
\& {L^\phi = L\tensor{W(R),\phi}W(R)} \arrow[r, "l\otimes x\mapsto l\otimes x","\sim"'] \dar["{\substack{l\otimes x\\ \rotatebox[origin=c]{-90}{\(\mapsto\)}\\ (V(y_l)\otimes x)\otimes (a_l\otimes1)}}"',"c^\phi"] \ar[dd,"\dot c", bend left=60, shift left=12.5] \& {\left(L\tensor{W(R),\phi}W(R)[t^{\pm1}]\right)_0} \dar["{\rotatebox[origin=c]{90}{\(\sim\)}}"',"{\substack{l\otimes \sigma(x)y\\ \rotatebox[origin=c]{90}{\(\mapsto\)}\\ l\otimes x\otimes y}}"] \\[10pt]
(V(y_l)\otimes x)\otimes(a_l\otimes 1) \arrow[d, maps to] \& \big(I_R\otimes_\phi W(R)\big)\tensor{W(R)}\big(T'\otimes_\phi W(R)\big) \dar["{\substack{(V(y)\otimes x_1)\otimes(m\otimes x_2)\\ \rotatebox[origin=c]{-90}{\(\mapsto\)}\\ m\otimes y x_1 x_2}}"',"\phi^\mathrm{div}\otimes \id"] \& {\left(L\tensor{W(R)}\Rees\bigl(W(R)\bigr)\otimes_{\tilde\sigma} W(R)[t^{\pm1}]\right)_0}=\sigma^*L \dar["\sigma^*\tilde c"',"{\substack{l\otimes x\otimes y \\ \rotatebox[origin=c]{-90}{\(\mapsto\)}\\ \tilde c(l\otimes x)\otimes y}}"]  \& l\otimes 1\otimes x \arrow[d, maps to]             \\[10pt]
a_l\otimes y_l x \arrow[rd, maps to]                      \& T'\otimes_\phi W(R) \arrow[rd,"m\otimes x\mapsto m\otimes xt\inv"' sloped] \& {\left(T'\tensor{W(R)}\Rees\bigl(W(R)\bigr)\otimes_{\tilde\sigma} W(R)[t^{\pm1}]\right)_1} = \sigma^* T' \dar["{\rotatebox[origin=c]{90}{\(\sim\)}}"',"{\substack{m\otimes x\otimes y\\ \rotatebox[origin=c]{-90}{\(\mapsto\)}\\ m\otimes \sigma(x)y}}"] \& a_l\otimes V(y_l)t\inv\otimes x \arrow[d, maps to] \\[10pt]
\& a_l\otimes y_l x t\inv \& {\left(T'\tensor{W(R),\phi}W(R)[t^{\pm1}]\right)_1} \& a_l \otimes y_l x t\inv
\end{tikzcd}
}\]
\end{proof}

% \section{\texorpdfstring{\(G\)}{G}-Displays}

% A closed immersion \(i:G\to \GL_n\) and a cocharacter \(\mu:\GG_m\to G\), we get an induces map 
% \begin{align*}
%     \Par_{G,\mu} &\to \Par_{\GL_n,i\circ\mu} \\
%     {}^g P_\mu &\mapsto {}^{i(g)} P_{i\circ\mu},
% \end{align*}
% where \(P_\mu(R)=\menge{g\in G(R)}{\lim_{t\to0}\mu(t)g\mu(t)\inv\text{ exists}}\).

%\printbibliography
%

\bibliographystyle{amsalpha}
\bibliography{references}

\end{document}